 \documentclass[11pt,letterpaper]{amsart}

 \usepackage[english]{babel}
  \usepackage{amsmath,amsfonts,amssymb,amsthm,relsize,setspace,nicefrac,yhmath,amscd,eucal}
  \usepackage{amsrefs, mathrsfs}
  \usepackage{mathtools,tkz-euclide,tikz-3dplot,tkz-tab}
  \usepackage{xcolor}
  \usepackage{tikz}
  \usepackage{enumitem} 
  \usetikzlibrary{calc}
  \usepackage{anyfontsize}
 \usepackage[colorlinks, linkcolor=red, citecolor=blue, urlcolor=blue, hypertexnames=true]{hyperref}

  \usepackage{hyperref} 

 \usepackage[active]{srcltx} 
 \usepackage{color,soul} 
\newtheorem{theorem}{Theorem}

\newtheorem{lemma}[theorem]{Lemma}
\newtheorem{proposition}[theorem]{Proposition}

\usepackage[toc,page]{appendix}
\newtheorem{lma}{Lemma A.\ignorespaces}
\newtheorem{tha}[lma]{Theorem A.\ignorespaces}
\newtheorem{defa}[lma]{Definition A.\ignorespaces}
\newtheorem{remark}{Remark}

\theoremstyle{definition}
\newtheorem{definition}[theorem]{Definition}

\newcommand{\be}{\begin{equation}}
\newcommand{\bel}[1]{\begin{equation}\label{#1}}
\newcommand{\ee}{\end{equation}}

\newcommand{\barr}{\begin{eqnarray}}
\newcommand{\earr}{\end{eqnarray}}
\newcommand{\bars}{\begin{eqnarray*}}
\newcommand{\ears}{\end{eqnarray*}}

\newtheorem{subn}{\name}

\newcommand{\bsn}[1]{\def\name{#1}\begin{subn}}
\newcommand{\esn}{\end{subn}}

\newtheorem{sub}{\name}[section]

\newcommand{\bs}{\begin{sub}}
\newcommand{\es}{\end{sub}}

\newcommand{\bth}[1]{\def\name{Theorem}
\begin{sub}\label{t:#1}}
\newcommand{\blemma}[1]{\def\name{Lemma}
\begin{sub}\label{l:#1}}
\newcommand{\bcor}[1]{\def\name{Corollary}
\begin{sub}\label{c:#1}}
\newcommand{\bdef}[1]{\def\name{Definition}
\begin{sub}\label{d:#1}}
\newcommand{\bprop}[1]{\def\name{Proposition}
\begin{sub}\label{p:#1}}

\newcommand{\BA}{\begin{array}}
\newcommand{\EA}{\end{array}}
\newcommand{\BAN}{\renewcommand{\arraystretch}{1.2}
\setlength{\arraycolsep}{2pt}\begin{array}}
\newcommand{\BAV}[2]{\renewcommand{\arraystretch}{#1}
\setlength{\arraycolsep}{#2}\begin{array}}
\newcommand{\BSA}{\begin{subarray}}
\newcommand{\ESA}{\end{subarray}}

\newcommand{\BAL}{\begin{aligned}}
\newcommand{\EAL}{\end{aligned}}
\newcommand{\BALG}{\begin{alignat}}
\newcommand{\EALG}{\end{alignat}}
\newcommand{\BALGN}{\begin{alignat*}}
\newcommand{\EALGN}{\end{alignat*}}
\def\angb<#1>{\langle #1 \rangle}

\newcommand{\supp}{\opname{supp}}

\newcommand {\rd}{\color{red}}

\def\R{\mathbb{R}}

\let\.=\cdot
\let\0=\emptyset

\def\O{\Omega}

\def\supp{\text{\rm supp}}

\numberwithin{equation}{section}

\theoremstyle{definition}

\let\.=\cdot
\let\0=\emptyset

\def\O{\Omega}

\def\supp{\text{\rm supp}}

\newenvironment{formula}[1]{\begin{equation}\label{eq:#1}}
                       {\end{equation}\noindent}

\def\Fi#1{\begin{formula}{#1}}
\def\Ff{\end{formula}\noindent}

\setstcolor{red}

\begin{document}


\title[]{Spatio-temporal dynamics of an age-structured  reaction-diffusion system of epidemic type subjected by Neumann boundary condition}

\author{Cong-Bang TRANG}
\address{Faculty of Fundamental Science, Industrial University of Ho Chi Minh City, 12 Nguyen Van Bao, Ward 4, Dist. Go Vap, Ho Chi Minh City, Viet Nam}
\email{\rd trangcongbang@iuh.edu.vn}

\author{Hoang-Hung Vo$^{*}$}
\address{Faculty of Mathematics and Applications, Saigon University, 273 An Duong Vuong st., Ward 3, Dist. 5, Ho Chi Minh City, Viet Nam}
\email{\rd vhhung@sgu.edu.vn}

\date{\today}

\begin{abstract}
 	This paper is concerned with the spatio-temporal dynamics of an age-structured reaction-diffusion system of KPP-epidemic type (SIS), subject to Neumann boundary conditions and incorporating $L^1$ blow-up type death rate. We first establish the existence of time dependent solutions using age-structured semigroup theory. Afterward, the basic reproduction number $\mathcal{R}_0$ is derived by linearizing the system around the disease-free equilibrium state. In the case $\mathcal{R}_0<1$,  the existence, uniqueness and stability of disease-free equilibrium are shown by using $\omega$-limit set approach of Langlais \cite{langlais_large_1988}, combined with the technique developed in recent works of Zhao et al. \cite{zhao_spatiotemporal_2023} and Ducrot et al. \cite{ducrot_age-structured_2024}. We highlight that the absence of a general comparison principle for the age-structured SIS-model and non-separable variable mortality rate prevent the direct application of the semi-flow technique developed in \cite{ducrot_age-structured_2024} to study the long time dynamics. To overcome this difficulty, sub- and super-solutions are constructed from two distinct monotone systems to estimate the solution of our model. In particular, it is  challenging to establish the positive co-existence steady states as $\mathcal{R}_0>1$ for the age-structure models by developing the using of the Leray–Schauder topological degree  technique in Conti et al. \cite{conti_asymptotic_2005}, Tavares et al. \cite{tavares_existence_2011} and global bifurcation technique of Walker \cite{walker_positive_2011}. Due to the lack of uniqueness of the endemic equilibrium, we partially characterize the spatio-temporal dynamics of our SIS disease transmission model to confirm the persistence of infected individuals in the population. Finally, we investigate the long time stability of the unique endemic equilibrium in a special case of our model, where we reduce the effect KPP net growth rate of the susceptible population.   We resolve this by re-evaluating the eigenvalue problem for the age-structured model.  We believe that our approach in this work can be applied to deeper study  the spatio-temporal dynamics of epidemic or ecological reaction-diffusion systems incorporating with age-structure arising in applied science or in the real world proposed in \cite{terracini_uniform_2016,noris_uniform_2010,conti_asymptotic_2005,tavares_existence_2011,deng_pulsating_2024, soave_anisotropic_2023,walker_positive_2011}


%
\end{abstract}




\keywords{epidemic system, long-time behavior, age-structure, spectral theory, Leray-Schauder topological degree, $\omega$-limit set}

\maketitle

\tableofcontents

\section{\bf Introduction}

Our current paper focuses on spatio-temporal dynamics of an age-dependent transmission disease model with spatial diffusion. The model takes the form of an age-structured heterogeneous reaction-diffusion system  as follows
\begin{align}\label{eq:main}
\begin{split}
\left\{\begin{array}{lllll}
u_t + u_{a}=d_u\Delta_N u+m(x)u-n(x)u^2
-\dfrac{p(x)uv}{u+v}+q(x)v-\mu(x,a)u, \\
v_t + v_{a}=d_v\Delta_N v+
\dfrac{p(x)uv}{u+v}-q(x)v-\mu(x,a)v,\\
u(t,x,0) = \displaystyle\int_0^{A} \beta(x,a) u(t,x,a)da,\\v(t,x,0) = \displaystyle\int_0^{A} \beta(x,a) v(t,x,a)da,\\
\dfrac{\partial u}{\partial \textbf{n}}=\dfrac{\partial v}{\partial \textbf{n}}=0,\\
u(0,x,a) = u_0(x,a),~
v(0,x,a) = v_0(x,a),
\end{array}\right.
\end{split}
\end{align}
where $t > 0$ denotes time, $x \in \O$ represents the spatial position, $\O\subset \R^N$ is a bounded domain with smooth boundary $\partial \O$,  $d_u>0$, $d_v>0$ are diffusion constants,  $a\in [0,A]$ denotes the physiological age, $0<A<\infty$ is the maximal age of the species, $m=m(x)$ is the intrinsic growth rate, $n=n(x)$ with $\dfrac{m(x)}{n(x)}$ is the carrying capacity of the environment, $p=p(x)$ is the transmission rate, $q=q(x)$ denote the recovery rate of infectious individuals and $\beta=\beta(x,a),~\mu=\mu(x,a)$ are the natural birth rate and morality rate of species, respectively. 

System \eqref{eq:main} takes the form of a susceptible-infected-susceptible (SIS) model, first proposed by Allen, Bolker, Lou, and Neval \cite{allen_asymptotic_2008}, who analyzed the impact of spatial heterogeneity of environment and movement of the species on the persistence and extinction of a disease via SIS epidemic reaction-diffusion model, which reads as follows	
\begin{align*}
\begin{split}
\left\{\begin{array}{lllll}
S_t =d_S\Delta_N S
-\dfrac{\beta(x)SI}{S+I}+\gamma(x)I,&t>0,~x\in\O, \\
I_t =d_I\Delta_N S+
\dfrac{\beta(x)SI}{S+I}-\gamma(x)I,&t>0,~x\in \O, \\
\dfrac{\partial S}{\partial \textbf{n}}=\dfrac{\partial I}{\partial \textbf{n}}=0,&t>0,~x,\in\partial \O,\\ 
S(0,x) = S_0(x),~
I(0,x) = I_0(x),&x\in \O,
\end{array}\right.
\end{split}
\end{align*}
where $S,~I$ denote for the densities of susceptible and infected populations, $d_S,~d_I>0$ are diffusion constants measuring the mobility of susceptible and infected groups, respectively. $\beta,~\gamma$ are positive function on $\overline{\O}$ presenting the disease transmission and recovery rate, respectively. The literature established the existence, uniqueness of disease-free equilibrium $(DFE)$ and endemic equilibrium $(EE)$. Then, the authors demonstrated the stability of $(DFE)$ and, in addition, the asymptotic behavior of the unique endemic equilibrium when $d_S\to 0$. They further proposed open questions concerning the potential stability of the endemic equilibrium $(EE)$, as its uniqueness had already been established. In this work, we shall focus on investigation of spatial and temporal dynamics of system \eqref{eq:main}, which focus on the  the existence of the disease-free equilibrium $(DFE)$ and the co-existence  endemic equilibrium $(EE)$ via the age-structured spectral theory. The long-time stability is further investigated in the special case reducing logistic term in the first equation \eqref{eq:main} by by using $\omega$-limit set approach of Langlais \cite{langlais_large_1988}, combined with the technique developed in recent works of Zhao et al. \cite{zhao_spatiotemporal_2023} and Ducrot et al. \cite{ducrot_age-structured_2024}.

In epidemiology, such a model describes the evolution of disease transmission and recovery within a spatially distributed population. Here, $u=u(t,x,a)$ and $v=v(t,x,a)$ present the density of a susceptible population and an infectious population, respectively, at location $x\in \overline{\O}$, time $t>0$ and the age $a\in [0,A]$. For convenience, we put $\mathcal{O}_a:=\O\times (0,a)$ for any $a\in [0,A]$ and
\begin{align*}
\mu_{\max}(a):=\max\limits_{x\in \O}\mu(x,a)\text{  and }\mu_{\min}(a):=\min\limits_{x\in \O}\mu(x,a),\\
\beta_{\max}(a):=\max\limits_{x\in \O}\beta(x,a)\text{  and }\beta_{\min}(a):=\min\limits_{x\in \O}\beta(x,a).
\end{align*} 
and with $f:\overline{\mathcal{O}_A}\rightarrow\R$, recall
\begin{align*}
\supp(f)=\{(x,a)\in \overline{\mathcal{O}_A}:~f(x,a)\neq 0 \}.
\end{align*}

Throughout this paper, we shall assume that
\begin{enumerate}[label=(A\arabic*)]  
\item \label{cond:cond1}$m$, $n$, $p$, $q$ are in $C^{2}(\overline{\O})$. 

\item \label{cond:cond2} Suppose $n>0$, $p\geq 0$, $p\not\equiv 0$ and $q>0$. Also, $\overline{m}:=\dfrac{1}{|\O|}\displaystyle\int_{\O}m dx\geq 0$ and $m$ is non-trivial.



\item \label{cond:cond3} $\mu=\mu(x,a)$ is a positive $C^{2,1}$ function on $\overline{\O}\times[0,A)$. 


\item \label{cond:cond4} $\mu$ satisfies
\begin{align*} 
\int_0^a \mu_{\max}(k)dk <\infty,~\forall a\in (0,A),~\int_0^A \mu_{\min}(k)dk = \infty.
\end{align*}
and
\begin{align*}
a_0\mapsto \sup_{a\in[0,a_0]}\int_{\O}\mu^2(x,a)da \text{ is continuous with respect to $a_0\in[0,A]$}
\end{align*}

\item \label{cond:cond5}  $\beta=\beta(x,a)$ is a non-negative $C^{2,1}$ function on $\overline{\mathcal{O}_A}$, and
there exists $A_0,~A_1,~A_2 \in (0,A)$, $A_0<A_1<A_2$ such that 
\begin{align*}
\beta_{\min}(a)>0,~\forall a\in (A_0,A_1) \text{ and }\supp(\beta)\subset\overline{\O}\times (A_0,A_2],
\end{align*}
and
\begin{align*}
\displaystyle\int_0^A\beta_{\max}(a)e^{-\displaystyle\int_0^a\mu_{\min}(k)dk}da\leq 1.
\end{align*}

  
\end{enumerate}

To ensure the non-triviality of $A_2$ in \ref{cond:cond5}, we may further assume that $\beta_{\min}>0$ on $(A_2-\epsilon,A_2)$ some sufficiently small $\epsilon>0$. 

Under condition \ref{cond:cond2}, if $m$ is a non-constant function, it is not necessarily positive over the domain $\overline{\O}$. This implies that some regions favor the population, while others do not. The last three conditions follow the spirit of Guo and Chan \cite{guo_on_1994}, Delgado, Molina-Becerra and Suárez \cite{delgado_nonlinear_2008} to develop the spectral theory in the age-structured model. A typical example for the condition \ref{cond:cond4} can be thought of is an increasing function with respect to the age variable $a$, which is reasonable since the morality rate should increase as an individual ages. In addition, the $L^1$ blow-up property of the morality rate $\mu$  ensures no individual of the species can live beyond the maximum age $A$. On the other hand, in the condition \ref{cond:cond5}, the birth rate $\beta$ is restricted to younger individuals of the species, meaning it vanishes beyond a certain age. This is biologically consistent, as the older individuals are unable to reproduce the new generation. Furthermore, the last condition in \ref{cond:cond5} is inspired by the population equilibrium condition for demographic function in Ducrot \cite{ducrot_travelling_2007}, see also Gurtin and Maccamy \cite[Theorem 6]{gurtin_non-linear_1974}, Anita \cite[Theorem 2.3.3]{anita_2000}. This prevents unbounded population growth driven by reproduction. For historical reasons, we define
\begin{align*}
\pi(a)=e^{-\displaystyle\int_0^a\mu_{\min}(k)dk},~a\in [0,A) \text{ and }\pi(A)=0.
\end{align*}
Clearly, $\pi(a)\in C([0,A],\R)\cap C^{1}((0,A),\R)$ and $0\leq \pi(a)\leq 1$ for any $a\in [0,A]$. Define $L=\pi^{-1}$ on $[0,A)$. One can check $L\rightarrow \infty$ as $a\rightarrow A$.


Over the past a few decades, the age-structured model has been topic of intensive research. In particular, in their celebrated work, Gurtin and Maccamy \cite{gurtin_non-linear_1974} have first proposed a nonlinear model for population dynamics, in which the birth and death rates are age specific and depend upon the total population size and provide a  necessary and sufficient condition for an equilibrium age-distribution. After that, Langlais \cite{langlais_large_1988} was the first to analyze the long-time dynamics of a nonlinear age-dependent population model with spatial diffusion by using $\omega$-limits set approach in $L^2$. The author demonstrated that, as $t\rightarrow \infty$, the solution either tends to $0$ or stabilizes to a nontrivial stationary solution with a separable density-dependent death rate in bounded domain. 

Alternatively, in a pioneering work, Walker \cite{walker_coexistence_2010} first investigated an age-structured predator–prey system incorporating spatial diffusion and Holling –Tanner-type nonlinearities as follows
\begin{align*}
\left\{\begin{array}{llll}
u_t+u_a -d_1\Delta_D u=-\mu_1(x,u,v)u,&t>0,~x\in \O,~a\in (0,A),\\
v_t+v_a -d_2\Delta_D v=-\mu_2(x,u,v)u,&t>0,~x\in \O,~a\in (0,A),\\
u(t,x,0)=\displaystyle\int_0^A\beta_1(a,u,v)u(t,x,a)da,\\v(t,x,0)=\displaystyle\int_0^A\beta_2(a,u,v)v(t,x,a)da
\end{array}\right.
\end{align*} 
where $\mu_j$, $\beta_j$, $j=1,2$ are   the death and birth rates, respectively, depending nonlinearly on the predator $v$ and on the prey $u$, $\Omega \subset \R^N$ is a bounded and smooth domain and $A\in (0,\infty]$ is the maximal age. The author focused on steady-state solutions, that is, time-independent non-negative solutions of the preceding system subjected by Dirichlet boundary condition. Regarding the intensity of the fertility of the predator as a bifurcation parameter, a branch of positive coexistence steady states is established via the bifurcation from the marginal steady state without predators, owing to the compactness of the parabolic Laplacian semigroup. An analogous result is observed when the fertility of the prey varies. Afterward, in a deep work, Walker \cite{walker_positive_2011} focused on describing the structure of the set of positive solutions with respect to two parameters measuring the intensities of the fertility of the species, especially establishing co-existence steady-states by using global bifurcation techniques of an age-structured predator-prey model with Dirichlet boundary condition read as follows
\begin{align*}
\left\{\begin{array}{llll}
u_a -\Delta_D u=-(\alpha_1u + \alpha_2 v)u,&t>0,~x\in \O,~a\in (0,A),\\
v_a -\Delta_D v=-(\beta_1u - \beta_2 v)v,&t>0,~x\in \O,~a\in (0,A),\\
u(t,x,0)=\displaystyle\int_0^A\eta\beta_1(a)u(t,x,a)da,\\v(t,x,0)=\displaystyle\int_0^A\chi\beta_2(a)v(t,x,a)da.
\end{array}\right.
\end{align*} 
Such equations arise as steady-state equations in an age-structured predator-prey model with diffusion.  solutions which are nonnegative and nontrivial in both components. For additional studies on bifurcation analysis, we refer the reader to  \cite{walker_global_2010,walker_nonlocal_2011,walker_principle_2022,walker_well-posedness_2023,walker_age-dependent_2010} for age-structured model including the well-posedness, and the long-time exponential stability of the equilibrium solution  under a properly chosen initial condition; Nirenberg \cite[Chapter 3]{nirenberg_nonlinear_2001} and the recent work by Li and Terracini \cite{li_bifurcation_2024} for reaction-diffusion models. Research for bifurcation examines the existence, uniqueness and topological structure of steady state under changes in the qualitative features or parameters. This serves as a powerful tool for investigating the global dynamics of the age-structured reaction-diffusion models. 



On the other hand, in a rigorous analysis, Kang and Ruan \cite{kang_principal_2022,kang_principal_2023} and Ducrot, Kang and Ruan \cite{ducrot_age_2024} studied the principal spectral theory for age-structured model with nonlocal diffusion. The theory is a powerful tool to deal with many other important problems in the field of age-structured reaction-diffusion equations. In particular, this serves as the first important step toward studying the global dynamics of the age-structured model with non-local diffusion. Indeed,  Ducrot, Kang and Ruan \cite{ducrot_age-structured_2024} investigated an age-structured model with nonlocal diffusion of Dirichlet type and a monotone nonlinearity in the birth rate, given by
\begin{align*}
\left\{\begin{array}{llll}
u_t+u_a = D\left[\displaystyle\int_{\O} J(x-y)u(t,y,a)dy-u(t,x,a)\right] - \mu(x,a)u,\\
u(t,x,0)=f\left(\displaystyle\int_0^A\beta(x,a)u(t,x,a)da\right),\\
u(0,x,a)=u_0(x,a).
\end{array}\right.
\end{align*}
where $t > 0$ denotes time, $x \in \O$ represents the spatial position, $\O\subset \R^N$ is a bounded domain with smooth boundary $\partial \O$, $a\in (0,A)$ denotes the physiological age, $0<A<\infty$ is the maximum age of an individual, $D>0$ is the diffusion rate, $f$ is a monotone type nonlinearity describing the birth rate of the population, $\beta$, $\mu$ denote the age-specific birth rate and death rate, respectively, $u = u(t,x,a)$ presents the population density, $J \in C^1(\R^N)$ is a non-local diffusion kernel is assumed to satisfy
\begin{center}
  $J\geq 0$, $J(0)>0$, $\displaystyle\int_{\R^N}J(x)dx=1$ and $\supp(J)\subset B(0,r)\subset \R^N$ for some $r>0$.
\end{center}
This age-structured model describes population movement in a non-local sense, incorporating the natural death rate and the nonlinear birth rate. By applying the spectral theory, the authors established the existence, uniqueness and regularity of a positive equilibrium via the classical super-sub-solutions method for the age-structured model. They further emphasize that, due to the lack of compactness and regularity in the non-local diffusion, the well-known fixed point theorems, such as Schauder theorem, and the bifurcation methods proposed by Walker \cite{walker_global_2010,walker_nonlocal_2011,walker_positive_2011,walker_coexistence_2010}, cannot be directly applied in their work. Afterward, the existence, uniqueness and positiveness of the original time-dependent non-local age-structured model are proven using the semi-group theory. The stability of the steady state is also provided via the semi-flow technique. In addition, the asymptotic behavior of the steady state in terms of diffusion rate and diffusion range is investigated in this work.




To deeper understand the influence of the age structure on the population dynamics, which simplifies epidemic model,   Ducrot \cite{ducrot_travelling_2007} also  investigated traveling wave solution for a simplified SI model with age-structure.
The author further assumed the birth and morality rate to equilibrate the population to prevent the population from going extinct and from exploding as time increases. Mathematically, the condition can be described by the demographic function $\displaystyle\int_0^A\beta(a)e^{-\displaystyle\int_0^a \mu(s)ds}da = 1$. Then, Ducrot et al. \cite{ducrot_travelling_2007,ducrot_travelling_2009, ducrot_integrated_2021,ducrot_age-structured_2024} tactfully used the Schauder fixed point theorem and the parabolic comparison principle to prove the existence of traveling wave for the age-structured epidemic model associated with an implicit admissible speed $c>c^{*}$.
%
Afterward,    Ducrot, Magal, and Ruan \cite{ducrot_travelling_2010} provided sufficient conditions that ensure the existence of traveling wave solutions for the age-structured  epidemic population of multigroups. Each group is further divided into two subclasses: susceptible and infected. A susceptible individual can become infected through crisscross transmission by direct contact with infective individuals from any group, with the transmission process depending on the disease age of the infections. As an application, the authors numerically investigate a two-group model, illustrating its relevance to the crisscross transmission of feline immunodeficiency virus (FIV) and certain sexually transmitted diseases. Recently, in a deep work,  Deng and Ducrot \cite{deng_pulsating_2024} studied the existence of pulsating traveling waves for the classical endemic $(SI)$ system, which is further well-known as the Kermack–McKendrick epidemic model with Fickian diffusion. The authors first provided a necessary conditions to ensure the existence of an admissible speed $c^*(e)$ for each unit direction $e$. Subsequently, the existence of the pulsating wave was  established for wave speed $c>c^*(e)$ in each propagating unit direction $e$. Specifically, it was shown that these wave fronts are positive, globally bounded and enjoy certain limiting properties. In addition, the wave front associated with the infected population exhibits exponential decay.


Collaterally, Bekkal~Brikci, Clairambault, Ribba and Perthame \cite{bekkal_brikci_age-and-cyclin-structured_2008} used different way to study the age-structured model. In particular, the authors neglected the spatial position and worked with the phenotypic trait of the species of the age $a$, at the time $t$. This is an alternative way to understand how the age-structure affects the population, see also  \cite{almeida_asymptotic_2022, nordmann_dynamics_2021, nordmann_dynamics_2018} for more investigations about this subject. Besides,  Ferrrière and Tran \cite{ferriere_stochastic_2009} studied the age-structure by examining the mutation of general traits in species. They showed that a random process converges to the solution of a Gurtin-McCamy type PDE. The result on the limiting PDE and large deviation techniques in Freidlin-Wentzell provides estimates of the extinction time and a better understanding of the long-time behavior time of the stochastic process, see also \cite{roget_long-time_2019}. Further investigations on the age-structured model are referred to \cite{chekroun_global_2020,ducrot_differential_2022, ducrot_travelling_nodate,magal_monotone_2019, ducrot_integrated_2021, park_optimal_1998, tian_traveling_2022a, tian_traveling_2022b}. Beside that,  Zhao and Ruan \cite{zhao_spatiotemporal_2023} proposed an epidemic model (SIS) with fractional diffusion, subject to Neumann boundary conditions and incorporating the logistic source as follows
\begin{align*}
\left\{\begin{array}{lllll}
u_t+ (-d_u\Delta_N)^{s_1}u = a(x)-b(x)u^2-\dfrac{p(x)uv}{u+v}+q(x)v,&~(t,x)\in \R^+\times \O\\
v_t+ (-d_v\Delta_N)^{s_2}v =\dfrac{p(x)uv}{u+v}-q(x)v,&~(t,x)\in \R^+\times \O\\
u(0,x)=u_0(x),~v(0,x)=v_0(x),
\end{array}\right.
\end{align*}
where $u,~v$ represents the densities of susceptible and infected populations, $0<s_i<1$ for $i=1,~2$ are the fractional powers of the Laplace operator $-\Delta_N$, $(-d\Delta_N)^{s}$ is the Neumann fractional Laplace operator for a diffusion constant $d=d_u,d_v$ and a fractional power $s=s_1,s_2$, respectively. In this work, the authors  explained that the human mobility patterns exhibit scale-free, non-local dynamics with heavy-tailed distributions. Such movements are modeled by Lévy flights, which are often mathematically represented via the  fractional Laplacian $(-d\Delta_N)^{s}$, see Metzler and Klafter \cite{metzler_random_2000}. This approach fundamentally differs from the classical one, in which human mobility is described by Brownian motion associated with the standard Laplacian operator. As their main goals, Zhao and Ruan focused on the spatio-temporal dynamics via the existence and stability of the disease-free equilibrium $(DFE)$ and the endemic equilibrium $(EE)$. The uniqueness and stability of $(DFE)$ and the asymptotic behavior of $(EE)$ as $d_u\to 0,$ $d_u\to\infty$ and $d_v\to \infty$ are also established via in-depth studies of the fractional Laplacian spectral theory. In addition, in a special case, the stability of $(EE)$ is established. For further investigations of SIS endemic models, we refer the reader to \cite{li_analysis_2017,peng_reactiondiffusion_2012,ge_sis_2015,wu_asymptotic_2016,li_varying_2017}.

For relevant works in the spatial dynamics of of a biological system,  Conti, Terracini and Verzini \cite{conti_asymptotic_2005} and Soave and Terracini \cite{soave_anisotropic_2023} studied the system of $k$-partial differential equation as follows
\begin{align*}
\left\{\begin{array}{lllll}
-\Delta u_i&= -\kappa u_i\displaystyle\sum_{i\neq j}a_{ij}u_j+f(x,u_i),&x\in \O,\\
u_i(x)&=\phi_i(x),&x\in \partial \O.
\end{array}\right.
\end{align*}
where $i=1,...,k$, $u_i$ represents the population density of the $i$th species, $f_i=f_i(x,u_i)$ prescribes internal dynamic of $u_i$; the positive constants $\kappa a_{ij}$ determine the interaction between the population $u_i$ and $u_j$, which is possibly asymmetric. Besides, the boundary data $\phi_i$ are positive $W^{1,\infty}(\partial \O)$-functions with disjoint supports, namely $\phi_i \phi_j=0$ for $i\neq j$, almost everywhere on $\partial \O$. This system governs the steady state of $k$-competing species coexisting in the same limited area $\O$. In this work,  the authors utilized the topological degree theory to demonstrate the existence of a solution for the system. Subsequently, the asymptotic behavior of the positive solutions was investigated as the competition rate $\kappa$ tends to infinity.  Beside that, Noris, Terracini, Tavares and Verzini \cite{noris_uniform_2010} investigated the Gross–Pitaevskii system or nonlinear Schrödinger equations with the competition parameter $\beta>0$. The authors first showed that the uniform boundedness of the solution leads to the $C^{0,\alpha}$-Hölder uniform boundedness for $\alpha \in (0,1)$ and, then, as $\beta\rightarrow \infty$, the limiting profile is shown to be Lipschitz-continuous. Subsequently, Tavares, Terracini, Verzini, and Weth \cite{tavares_existence_2011} investigated the existence and non-existence of non-trivial solutions to the cubic Schrödinger system associated with the coefficient matrix $\beta$, employing topological degree theory. Afterward, Terracini, Verzini, and Zilio \cite{terracini_uniform_2016} studied the fractional Gross–Pitaevskii system with the order $s=\dfrac{1}{2}$. Similar to the work in \cite{noris_uniform_2010}, they established that the uniform boundedness of the solution leads to $C^{0,\alpha}$-Hölder uniform boundedness for $\alpha \in \left(0, \dfrac{1}{2}\right)$, provided that $\beta$ is sufficiently large.




Coming back to our system \eqref{eq:main}, as previously mentioned, the aim of the present paper is to investigate the spatio-temporal dynamics associated with the spatially heterogeneous reaction-diffusion system \eqref{eq:main}. To begin the study, the \textit{basic reproduction number} $\mathcal{R}_0$ is developed thanks to approaches in Chekroun and Kuniya \cite{chekroun_global_2020}, Yang, Gon and Sun \cite{yang_asymptotical_2023}. The main difficulty lies in the non-self-adjoint nature of the age-structured operator, which complicates the calculation of $\mathcal{R}_0$ and its comparison with the principal eigenvalue of these operators, as we shall discuss later in this work. This property distinguishes our study from the existing literature, such as \cite{allen_asymptotic_2008,zhao_spatiotemporal_2023}, where the operator is self-adjoint. After that, we develop the technique of using the topological degree theory, based on  \cite{nirenberg_nonlinear_2001,fonseca_degree_1995,ambrosio_functions_2000}, to establish the co-existence of endemic equilibrium $(EE)$ for the system \eqref{eq:main}. The key point is to first establish existence for a simplified system, and then use the Leray–Schauder  degree to determine the existence of a solution to the original system via the homotopy invariance property. To our knowledge, this is the first application of the topological degree theory to verify the existence of solutions for an age-structured model with a quotient-type nonlinear term. Furthermore, due to the complexity of the quotient-type nonlinearity in the system, we emphasize that it is not feasible to apply Schauder or Banach fixed theorem to construct a $(EE)$-solution. This research marks the first significant application of topological degree theory to study the long-time behavior for the age-structured disease transmission model. On other hand, the existence, uniqueness of disease-free equilibrium $(DFE)$ is provided thanks to the classical super-sub-solution method and the age-structured comparison principle between these super-sub-solutions. 

For the convenience of the reader, we define the endemic equilibrium system as follows
\begin{align}\label{eq:main1}
\begin{split}
\left\{\begin{array}{llll}
 u_{a}=d_u\Delta_Nu +m(x)u-n(x)u^2
-\dfrac{p(x)uv}{u+v}+q(x)v-\mu(x,a)u,&(x,a)\in\mathcal{O}_A, \\
 v_{a}=d_v\Delta_Nv+
\dfrac{p(x)uv}{u+v}-q(x)v-\mu(x,a)v,&(x,a)\in\mathcal{O}_A,\\
u(x,0) = \displaystyle\int_0^{A} \beta(x,a) u(x,a)da,~v(x,0) = \displaystyle\int_0^{A} \beta(x,a) v(x,a)da,&x\in \O.
\end{array}\right.
\end{split}
\end{align}

 Now, let us state our main results. The first result is about the existence of the equilibrium state  of the system  \eqref{eq:main} as follows
 
\begin{theorem} \label{main:theo1}
Assume \ref{cond:cond1} to \ref{cond:cond5} hold. Let $\lambda_{d_u,-m+\mu}$, $\lambda_{d_u,-m+p+\mu}$ be the principal eigenvalue of $\mathcal{A}_{d_u,-m+\mu}$, $\mathcal{A}_{d_u,-m+p+\mu}$, respectively. The following statements hold
\begin{enumerate}[label=(\roman*)]
\item \label{DFE} If $\lambda_{d_u,-m+\mu}>0$ and $\mathcal{R}_0\leq 1$, then \eqref{eq:main1} has no positive solution and it only has one semi-positive solution $(u_1,0)$, where $u_1$ is the unique positive solution of 
\begin{align}\label{eq:nodisease}
\begin{split}
\left\{\begin{array}{llll}
 u_{a}=d_u\Delta_Nu +m(x)u-n(x)u^2
-\mu(x,a)u,&(x,a)\in\mathcal{O}_A, \\
u(x,0) = \displaystyle\int_0^{A} \beta(x,a) u(x,a)da,&x\in \O.
\end{array}\right.
\end{split}
\end{align}
\item \label{EE} If $\lambda_{d_u,-m+p+\mu}>0$ and $\mathcal{R}_0>1$, then \eqref{eq:main1} admits a positive endemic steady state $(u,v)$.
\end{enumerate}
\end{theorem} 
The formula for the age-structured operator  $\mathcal{A}_{d_u,\mu^*}$ will be provided in Section \ref{sec 2.1}, along with a proper choice of $\mu^*$. The significant challenge in our work is to use the age-structured spectral theory and the topological degree theory to prove the positive co-existence of endemic steady state $(u,v)$. The principal eigenvalues $\lambda_{d_u,-m+\mu}$ and $\lambda_{d_u,-m+p+\mu}$ are both employed to construct the sub-solution of \eqref{eq:main1}, ensuring the nontriviality of its solution. In addition, a super-solution is determined to guarantee that the solution does not blow up before reaching the maximal age  $A$. We emphasize that, due to the quotient-type nonlinear term, we are unable to employ other well-known fixed point theorems, such as the Banach contraction principle or Schauder fixed point theorem, to establish the positive co-existence steady state, which is a counterpart of using global bifurcation approach of Walker \cite{walker_positive_2011} to prove the positive co-existence of steady states.

Now, we further characterize  the long time behavior for  system \eqref{eq:main}.

\begin{theorem}\label{Theo:2}
Assume \ref{cond:cond1} to \ref{cond:cond5} hold. Suppose further that $(u_0, v_0) \in C(\overline{\mathcal{O}_A})\times C(\overline{\mathcal{O}_A})$ with $u_0,~v_0\geq 0$, $u_0,~v_0\not\equiv 0$ and $\supp(u_0),~\supp(v_0)\subset \O\times [0,A)$. Let $(u,v)=(u(t;u_0,v_0),v(t;u_0,v_0))$ be a solution of \eqref{eq:main} starting by the initial condition $(u_0,v_0)$. Then, for any $A_2<A_3<A$, the following statements hold

\begin{enumerate}[label=(\roman*)]
\item \label{sta:DFE} If $\lambda_{d_u,-m+\mu}>0$ and  $\mathcal{R}_0<1$, one has \begin{align*}
||u(t;u_0,v_0)-u_1||_{L^{\infty}(\overline{\mathcal{O}_{A_3}})}\rightarrow 0,~||v(t;u_0,v_0)||_{L^{\infty}(\overline{\mathcal{O}_{A_3}})}\rightarrow 0.
\end{align*}
and
\begin{align*}
||u(t;u_0,v_0)-u_1||_{L^{2}(\mathcal{O}_A)}\rightarrow 0,~||v(t;u_0,v_0)||_{L^{\infty}(\mathcal{O}_A)}\rightarrow 0.
\end{align*}
\item \label{sta:EE} If $\lambda_{d_u,-m+\mu+p}>0$ and  $\mathcal{R}_0>1$, there exist $\varepsilon_1,~\varepsilon_2,~K_1,~K_2>0$ such that
\begin{align*}
\begin{array}{llll}
&0<\varepsilon_1< \liminf\limits_{t\rightarrow \infty}u(t;u_0,v_0)\leq \limsup\limits_{t\rightarrow \infty}u(t;u_0,v_0)\leq  K_1,\\
&0<\varepsilon_2< \liminf\limits_{t\rightarrow \infty}v(t;u_0,v_0)\leq \limsup\limits_{t\rightarrow \infty}v(t;u_0,v_0)\leq K_2,
\end{array} \text{ uniformly on $\overline{\mathcal{O}_{A_3}}$.}
\end{align*}
\end{enumerate}

\end{theorem}

One of the main difficulties in this result lies in the absence of a general age-structured  comparison principle, as the nonlinear term is not monotone. Consequently, system \eqref{eq:main} cannot generate a semi-flow, preventing the direct application of the technique in the work of Ducrot, Kang, and Ruan \cite[Theorem 4.11]{ducrot_age-structured_2024}. In the spirit of Zhao and Ruan \cite[Theorem 3.3]{zhao_spatiotemporal_2023}, a super solution - being a solution of an age-structured monotone system - is tactfully chosen  to form a semi-flow. A similar construction is used for a sub-solution.  Afterward, it follows from the squeeze principle to establish the stability of disease-free equilibrium. On the other hand, due to the lack of the uniqueness of the endemic equilibrium, we are unable to estimate the long-time dynamics of the aforementioned super-solution, which prevents a complete analysis of system \eqref{eq:main} on the long-time stability of the positive co-existence steady states.

Finally, we reduce the logistic term $mu-nu^2$ to obtain the uniqueness of the endemic equilibrium in a certain sense. The main objective is to establish the long time stability of the positive co-existence endemic equilibrium. In this case, we shall consider a new set of conditions \ref{condA'1}, \ref{condA'2} and \ref{condA'3}, which are closely related to the original ones and will be precisely described  in the Section \ref{section:5}. Now, the system \eqref{eq:main} is reduced as follows 
\begin{align}\label{eq:main-no_m_n}
\begin{split}
\left\{\begin{array}{ll}
u_t + u_{a}=d\Delta_N u
-\dfrac{puv}{u+v}+qv-\mu(a)u,&t>0,~(x,a)\in\mathcal{O}_A, \\
v_t + v_{a}=d\Delta_N v+
\dfrac{puv}{u+v}-qv-\mu(a)v,&t>0,~(x,a)\in\mathcal{O}_A,\\
u(t,x,0) = \displaystyle\int_0^{A} \beta(a) u(t,x,a)da,\\
v(t,x,0) = \displaystyle\int_0^{A} \beta(a) v(t,x,a)da,\\
u(0,x,a) = u_0(x,a),~
v(0,x,a) = v_0(x,a).
\end{array}\right.
\end{split}
\end{align}
and the steady state system 
\begin{align}\label{eq:main1_nospace}
\begin{split}
\left\{\begin{array}{ll}
u_{a}=
-\dfrac{puv}{u+v}+qv-\mu(a)u,&a\in(0,A),\\
v_{a}=
\dfrac{puv}{u+v}-qv-\mu(a)v,&a\in(0,A),\\
u(0) = \displaystyle\int_0^{A} \beta(a) u(a)da,~v(0) = \displaystyle\int_0^{A} \beta(a) v(a)da.
\end{array}\right.
\end{split}
\end{align}
We shall prove that the endemic equilibrium  is unique and independent of the spatial diffusion in Theorem \ref{theo:16}, which is a noteworthy result. Next, we state the stability results as follows
\begin{theorem}\label{theo:stab_no_m_n}
Assume \ref{condA'1} to \ref{condA'3} hold and $\mathcal{R}_0>1$. Suppose further that $u_0\in C(\overline{\mathcal{O}_A})$ with $u_0\geq 0$ and $ \supp(u_0)\subset \O\times [0,A)$ and $v_0:=KY_{\alpha}(a)$ for constants $K,~\alpha>0$. Then, the system \eqref{eq:main1_nospace} admits a unique positive solution $(u^*,v^*)$, depending on $(u_0,v_0)$, satisfying
\begin{align*}
w^*(0):=\dfrac{\displaystyle\int_0^A\beta(b)\displaystyle\int_0^be^{-\displaystyle \int_b^e\mu(s)ds}\int_{\O}w_0(x,e)dxdeda}{\displaystyle\int_0^A\beta(b)b\pi_0(b)db}>0.
\end{align*}
where $w_0:=u_0+v_0$ and $w^*:=u^*+v^*$. Moreover, if  $(u(t;u_0,v_0),v(t;u_0,v_0))$ is a solution of the system \eqref{eq:main-no_m_n} starting by the initial condition $(u_0,v_0)$, then there exists $K,~\alpha>0$ large enough so that
\begin{align*}
||u(t;u_0,v_0)-u^*||_{L^{2}(\mathcal{O}_A)}\rightarrow 0,~||v(t;u_0,v_0)-v^*||_{L^{2}(\mathcal{O}_A)}\rightarrow 0.
\end{align*}

\end{theorem}
The function $Y_{\alpha}$ will be determined in Lemma \ref{Lemma:super} of this work. Upon examining the results, it is observed that the existence, uniqueness and stability steady state deeply depend on the initial condition, relating with the classical approaches of Langlais \cite[Theorem 4.9]{langlais_large_1988}. The key idea here is to sum the two equations of system \eqref{eq:main-no_m_n}, thereby deriving an age-structured KPP-type equation for $v$. This  equation admits a unique steady state, depending on the choice of the constant $w^*(0)$, in a similar manner to Appendix section of this work. Afterward, by choosing a suitable initial condition for $v$, one obtains the local $L^2$-convergence in the age variable. Then, Lebesgue dominated convergence theorem yields the global $L^2$-convergence.


\textbf{We structure the paper as follows}:  In Section \ref{section:2}, we introduce the preliminary results that are used throughout the paper.  Additionally, we prove the global-in-time existence of a unique solution to system \eqref{eq:main} based on the semi-group theory for a certain class of initial conditions. Three comparison principles, which will be used in future work, are also stated in this section. In Section \ref{section:3}, we present a series of properties concerning the basic reproduction number $\mathcal{R}_0$ and establish sufficient conditions to ensure the existence of the endemic equilibrium. Section \ref{section:4} first provides the stability of the disease-free equilibrium $(DFE)$ for a super-solution to system \eqref{eq:main} using the age-structured semi-flow theory. Afterward, we develop a specific sub-solution and demonstrate its stability in a similar manner with the help of implicit function theorem. Then, the long time behavior the disease-free equilibrium $(DFE)$ for system \eqref{eq:main} is obtained. Finally, in Section \ref{section:5}, we completely investigate the long time stability of positive co-existence endemic equilibrium $(EE)$ for system \eqref{eq:main-no_m_n}.

%

\textbf{Acknowlegment}: The second author expresses his gratitude to Professor Benoît Perthame  his course about the age-structured model at Summer School on Mathematical Biology 2023 at Vietnam Institute for Advanced Study in Mathematics (VIASM) and Professor Arnaud Ducrot for fruitful discussion. We deeply appreciate their stimulating, tireless encouragement during this work. Part of this work had been done while the second author was visiting VIASM, whose hospitality and financial support are acknowledged.

\section{\bf Preliminary results} \label{section:2}

Let us define $\mathcal{H}(x,u) = m(x)u-n(x)u^2$. Based on the  assumptions \ref{cond:cond1}, it is easy to check that for any $K>0$, there exists $k>0$ such that
\begin{align*}
\left|\left(m(x)u_1-n(x)u_1^2\right)-\left(m(x)u_2-n(x)u_2^2\right)\right| \leq k|u_1 - u_2|,
\end{align*}
for any $u_1,\text{ }u_2 \in [-K,K] \text{ and } x \in \overline{\O}$. This implies that the function $\mathcal{H}$ is Lipschitz continuous in $u\in [-K,K]$, uniformly in $x\in \overline{\O}$. Also, consider
\begin{align*}
F^l(u)(x)=[m(x)-p(x)-n(x)u]u;\textbf{ }\quad F^r(u,v)(x)=[m(x)-n(x)u]u+q(x)v,
\end{align*}
and
\begin{align*}
F^c(u,v)(x)=\left[m(x)-n(x)u-\dfrac{p(x)v}{u+v} \right]u+q(x)v;~ G(u,v)(x)=\left[\dfrac{p(x)u}{u+v}-q(x)\right]v.
\end{align*}
Based on the arguments in \cite[Section 1.1]{allen_asymptotic_2008}, it is known that $F^l$, $F^c$, $F^r$ and $G$ are local Lipchitz functions on $(u,v)\in \mathbb{R}_{+}^2$, independent with $x$. For the future purpose, with $u,~v\geq 0$, we define 
\begin{align*}
F^l(0,0)=F^c(0,0)=F^r(0,0)=G(0,0)=0.
\end{align*}
This still ensures the Lipschitz property of them. 

\begin{definition}\label{def:4}
Let $F:\R_+\times \overline{\mathcal{O}_A}\times\R\mapsto \R$, $F=F(t,x,a,u)$ be a continuous function. We say $F$ is \textit{monotone} if for any $M_0>0$ and $\sigma>0$, there exists $\lambda>0$ such that 
\begin{align*}
0\leq F(t,x,a,u)+\lambda u\leq F(t,x,a,v)+\lambda v,
\end{align*}
for any $ u,~v\in \R,~ 0\leq u\leq v\leq M_0,~(x,a)\in \overline{\mathcal{O}_A},~t\in [0,\sigma]$
\end{definition}
We recall some notations: For any $\textbf{u}=(u_1,u_2),\textbf{v}=(v_1,v_2)\in \R^2$, we define 
\begin{align*}
\textbf{u}\leq \textbf{v} \text{ if and only if } u_1\leq v_1,~u_2\leq v_2. 
\end{align*} 
Based on this notation, we can define the monotonicity similar to the function to $\R^2$ as follows

\begin{definition}\label{def:5}
Let $\textbf{G}:\overline{\mathcal{O}_A}\times\R^2\mapsto \R^2$, $\textbf{G}=\textbf{G}(x,a,\textbf{u})$ be a continuous function. We say $\textbf{G}$ is \textit{monotone} if for any $M_0,~M_1>0$, there exists $\lambda>0$ such that 
\begin{align*}
0\leq \textbf{G}(x,a,\textbf{u})+\lambda \textbf{u}\leq \textbf{G}(x,a,\textbf{v})+\lambda \textbf{v},
\end{align*}
for any $ \textbf{u},\textbf{v}\in \R^2,~ 0\leq \textbf{u}\leq \textbf{v}\leq (M_0,M_1),~(x,a)\in \overline{\mathcal{O}_A}$.
\end{definition}
The definition of a monotone function is useful for future work, as it is an essential condition for applying the semi-flow technique.

Next, for any $\lambda>0$ large enough, we consider
\begin{align*}
\textbf{F}_{\lambda}(u,v):=(F^c_{\lambda}(u,v),G_{\lambda}(u,v)):=\lambda (u,v)+(F^r(u,v),G(u,v)).
\end{align*}
Then, the Jacobian matrix of $\textbf{F}_{\lambda}$ is 
{\small
\begin{align*}
J\textbf{F}_{\lambda}(u,v)=\left(\begin{matrix}
\lambda + m(x)-2n(x)u -p(x)\dfrac{v^2}{(u+v)^2} & -p(x)\dfrac{u^2}{(u+v)^2}+q(x)\\
p(x)\dfrac{v^2}{(u+v)^2}& \lambda + p(x)\dfrac{u^2}{(u+v)^2}-q(x)
\end{matrix}\right)
\end{align*}}

Now, for any $M>0$, with $0\leq u_1\leq u_2\leq M$, $0\leq v_1\leq v_2\leq M$ and $\textbf{w}_1=(u_1,v_1),~\textbf{w}_2=(u_2,v_2),~\textbf{h}=\textbf{w}_2-\textbf{w}_1$, one has
\begin{align*}
\textbf{F}_{\lambda}(\textbf{w}_1)-\textbf{F}_{\lambda}(\textbf{w}_2)=\int_0^1J\textbf{F}_{\lambda}(\textbf{w}_1+t\textbf{h})\textbf{h}dt
\end{align*}
A direct calculation yields that
\begin{align}\label{boundedblow}
\begin{split}
\begin{array}{llll}
F^r_{\lambda}(\textbf{w}_1)-F^r_{\lambda}(\textbf{w}_2)
&\geq \left(\lambda -2\overline{n}M - \overline{p}\right)(u_2-u_1) -\overline{p}(v_2-v_1)\\
&\geq 2L(u_2-u_1)- L(v_2-v_1),\\
G_{\lambda}(\textbf{w}_1) - G_{\lambda}(\textbf{w}_2)&\geq (\lambda -\overline{q})(v_2-v_1) \\
&\geq L(v_2-v_1),
\end{array}
\end{split}
\end{align}
for $\lambda$ and $L>0$ large enough. 
This shows the non-monotone structure of $\textbf{F}$.

For convenience, we introduce the following notations
\begin{align*}
\begin{array}{llll}
\mathcal{X}:=C(\overline{\mathcal{O}_A}); &\mathcal{X}_{+}:=\{u\in \mathcal{X}:u\geq 0\};&\mathcal{X}_{0}:=\{u\in \mathcal{X}:u(x,A)=0,~\forall x\in \overline{\O}\}\\
\mathcal{Y}:=C(\overline{\O});&\mathcal{Y}_+:=\{u\in \mathcal{Y}:u\geq 0\},
\end{array}
\end{align*}
and
\begin{align*}
\begin{array}{lll}
&\mathcal{E}:=L^2(\mathcal{O}_A), ~\mathcal{E}_+:=\{u\in \mathcal{E}:u\geq 0~a.e.\},\\ 
& L^{\infty}_{+}(\mathcal{O}_A):=\{u\in L^{\infty}(\mathcal{O}_A):u\geq 0~a.e.\}.
\end{array}
\end{align*}


Throughout this paper, whenever we say a positive solution $u$ meaning that $u>0$ on $\mathcal{O}_A$ and $u\geq 0$ on $\overline{\mathcal{O}_A}$. It maybe possible that $u(x,0)>0$ and $u(x,A)=0$ for any $x\in \overline{\O}$. Consider $A_2<A_3<A$ and for any appropriate function $f$, it is known that
\begin{align*}
\displaystyle\int_0^{A} \beta(x,a) f(x,a)da=\displaystyle\int_0^{A_2} \beta(x,a) f(x,a)da = \displaystyle\int_0^{A_3} \beta(x,a) f(x,a)da
\end{align*}
since $\supp(\beta)\subset \overline{\O}\times (A_0,A_2]$. Thus, one can consider the system \eqref{eq:main} or \eqref{eq:main1} on $\overline{\mathcal{O}_{A_3}}$ instead of $\overline{\mathcal{O}_A}$, depending on the situation.

\subsection{\bf Spectral theory for age-structured model} \label{sec 2.1}

Let us begin the section with the spectral theory for our model in  \cite{guo_on_1994,kang_effects_2022,delgado_nonlinear_2006}.  For $d>0$, $\beta^*,$ $\mu^*\in C^{2,1}(\overline{\O}\times [0,A))$, $\mu^*$ satisfies \ref{cond:cond4} but not necessary positive. Define 
\begin{align*}
\mathcal{A}_{d,\mu^*}\phi(x,a) := d \Delta_N \phi(x,a) -\phi_a(x,a)-\mu^*(x,a)\phi(x,a),
\end{align*}
where
\begin{align*}
D(\mathcal{A}_{d,\mu^*})&:=\left\{\phi:\phi,\textbf{ }\mathcal{A}_{d,\mu^*}\phi\in \mathcal{E},~\left.\dfrac{\partial \phi}{\partial \textbf{n}}\right|_{\partial \O}=0,~\phi(x,0)=\int^A_0\beta(x,a)\phi(x,a)da \right\}\\
&\subset H^{1}\left((0,A),L^2(\O)\right)\cap L^{2}\left((0,A),H^2(\O)\right) \.
\end{align*}
(see, for example \cite{arendt_heat_2005}, about the Laplacian in Neumann boundary condition). It is well-known that if $u\in H^{1}\left((0,A),L^2(\O)\right)$ then $u\in C([0,A],L^2(\O))$ (see \cite[Section 5.9.2]{evans_partial_2010}). Thus, the continuity of $\phi\in D(\mathcal{A}_{d,\mu^*})$ ensures that the equation $\phi(x,0)=\displaystyle\int^A_0\beta(x,a)\phi(x,a)da$ is well-defined. 
Let us consider the eigenvalue problem as follows
\begin{align*}
\mathcal{A}_{d,\mu^*}\phi=\lambda \phi
\text{ or }
\left\{\begin{array}{l}
d \Delta_N \phi -\phi_a-\mu^*\phi=\lambda\phi,\\
 \phi(x,0) = \displaystyle\int_0^{A} \beta(x,a) \phi(x,a)da,\\
 \left.\dfrac{\partial \phi}{\partial \textbf{n}}\right|_{\partial \O}=0,
\end{array}\right.
\end{align*}
and define 
\begin{align*}
\mathcal{B}_{d,\mu^*}\phi:=-\mathcal{A}_{d,\mu^*}\phi =\phi_a(x,a)-d \Delta_N \phi(x,a)-\mu^*(x,a) \phi(x,a),
\end{align*}
with domain $D(\mathcal{B}_{d,\mu^*}):=D(\mathcal{A}_{d,\mu^*})$. Thanks to Theorem A.\ref{TheoA.1}, there exists unique principal eigenvalue $\lambda_{d,\mu^*}^*$ of the operator $\mathcal{B}_{d,\mu^*}$. Hence, $\mathcal{A}_{d,\mu^*}$ admits unique principal eigenvalue $\lambda_{d,\mu^*}=-\lambda_{d,\mu^*}^*$. Furthermore, for any other eigenvalue $\lambda$ of $\mathcal{A}_{d,\mu^*}$, one has $Re(\lambda)< \lambda_{d,\mu^*}$ and $\mu^*\mapsto \lambda_{d,\mu^*}$ is decreasing. The work in the paper will be based on the sign of $\lambda_{d,\mu^*}$ with the appropriate chosen of $\mu^*$ and $d>0$. 

Now, in the case $\mu^*\geq 0$, let us prove the $m$-disspative property of $\mathcal{A}_{d,\mu^*}$. For any $\phi \in D(\mathcal{A}_{d,\mu^*})$, it is well-known that $\left(d \Delta_N \phi,\phi\right)_{\mathcal{E}}\leq 0$. Thus, 
\begin{align*}
\left(\mathcal{A}_{d,\mu^*}\phi,\phi\right)_{\mathcal{E}}&= \left(d \Delta_N \phi,\phi\right)_{\mathcal{E}} -\left(\phi_a,\phi\right)_{\mathcal{E}}-\left(\mu^*\phi,\phi\right)_{\mathcal{E}},\\
&\leq -\int_0^A\int_{\O}\phi_a\phi dxda=-\dfrac{1}{2}\int_{\O}\int_0^A(\phi^2)_a da dx\\
&\leq \dfrac{1}{2}\int_{\O}\phi^2(x,0)dx=\dfrac{1}{2}\int_{\O}\left(\int^A_0\beta(x,a)\phi(x,a)da\right)^2dx\\
&\leq \dfrac{1}{2}\int_{\O}\int^A_0\beta^2(x,a)da\int_0^A\phi^2(x,a)dadx\\
&\leq \dfrac{1}{2}\sup_{x\in \overline{\O}}\int^A_0\beta^2(x,a)da \left(\phi,\phi\right)_{\mathcal{E}}
\end{align*}
which is $m$-disspative property. Similar to the one in \cite[Theorem 1]{guo_on_1994}, $(\lambda I-\mathcal{A}_{d,\mu})^{-1}$ exists for sufficiently large $\lambda>0$. In addition, it is a compact map. As the result,   $(\lambda I+\mathcal{B}_{d,\mu})^{-1}$ exists $\lambda>0$ large enough.


\subsection{\bf Semi-group theory for age-structured model in the bounded domain} \label{sec:semi-group}

Inspired by the semi-group in \cite{ducrot_travelling_2009}, it is necessary to consider an Hille-Yosida linear operator to connect $u_a$, $d\Delta_N u$ and the age-structure. In the case $\mu^*=\mu$, it is known that \ref{cond:cond5} implies the following Lebesgue measure 
\begin{align*}
\left|\{a\in [0,A]:\beta_{\min}(a)>0\}\right|>0.
\end{align*}
Thus, similar to the one in \cite[Theorem 1]{guo_on_1994}, the equation
\begin{align}\label{eq:original}
\left\{\begin{array}{lll}
w_t(t,x,a)=\mathcal{A}_{d,\mu}w(t,x,a),&t>0,~(x,a)\in \mathcal{O}_A,\\
w(0,x,a)=w_0(x,a),&(x,a)\in \mathcal{O}_A,&
\end{array}\right.
\end{align}
admits a unique $C_0$-semi-group $\Phi_{d,\mu}(t)$ generated from the infinitesimal generator $\mathcal{A}_{d,\mu}$ and
\begin{align*}
\left\{\begin{array}{lll}
w(t,\cdot,\cdot)=\Phi_{d,\mu}(t)w_0\in C([0,\infty),\mathcal{E}) &\text{ if }w_0\in \mathcal{E},\\
w(t,\cdot,\cdot)=\Phi_{d,\mu}(t)w_0\in C^1([0,\infty),\mathcal{E}) &\text{ if }w_0\in D(\mathcal{A}_{d,\mu}).
\end{array}\right.
\end{align*} 
Furthermore, by similar arguments to those in \cite[Lemma 1]{guo_on_1994}, with $0\leq a_0<A$, we obtain the following equation
\begin{align*}
\left\{\begin{array}{lll}
u_a=d\Delta_N u-\mu(x,a_0+a)u\\
u(\tau,x)=\phi(x)
\end{array}\right.
\end{align*}
admits a unique mild solution $u(x,a;\tau,a_0,\phi)=\mathcal{H}(a_0,\tau,a)\phi(x)$, where\\ $\mathcal{H}(a_0,\tau,a)$, $0\leq \tau\leq a\leq A-a_0$ is a family of uniformly linear bounded compact positive operator on $\mathcal{E}$ and strongly continuous on $\tau,s$. In addition, one has
\begin{align}\label{estimate1}
\begin{split}
&\mathcal{H}(a_0,\tau,a) \geq e^{\displaystyle -\int_\tau^a\mu_{\max}(a_0+k)dk}T_{d\Delta_N}(a-\tau) \\
&\mathcal{H}(a_0,\tau,a)\leq e^{\displaystyle -\int_\tau^a\mu_{\min}(a_0+k)dk}T_{d\Delta_N}(a-\tau)
\end{split}
\end{align}
where $T_{d\Delta_N}(t)$, $t>0$ is the analytic semi-group generated by Neumann Laplacian $-d\Delta_N$, which is well-known to be strongly positive on $C(\overline{\O})$ and $L^2(\O)$. On the other hand, from \eqref{eq:original}, the characteristic line calculation yields that
\begin{align*}
w(t,x,a)=\left\{\begin{array}{lll}
\mathcal{H}(a-t,0,t)w_0(\cdot,a-t)(x),&a\geq t,\\
\mathcal{H}(0,0,a)\displaystyle \int_0^A\beta(\cdot,s)w(t-a,\cdot,s)ds(x),&t>a.\\
\end{array}\right.
\end{align*}
or the second version
\begin{align}\label{charaterics-line}
w(t,x,a)=\left\{\begin{array}{lll}
\mathcal{H}(a-t,0,t)w_0(\cdot,a-t)(x),&a\geq  t,\\
\mathcal{H}(0,0,a)\displaystyle \int_{A_0}^{A_2}\beta(\cdot,s)w(t-a,\cdot,s)ds(x),&t> a.\\
\end{array}\right.
\end{align}
since $\beta(x,a)=0$, $a>A_2$ or $a<A_0$ for any $x\in \overline{\O}$. Thus, if $w_0\geq 0$, one has that $w>0$ on $\overline{\O}\times [0,A)$ thanks to the strong positivity of the Neumann Laplacian semigroup $T_{d\Delta_N}(t)$, $t>0$. In addition, thanks to positive property of the semi-group $\Phi_{d,\mu^*}(t)$ and  \cite[Theorem 4.3 and 4.5]{magal_monotone_2019}, the comparison principle holds for \eqref{eq:original}. 

Next, consider $w_0\in L^{\infty}(\mathcal{O}_A)$ and there exists $M>0$ such that
\begin{align*}
0\leq w_0(x,a)\leq M\pi(a) \text{ on }(x,a)\in \overline{\mathcal{O}_A}.
\end{align*}
Define $z=M\pi(a)$, we can check that $A_{d,\mu}z-z_t=\mu_{\min}z-\mu z\leq 0$ on $\overline{\mathcal{O}_A}$. Then, one has
\begin{align*}
0\leq \Phi_{d,\mu^*}(t)w_0=w(t,\cdot,\cdot;w_0)\leq  M\pi(a) \text{ on }(x,a)\in \overline{\mathcal{O}_A},~t>0.
\end{align*}
For our purpose, consider
\begin{align*}
\mathcal{B}_{\pi}(\mathcal{O}_A):=\{u:\mathcal{O}_A \rightarrow \R:u\text{ is measurable},~L(a)|u|<\infty \text{ on }\mathcal{O}_A\}
\end{align*}
and norm $||u||_{\mathcal{B}_{\pi}(\mathcal{O}_A)}=||Lu||_{L^{\infty}(\mathcal{O}_A)}$. One can check it is a Banach space and $\mathcal{B}_{\pi}(\mathcal{O}_A)\subset L^{\infty}(\mathcal{O}_A)$.

Now, suppose that $\supp(w_0)\cap \mathcal{O}_{A_2}\neq \emptyset$. Using similar arguments to those in \cite[Lemma 4.10]{langlais_large_1988}
(with $\beta_{n}=\beta_{\min},~\mu_n=\mu_{\max}(a)$ and $\nu=0$) and the age-structured comparison principle, there exists $T^*_d>0$ such that
\begin{align*}
w(t,x,0)>0,~t>T^*_d,~x\in\O.
\end{align*}
Thus, for any $t> T^*_d+A$, one has
\begin{align*}
w(t,x,a)=\mathcal{H}(0,0,a)\left[w(t,\cdot,0)\right](x)>0,~\forall x\in \overline{\O},~ 0<a<A,
\end{align*}
since $w(t-a,\cdot,0)=\displaystyle\int_0^{A}\beta(\cdot,s)w(t-a,\cdot,s)ds$ and thanks to the strongly positivity of  $T_{d\Delta_N}(t)$ on $\overline{\O}$ for $t>0$.  Then again, one has 
\begin{align*}
w(t,x,0)=\displaystyle\int_0^{A}\beta(x,s)w(t,x,s)ds>0,~\forall t>T_d^*+A,~x\in \overline{\O}.
\end{align*}
The definition of $w>0$ may understand as almost everywhere or the one in Definition A.\ref{DefA.1} depending on the situation.




Now, we are ready to state the first existence result 


%

\begin{lemma}\label{local}
Assume \ref{cond:cond1} to \ref{cond:cond5} hold. Suppose further that $u_0,~v_0\in \mathcal{B}_{\pi}(\mathcal{O}_A)$ and $u_0,~v_0\geq 0$. Then, there exists a unique local positive solution $(u,v)=(u(t),v(t))=(u(t,x,a; u_0, v_0), v(t,x,a; u_0, v_0))$  to \eqref{eq:main} on $[0,T^*)\times\overline{\mathcal{O}_A}$ starting by the initial condition $(u_0,v_0)$  with  $T^*\leq \infty$ is the maximum existence time and $u,v\geq 0$.
\end{lemma}

\begin{proof}
It is easy to check $(u_0,v_0)\in \mathcal{E}\times \mathcal{E}$.  We can rewrite the system \eqref{eq:main} as the following Cauchy problem
\begin{align}
\left\{\begin{array}{lll}
u_t = \mathcal{A}_{d_u,\mu}u + F^c(u,v),&t>0,~(x,a)\in \mathcal{O}_A,\\
v_t = \mathcal{A}_{d_v,\mu}v + G(u,v),&t>0,~(x,a)\in \mathcal{O}_A,\\
u(0,x,a)=u_0(x,a),&(x,a)\in \mathcal{O}_A,\\
v(0,x,a)=v_0(x,a),&(x,a)\in \mathcal{O}_A.
\end{array}\right.
\end{align}

Let us define $\textbf{A} := (\mathcal{A}_{d_u,\mu},\mathcal{A}_{d_v,\mu})$, $D(\textbf{A}):=D(\mathcal{A}_{d_u,\mu})\times D(\mathcal{A}_{d_v,\mu})$ and $\textbf{F}^c(u,v)=(F^c(u,v),G(u,v))$ for any $(u,v)\in \mathcal{E}\times \mathcal{E}$. It is easy to check that $\textbf{A}$ is a closed operator since $\mathcal{A}_{d_u,\mu}$ and $\mathcal{A}_{d_v,\mu}$ are closed operators. Using same arguments in \cite{guo_on_1994}, with $i=u \text{ or }v$, one has the $m$-disspative property of $\mathcal{A}_{d_i,\mu}$ as follows
\begin{align*}
\left(\mathcal{A}_{d_i,\mu}\phi-C\phi,\phi \right)\leq 0,~\forall \phi \in L^2(\mathcal{O}_A),
\end{align*}
where $C=\dfrac{1}{2}\sup\limits_{x\in \O}\displaystyle\int_0^A\beta^2(x,a)da$. Thus, thanks to \cite[Theorem 4.2]{pazy_semigroups_2012}, we have that
\begin{align*}
\left\|[(K +C)I - \mathcal{A}_{d_i,\mu}]\phi\right\|_{L^2(\mathcal{O}_A)}\geq K ||\phi||_{L^2(\mathcal{O}_A)},~\forall K>0,~\phi \in L^2(\mathcal{O}_A).
\end{align*}
Define $\lambda = K+C>C>0$, then
\begin{align*}
||(\lambda I - \mathcal{A}_{d_i,\mu})\phi||_{L^2(\mathcal{O}_A)}\geq (\lambda - C)||\phi||_{L^2(\mathcal{O}_A)},~\forall \lambda>C.
\end{align*}
In the proof of \cite[Theorem 1]{guo_on_1994}, for $\lambda^0>0$ large enough, then $\lambda^0 \in \rho(\mathcal{A}_{d_i,\mu})$, resolvent set of $\mathcal{A}_{d_i,\mu}$. We also know that $(\lambda^0 I - \mathcal{A}_{d_i,\mu})^{-1}$
is a compact mapping from $L^2(\mathcal{O}_A)$ to itself whenever $\lambda^0 \in \rho(\mathcal{A}_{d_i,\mu})$. Thus, $\sigma(\mathcal{A}_{d_i,\mu})=\mathbb{C}\setminus \rho(\mathcal{A}_{d_i,\mu})$ consists only eigenvalue of $\mathcal{A}_{d_i,\mu}$. On the other hand, based on Appendix \ref{Appen A}, there exists unique principal eigenvalue $\lambda_{i,\mu}\in \R$ of $\mathcal{A}_{d_i,\mu}$ so that if $\lambda_0$ is any other eigenvalue of $\mathcal{A}_{d_i,\mu}$, then $Re(\lambda_0)< \lambda_{i,\mu} $. Thus, if  $\lambda^0 >\lambda_{i,\mu}$ and $\lambda^0>0$, then $\lambda^0 \in \rho(\mathcal{A}_{d_i,\mu})$. As the results, one has
\begin{align*}
||(\lambda I - \mathcal{A}_{d_i,\mu})^{-1}\phi||_{L^2(\mathcal{O}_A)}\leq \dfrac{1}{\lambda - C}||\phi||_{L^2(\mathcal{O}_A)},~\forall \lambda>\max\{C,\lambda_{i,\mu}\}.
\end{align*} 
Hence, for any $\lambda>\max\{C,\lambda_{u,\mu},\lambda_{v,\mu}\}$, 
\begin{align*}
||(\lambda I - \textbf{A})^{-1}||_{L^2(\mathcal{O}_A) \times L^2(\mathcal{O}_A)}\leq \dfrac{1}{\lambda - C},
\end{align*}
where $||(u,v)||_{L^2(\mathcal{O}_A)\times L^2(\mathcal{O}_A)} = \max\left\{||u||_{L^2(\mathcal{O}_A)},||v||_{L^2(\mathcal{O}_A)}\right\}$ is an usual product norm. As the results, \textbf{A} generates a positive $C_0$-semi-group $\Phi_{d_u,d_v,\mu}(t)$ thanks to Section \ref{sec:semi-group} from $\mathcal{E}\times \mathcal{E}$ into itself and $\Phi_{d_u,d_v,\mu}(t)=(\Phi_{d_u,\mu},\Phi_{d_v,\mu})$ for any $t>0$, where $\Phi_{d_i,\mu}$, $i=u,v$ is the semi-group generated by $\mathcal{A}_{d_i,\mu}$. Similarly to Section \ref{sec 2.1}, one can show that 
\begin{align*}
\Phi_{d_u,d_v,\mu}(t):\mathcal{B}_{\pi}(\mathcal{O}_A)\times \mathcal{B}_{\pi}(\mathcal{O}_A)\rightarrow \mathcal{B}_{\pi}(\mathcal{O}_A)\times \mathcal{B}_{\pi}(\mathcal{O}_A)
\end{align*}
and $||\Phi_{d_u,d_v,\mu}(t)(u,v)||_{\mathcal{B}_{\pi}(\mathcal{O}_A)\times\mathcal{B}_{\pi}(\mathcal{O}_A)}<M$ if $||(u,v)||_{\mathcal{B}_{\pi}(\mathcal{O}_A)\times\mathcal{B}_{\pi}(\mathcal{O}_A)}<M$ for any $t>0$, which is continuous.

%

On the other hand, \textbf{F} is a locally Lipchitz on $\mathcal{B}_{\pi}(\mathcal{O}_A)\times\mathcal{B}_{\pi}(\mathcal{O}_A)$ to itself. This comes directly from the fact that $F^c$ and $G$ are local Lipchitz functions on $(u,v)\in \mathbb{R}_{+}^2$ and $\mathcal{B}_{\pi}(\mathcal{O}_A) \subset L^{\infty}(\mathcal{O}_A)$.
By similar arguments to those in \cite[Lemma 5.2.4 and Lemma 5.2.5]{magal_theory_2018}, the system \eqref{eq:main} admits a unique mild solution $(u,v)=(u(t),v(t)) =(u(t,x,a; u_0, v_0), v(t,x,a; u_0, v_0))$
on $\overline{\mathcal{O}_A}$ and $0<T^*_c\leq \infty$ is a maximum existence time so that for any $ t\in [0,T^*_c)$, one has
\begin{align}\label{eq:mildc}
(u(t),v(t))=\Phi_{d_u,d_v,\mu}(t) (u_0,v_0)+\int_0^t\Phi_{d_u,d_v,\mu}(t-s)\textbf{F}^c(u(s),v(s))ds,
\end{align}
Similarly, with $\textbf{F}^l(u,v):=(F^l(u,v),G(u,v))$ and $\textbf{F}^r(u,v):=(F^r(u,v),G(u,v))$, there exist maximum existence time $0<T^*_{l},~T^*_{r}\leq \infty$ such that for any $ t\in [0,T^*_l)$, one has
\begin{align*}
&(u^l(t),v^l(t))=\Phi_{d_u,d_v,\mu}(t) (u_0,v_0)+\int_0^t\Phi_{d_u,d_v,\mu}(t-s)\textbf{F}^l(u^l(s),v^l(s))ds,
\end{align*}
and, for any $t\in [0,T^*_r)$,
\begin{align*}
&(u^r(t),v^r(t))=\Phi_{d_u,d_v,\mu}(t) (u_0,v_0)+\int_0^t\Phi_{d_u,d_v,\mu}(t-s)\textbf{F}^r(u^r(s),v^r(s))ds.
\end{align*}
Since $\textbf{F}^l(u,v)\leq \textbf{F}(u,v)\leq \textbf{F}^r(u,v)$, it means that $(u^l(t),v^l(t))$ is the sub-solution of \eqref{eq:mildc} and $(u^l(t),v^l(t))$ is the super-solution. 
One check that, with $\lambda>0$ large enough $(\lambda I - \textbf{A})^{-1}(u_1,u_2)\geq 0$ if $u_1,~u_2\geq 0$ and 
$(u_1,u_2)\mapsto \textbf{F}^l(u_1,u_2)+\lambda (u_1,u_2)$ is positive and non-decreasing in the sense of \cite[Assumption 4.1 and Assumption 4.4]{magal_monotone_2019}. Since $(u^l(t),v^l(t))$ is bounded in local time, it follows from \cite[Theorem 4.3 and Theorem 4.5]{magal_monotone_2019} that $u^l\geq 0$, $v^l\geq 0$ and the comparison principle holds for these solutions on $\overline{\mathcal{O}_A}$ and for any suitable time $t>0$. Same conclusion holds for $(u^r(t),v^r(t)) $. As the results, it follows from \cite[Proposition 5.2 and Proposition 5.4]{magal_monotone_2019} that
\begin{align*}
(0,0)\leq (u^l(t),v^l(t)) \leq (u(t),v(t)) \leq(u^r(t),v^r(t)) 
\end{align*}
for any suitable time $t>0$ since $(u(t),v(t))$ is a super-solution of $(u^l(t),v^l(t))$ and a sub-solution of $(u^r(t),v^r(t))$.


%

\end{proof}

Next, the global-in-time existence result will be shown. Before going to prove it, let $Y_{\alpha}$ be the solution of 
\begin{align*}
\left\{\begin{array}{llll}
Y_a+\mu_{\min}(a)Y=\alpha Y-Y^2,\\
Y(0) = y_0.
\end{array}\right.
\end{align*}
where $\alpha$ and $y_0$ are constants which will be determined later. Then, one has
\begin{align*}
Y_{\alpha}(a)=\dfrac{ {e^{\alpha a -  \displaystyle\int_0^a\mu_{\min}(k)dk}}}{1/y_0+e^{\alpha a -  \displaystyle\int_0^a\mu_{\min}(k)dk}}
=\dfrac{e^{F_\alpha(a)}}{1/y_0+e^{F_\alpha(a)}}=\dfrac{\pi(a)e^{\alpha a}}{1/y_0+\pi(a)e^{\alpha a}}
\end{align*}
where $F_\alpha(a):=\alpha a -  \displaystyle\int_0^a\mu_{\min}(k)dk$. Next, consider $\hat{\beta}=\max\limits_{\overline{\mathcal{O}_A}}\beta$. By similar arguments to the one in \cite[Theorem 14]{delgado_nonlinear_2006} and \cite[Proposition 5.3]{kang_effects_2022}, with $\alpha$ large enough, there exist $y_0>0$ such that
\begin{align*}
&\int_0^Ae^{F_\alpha(a)}da >\dfrac{1}{\hat{\beta}},\\
&\int_0^A\dfrac{e^{F_\alpha(a)}}{1+y_0e^{F_\alpha(a)}}da=\frac{1}{\hat{\beta}}
\end{align*}
Thus,
\begin{align*}
\int_0^A\beta(x,a)Y_{\alpha}(a)da\leq \hat{\beta}\int_0^A\dfrac{e^{F_\alpha(a)}}{1/y_0+e^{F_\alpha(a)}}da=y_0=Y_{\alpha}(0).
\end{align*}
So, with $\alpha>0$ large enough, $Y_{\alpha}$ satisfies upper-bounded of age-structured condition and $0\leq Y_{\alpha}\leq 1$. 

Now, we will construct a super-solution of \eqref{eq:main}. Let us define $\overline{u}:=MY_{\alpha}(a)$, $\overline{v}:=NY_{\alpha}(a)$.

\begin{lemma}\label{Lemma:super}
There exist $M,~N,~\alpha>0$ large enough such that $(\overline{u},\overline{v})$ is a super-solution of \eqref{eq:main}.

\end{lemma}

\begin{proof}
Consider $M>0$ large enough that will be determined later. Let $K>1$ large enough so that
\begin{align*} 
\sup_{x\in \overline{\O}} p(x) \leq K \inf_{x\in \overline{\O}} q(x),
\end{align*}
Choose $N=2KM>M>0$. 
For $\alpha>1$ large enough so that $Y_{\alpha}$ satisfying the upper-bound age-structured condition, one have
{\small
\begin{align*}
F^r(\overline{u},\overline{v})-\overline{u}_t+\mathcal{A}_{d_u,\mu}\overline{u} 
&\leq \hat{m} MY_{\alpha}(a)-\underline{n}M^2Y_{\alpha}^2(a) -M(\alpha Y_{\alpha}-Y_{\alpha}^2)+\hat{q}2KMY_{\alpha}(a)\\
&=MY_{\alpha}(a)[\hat{m}+(1-\underline{n}M)MY_{\alpha}(a)-\alpha +2\hat{q}K]\\
&\leq MY_{\alpha}(a)[\hat{m}-\alpha +2\hat{q}K]\leq 0,
\end{align*}}
provided that $M>\dfrac{1}{\hat{n}}$ and $\alpha>\hat{m}+2\hat{q}K$ large enough. Here, we use $\hat{m}=\max\limits_{\overline{\O}}m,~\underline{n}=\min\limits_{\overline{\O}}n>0,~\hat{q}=\max\limits_{\overline{\O}}q$. On the other hand, with these constants 
\begin{align*}
\mathcal{A}_{d_v,\mu}\overline{v} = -\overline{v}_a - \mu(x,a)\overline{v}&\leq -N(Y_{\alpha})_a- \mu_{\min}(a)NY_{\alpha}\\
&=NY_{\alpha}^2-\alpha NY_{\alpha} \leq NY_{\alpha}-\alpha NY_{\alpha}=NY_{\alpha}(1 - \alpha)\leq 0
\end{align*}
and
\begin{align*}
G(\overline{u},\overline{v})-\overline{v}_t+\mathcal{A}_{d_v,\mu}\overline{v}&\leq p(x)MY_{\alpha}(a)-q(x)NY_{\alpha}(a)\\
&\leq q(x)Y_{\alpha}(a)(KM-N)=-KMq(x)Y_{\alpha}(a)\leq 0.
\end{align*}
since $Y_{\alpha}\leq 1$. Furthermore,
\begin{align*}
\int_0^A\beta(x,a)\overline{u}(t,x,a)da=\int_0^A\beta(x,a)MY_{\alpha}(a)da \leq MY_{\alpha}(0)= \overline{u}(t,x,0).
\end{align*}
Similar to $\overline{v}$. Thus, $(\overline{u},\overline{v})$ is a super-solution of \eqref{eq:main}.
\end{proof}



We are in the position to show the global-time existence results with bounded property in time. 



\begin{theorem}\label{theo:2}
 Assume \ref{cond:cond1} to \ref{cond:cond5} hold and  $u_0,v_0\in  L^{\infty}_{+}(\mathcal{O}_A)$. Suppose further that
 \begin{align}\label{bounded:above}
0\leq u_0\leq MY_{\alpha},~0\leq v_0\leq NY_{\alpha} \text{ on }\overline{\mathcal{O}_A}
 \end{align}
where $M$, $N$ and $\alpha$ are determined in Lemma \ref{Lemma:super}. Then, there exists a unique global non-negative solution $(u,v)=(u(t),v(t))=(u(t,x,a; u_0, v_0), v(t,x,a; u_0, v_0))$ of \eqref{eq:main} on $\overline{\mathcal{O}_{A}}$ starting by the initial condition $(u_0,v_0)$. Moreover, it is bounded, i.e.
 \begin{align*}
 0\leq u(t,x,a;u_0,v_0)\leq M,~0\leq v(t,x,a;u_0,v_0)\leq N 
 \end{align*}
 on $\overline{\mathcal{O}_{A}}$ and for any $t>0$ if $u_0\not\equiv 0$ and $v_0\not\equiv 0$. Furthermore,  suppose that $\supp(u_0)\cap \mathcal{O}_{A_2}\neq \emptyset$ and  $\supp(v_0)\cap \mathcal{O}_{A_2}\neq \emptyset$, there exists $T^*>0$ such that
\begin{align*}
u(t)>0,~v(t)>0,~\forall t\geq  T^*+A,\text{ on }\overline{\O}\times [0,A).
\end{align*}

\end{theorem}

\begin{proof}

We can show that $u_0,~v_0\in \mathcal{B}_{\pi}(\mathcal{O}_A)$. Then, it follows from the fact that $(\overline{u},\overline{v})=(MY_\alpha,NY_\alpha)$ is a super-solution of \eqref{eq:main} and the comparison principle for age-structured model that
\begin{align*}
&0\leq u(t,x,a;u_0,v_0)\leq u^r(t,x,a;u_0,v_0)\leq MY_{\alpha}\leq M,\\
&0\leq v(t,x,a;u_0,v_0)\leq u^r(t,x,a;u_0,v_0) \leq  NY_{\alpha}\leq N=2KM,
\end{align*}
for any $(x,a)\in \overline{\mathcal{O}_A}$. In addition, one can prove that
\begin{align*}
\lim_{a\rightarrow A}u(t,x,a;u_0,v_0)=0,~
\lim_{a\rightarrow A}v(t,x,a;u_0,v_0)=0,
\end{align*}  
uniformly with $x\in\overline{\O}$ and $t>0$. Thus, $(u(t),v(t))$ exists globally in time. Similarly, $(u^l(t),v^l(t))$ also exists globally in time.

Next, thanks to Section \ref{sec:semi-group}, we can check that a modification formula for the mild solution in \eqref{eq:mildc} as follows
{\small
\begin{align*}
(u(t),v(t))&\geq (u^l(t),v^l(t))\\&=e^{-\lambda t}\Phi_{d_u,d_v,\mu}(t) (u_0,v_0)+\int_0^te^{-\lambda (t-s)}\Phi_{d_u,d_v,\mu}(t-s)\textbf{F}^l_{\lambda}(u(s),v(s))ds
\\&\geq e^{-\lambda t}\Phi_{d_u,d_v,\mu}(t) (u_0,v_0) \text{ on }\overline{\mathcal{O}_A},~\forall t>0,
\end{align*}}\\
where $\lambda>0$ large enough so that $\textbf{F}_{\lambda}^l(u,v):=\textbf{F}^l(u,v)+\lambda (u,v)\geq (0,0)$ for any $0\leq u,~v\leq N$. 

 Suppose that $\supp(u_0)\cap \mathcal{O}_{A_2}\neq \emptyset$ and  $\supp(v_0)\cap \mathcal{O}_{A_2}\neq \emptyset$, there exists $T^*_{d_u},~T^*_{d_v}>0$ such that 
\begin{align*}
\Phi_{d_u,\mu}(t)(u_0)>0,~\forall t>T^*_{d_u}+A,~\Phi_{d_v,\mu}(t)(v_0)>0,~\forall t>T^*_{d_u}+A
\end{align*}
on $\overline{\O}\times [0,A)$.
As the results, one can check that
\begin{align*}
\Phi_{d_u,d_v,\mu}(t) (u_0,v_0)>(0,0),~\forall t>T^*+A, \text{ on $\overline{\O}\times [0,A)$}
\end{align*}
where $T^*=\max\{T^*_{d_u},T^*_{d_v}\}$. Thus, one has $u(t)>0$, $v(t)>0$ $\text{on }\overline{\O}\times [0,A)$, $\forall t>T^*+A$.

This concludes the proof.

\end{proof}

\begin{remark}
When $u_0,~v_0 \in C_c(\overline{\O}\times [0,A))$, the space of continuous functions having compact support subset of $\overline{\O}\times [0,A)$, one can establish the continuity of $(t,x,a)\mapsto u(t,x,a;u_0,v_0)$, $(t,x,a)\mapsto v(t,x,a;u_0,v_0)$  on $(0,\infty)\times \overline{\mathcal{O}_A}$, except possibly on measure-zero sets such as $\{t=a\}$, thanks to characteristics line formula \eqref{charaterics-line}. If we suppose further that $u_0,~v_0$ satisfy the age-structured condition, the continuity holds over the entire domain $(0,\infty)\times \overline{\mathcal{O}_A}$.
\end{remark}


\begin{remark}\label{remark2}
The condition \eqref{bounded:above} can be replaced by an another bounded time-independent super-solution $(\overline{u},\overline{v})$ of \eqref{eq:main}. In this case, we assume 
\begin{align*}
0\leq u_0\leq \overline{u};~0\leq v_0\leq \overline{v}.
\end{align*}
By comparison principle, we also obtain that 
\begin{align*}
0\leq u(t,x,a;u_0,v_0)\leq \overline{u}\leq M_0;\quad 0\leq v(t,x,a;u_0,v_0)\leq \overline{v}\leq M_0
\end{align*}
for some constant $M_0>0$. As the results, the solution exists globally in time.

\end{remark}

\begin{remark}\label{remark3}
If we consider $A_3$ satisfies $A_2<A_3<A$, we can remove the condition \eqref{bounded:above} to obtain the global-time existence on $\overline{\O}\times [0,A_3]$. Since $Y_{\alpha}(a)>0$ on $\overline{\O}\times [0,A_3]$, one can find $N>0$ large enough such that
\begin{align*}
0\leq u_0\leq NY_{\alpha}(a),~0\leq v_0\leq NY_{\alpha}(a) \text{ on }\overline{\O}\times [0,A_3].
\end{align*} 

\end{remark}

\subsection{\bf Comparison principle for the age-structured model}

This section states comparison principle for some age-structured model that will be used throughout this paper. 

There are two type of age-structured comparison principles based on the nonlinear term that we state here: the monotone and super-sub comparison principle.

\subsection*{Monotone comparison principle}

The statements are based on the work of \cite[Theorem 4.3 and 4.5]{magal_monotone_2019} and \cite[Section 5]{ducrot_age-structured_2024}.

\begin{proposition}[Monotone comparison principle for one equation]\label{prof:comparoneeq}
Assume \ref{cond:cond3} to \ref{cond:cond5} hold. Let $F:\R_+\times \overline{\mathcal{O}_A}\times\R\mapsto \R$, $F=F(t,x,a,u)$ be a continuous function and local Lipschitz in $u$, independent with $t$, $x$ and $a$ and $F(t,x,a,0)=0$. Furthermore, $F$ satisfies the \textit{monotonicity} condition given in Definition \ref{def:4}.
Then, the following equation 
\begin{align}\label{eq:KPP1}
\left\{\begin{array}{lll}
w_t(t,x,a)=\mathcal{A}_{d,\mu}w(t,x,a)+F(t,x,a,w(t,x,a)),&t>0,~(x,a)\in \mathcal{O}_A,\\
w(0,x,a)=w_0(x,a),&(x,a)\in \mathcal{O}_A,&
\end{array}\right.
\end{align}
admits the comparison principle, that is for any $w_0,~w_1\in \mathcal{B}_{\pi}(\mathcal{O}_A)$ and $0\leq w_0\leq w_1$, one has
\begin{align*}
0\leq w(t,x,a;w_0)\leq w(t,x,a;w_1),~(x,a)\in \mathcal{O}_A\text{ and for any possible }t>0,
\end{align*}
where $w(t,x,a;w_i)$, $i=1,2$ are the solutions of \eqref{eq:KPP1} with initial condition $w(0,x,a;w_i)=w_i(x,a)$, $i=1,2$. The same property holds for the super-solution and sub-solution.
\end{proposition}

We present an analogue of the monotone comparison principle, a useful tool for future work, as follows

\begin{proposition}[Monotone comparison principle for system]\label{prof:comparsys}
Assume \ref{cond:cond3} to \ref{cond:cond5} hold. Let  $\textbf{A}$ be the operator in Lemma \ref{local} and $\textbf{G}:\overline{\mathcal{O}_A}\times\R^2\mapsto \R^2$, $\textbf{G}=\textbf{G}(x,a,u)$ be a continuous function and local Lipschitz in $u$, independent with $x$ and $a$, $\textbf{G}(x,a,0)=0$. Furthermore, $\textbf{G}$ satisfies the \textit{monotonicity} condition given in Definition \ref{def:5}.
Then, the following equation 
\begin{align}\label{eq:KPP3}
\left\{\begin{array}{lll}
\textbf{u}_t=\textbf{A}\textbf{u}+\textbf{G}(\textbf{u}),&t>0,\\
\textbf{u}(0)=\textbf{u}_0,
\end{array}\right.
\end{align}
admits the comparison principle, that is for any $\textbf{u}_0,~\textbf{u}_1\in \mathcal{B}_{\pi}(\mathcal{O}_A)\times   \mathcal{B}_{\pi}(\mathcal{O}_A)$ and $0\leq \textbf{u}_0\leq \textbf{u}_1$, one has
\begin{align*}
0\leq \textbf{u}(t;\textbf{u}_0)\leq \textbf{u}(t;\textbf{u}_1)\text{ for any possible }t>0,
\end{align*}
where $\textbf{G}(\textbf{u})(a)(x)=\textbf{G}(x,a,\textbf{u})$ and $\textbf{u}(t;\textbf{u}_i)$, $i=1,2$ are the solutions of \eqref{eq:KPP3} with initial condition $\textbf{u}(0;\textbf{u}_i)=\textbf{u}_i$, $i=1,2$. The same property holds for the super-solution and sub-solution.
\end{proposition}
One can view the comparison principle as follows
for any $\textbf{u}_0=(u_{0},v_{0}),\textbf{u}_1=(u_1,v_1)\in \mathcal{B}_{\pi}(\mathcal{O}_A)$ and $0\leq \textbf{u}_0\leq \textbf{u}_1$, one has
\begin{align*}
&0\leq u(t,x,a;u_{0},v_{0})\leq u(t,x,a;u_{1},v_{1}),\\ 
&0\leq v(t,x,a;u_{0},v_{0})\leq v(t,x,a;u_{1},v_{1}),
\end{align*}
for any $(x,a)\in\mathcal{O}_A$ and any possible $t>0$, where $(u(t,x,a;u_{i},v_{i})),v(t,x,a;u_{i},v_i))$, $i=0,1$ are the solution of the system \eqref{eq:main} with the initial condition $(u(0,x,a;u_{i},v_{i})),v(0,x,a;u_{i},v_i))=(u_i(x,a)),v_i(x,a))$, $i=0,1$. 


The systems associated with this type of function $F$ and $\textbf{G}$ are called \textit{monotone systems}.

\subsection*{Super-sub comparison principle}
Define $\textbf{F}^c=\textbf{F}$. Suppose there exists  $\textbf{F}^{-}$ and $\textbf{F}^+$ such that
\begin{align*}
\textbf{F}^{-}\leq \textbf{F}^c\leq \textbf{F}^+
\end{align*}
These three functions are distinct and $\textbf{F}^-$ and $\textbf{F}^+$ are monotone. Due to non-monotone structure of $\textbf{F}^c$ of our system \eqref{eq:main} (see \eqref{boundedblow}), the comparison principle can only be stated in the following form based on the following results \cite[Proposition 5.2 and 5.4]{magal_monotone_2019}, see also \cite[Lemma 2.1 of Chapter 8]{wu_theory_1996}.

\begin{proposition}[Super-sub comparison principle for system]\label{prof:comparsys1}
Assume \ref{cond:cond3} to \ref{cond:cond5} hold. Let  $\textbf{A}$ and $\textbf{F}$ be the operator and function in Lemma \ref{local}. Consider the following systems
\begin{align}\label{eq:KPP2}
\left\{\begin{array}{lll}
\textbf{u}_t=\textbf{A}\textbf{u}+\textbf{F}^{i}(\textbf{u}),&t>0,\\
\textbf{u}(0)=\textbf{u}_0,
\end{array}\right.
\end{align}
for $i=-,c,+$. Then, the system with $i=c$ or $\textbf{F}^i=\textbf{F}$ admits the comparison principle, that is for any $\textbf{u}_0\in \mathcal{B}_{\pi}(\mathcal{O}_A)\times   \mathcal{B}_{\pi}(\mathcal{O}_A)$ and $0\leq \textbf{u}_{-}(0)\leq \textbf{u}_0\leq \textbf{u}_{+}(0)$, one has
\begin{align*}
0\leq \textbf{u}_{-}(t)\leq \textbf{u}(t;\textbf{u}_0)\leq \textbf{u}_{+}(0)\text{ for any possible }t>0,
\end{align*}
where $\textbf{u}_{-}=\textbf{u}_{-}(t)$ is a sub-solution of \eqref{eq:KPP2} with $i=-$, $\textbf{u}_{+}=\textbf{u}_{+}(t)$ is a super-solution of \eqref{eq:KPP2} with $i=+$ and $\textbf{u}(t;\textbf{u}_0)$ is the solution of \eqref{eq:KPP2} with $i=c$, corresponding to the initial condition $\textbf{u}(0;\textbf{u}_0)=\textbf{u}_0$.
\end{proposition}
The arguments in Theorem \ref{theo:2} state that \eqref{eq:main} and \eqref{eq:KPP2} are equivalent. Thus, the comparison principle also holds for \eqref{eq:main}, that is
for any $\textbf{u}_0=(u_{0},v_{0})\in \mathcal{B}_{\pi}(\mathcal{O}_A)\times \mathcal{B}_{\pi}(\mathcal{O}_A)$ and $0\leq \textbf{u}_{-}(0)\leq \textbf{u}_0$, one has
\begin{align*}
&0\leq u_{-}(t,x,a)\leq u(t,x,a;u_{0},v_{0}),\\ 
&0\leq v_{-}(t,x,a)\leq v(t,x,a;u_{0},v_{0}),
\end{align*}
for any $(x,a)\in\mathcal{O}_A$ and any possible $t>0$, where $\textbf{u}_{-}(t)=(u_{-}(t,x,a),v_{-}(t,x,a))$ is the sub-solution of \eqref{eq:KPP2} with $i=-$ and $(u(t,x,a;u_{0},v_{0})),v(t,x,a;u_{0},v_0))$ is the solution of the system \eqref{eq:main} with the initial condition $(u_0(x,a)),v_0(x,a))$. This also holds for $\textbf{u}_{+}(t)$.

This comparison principle is a direct consequence of Proposition \ref{prof:comparsys} for a super-solution of the system associated with $\textbf{F}^-$ and a sub-solution of the system associated with $\textbf{F}^+$. It is worth noting that one can choose $\textbf{F}^- = \textbf{F}^{l}$ and $\textbf{F}^+ = \textbf{F}^{r}$ as an example. In most cases, the functions $\textbf{F}^-$ and $\textbf{F}^+$ are chosen to be monotone, depending on the objective of each section, so that Proposition \ref{prof:comparsys} can be applied to construct a semi-flow and estimate the solution of system \eqref{eq:main}. 

%
%
%

\section{\bf Existence of disease-free and endemic equilibrium} \label{section:3}

In this section, we study the existence of endemic equilibrium of \eqref{eq:main}. That is, the existence of positive solutions $(u,v)$ to the system \eqref{eq:main1}.

\subsection{\bf Basic reproduction number $\mathcal{R}_0$}

We studies \textit{basic reproduction number} for \eqref{eq:main1} using the approach in \cite{chekroun_global_2020,yang_asymptotical_2023}. First, let us consider a disease-free equation of $(U,0)=(U(t,x,a),0)$ as follows
\begin{align*}
\begin{split}
\left\{\begin{array}{llll}
u_t + u_{a}=d_u\Delta_N u+m(x)u-n(x)u^2-\mu(x,a)u,&t>0,~(x,a)\in\mathcal{O}_A,\\
u(t,x,0) = \displaystyle\int_0^{A} \beta(x,a) u(t,x,a)da,&t>0,~x\in \O,\\
u(0,x,a)=u_0(x,a)&(x,a)\in\mathcal{O}_A.
\end{array}\right.
\end{split}
\end{align*}
Using arguments similar to those in Theorem \ref{theo:2}, there exists a unique $\overline{U}=\overline{U}(t,x,a;u_0)$ for each $u_0\in C(\overline{\mathcal{O}_A})$ $\supp(u_0)\subset \O\times [0,A)$. If $\lambda_{d_u,-m+\mu}>0$, it admits a unique positive steady state $\overline{U}_*=\overline{U}_*(x,a)$ satisfies
\begin{align*}
\begin{split}
\left\{\begin{array}{llll}
w_{a}=d_u\Delta_N w+m(x)u-n(x)w^2-\mu(x,a)u,&t>0,~(x,a)\in\mathcal{O}_A,\\
u(t,x,0) = \displaystyle\int_0^{A} \beta(x,a) u(t,x,a)da,&t>0,~x\in \O.
\end{array}\right.
\end{split}
\end{align*}
where $\lambda_{d_u,-m+\mu}$ is the principal eigenvalue of $\mathcal{A}_{\mu^*}$ for the case $d=d_u$, $\mu^*=-m+\mu$ (see Theorem A.\ref{Lemma A.7}). Then, we linearize the second equation of \eqref{eq:main} around the  disease-free equilibrium
$E^0=(\overline{U}_*,0)$  to obtain
\begin{align}\label{eq:linerize}
\begin{split}
\left\{\begin{array}{ll}
V_t + V_{a}=d_v\Delta_N V+
\left[p(x)-q(x)-\mu(x,a)\right]V,&t>0,\textbf{ }(x,a)\in\mathcal{O}_A,\\
V(t,x,0) = \displaystyle\int_0^{A} \beta(x,a) V(t,x,a)da,&t>0,\textbf{ }x\in\O,\\
V(0,x,a)=V_0(x,a),&(x,a)\in \mathcal{O}_A.
\end{array}\right.
\end{split}
\end{align}

First, for any $0\leq a_0<A $, it is known that the following equation
\begin{align}\label{eq:semi-group}
\left\{\begin{array}{lllll}
W_{a}=d_v\Delta_N W+
\left[p(x)-q(x)-\mu(x,a+a_0)\right]W,&(x,a)\in \mathcal{O}_A,\\
W(0,x)=W_0(x),&x\in \O.
\end{array}\right.
\end{align}
admits a $C_0$-linear semi-group $\Psi(a_0,a):\mathcal{Y}\rightarrow \mathcal{Y}$ such that $W(a;a_0)=\Psi(a_0,a)W_0$ is a solution of \eqref{eq:semi-group} in the sense that $\Psi(a_0,a)W_0\rightarrow W_0$ as $a\rightarrow 0$ for any $W_0\in C(\overline{\O})\setminus\{0\}$ and $W(a)>0$ for any $a\in (0,A-a_0)$ whenever $W_0\geq 0$ (see \cite[Section 5.1.2]{lunardi_analytic_1995}). Furthermore, it is known that
\begin{align*}
W_a\leq d_v\Delta_N W+
\left[K-\mu_{\min}(a+a_0)\right]W
\end{align*}
for some constant $K>0$ large enough. By usual parabolic comparison principle, one has
\begin{align*}
0\leq W(x,a)\leq e^{Ka - \displaystyle\int_0^a \mu_{\min}(k+a_0)dk }T_{d_v \Delta_N}(a)W_0(x)
\end{align*}
where $W_0\geq 0$ and $T_{d_v \Delta_N}$ is the semi-group generated from Neumann-Laplacian operator $-d_v \Delta_N$. Passing to the limit $a\rightarrow A-a_0$,  $W(x,a;a_0)\rightarrow 0$ when $a\rightarrow A-a_0$ uniformly with $x\in\overline{\O}$ thanks to \ref{cond:cond4}. Furthermore, one has
\begin{align*}
||\Psi(a_0,a)W_0||_{\infty}\leq Ce^{NA}||W_0||_{\infty}
\end{align*}
since $||T_{d_v \Delta_N}(a)W_0||_{\infty}\leq C||W_0||_{\infty}$ for some $C>0$ (see \cite[Theorem 5.1.11]{lunardi_analytic_1995}).

Now, using the characteristic line method, we solve the first equation in \eqref{eq:linerize} with the initial condition to obtain
\begin{align*}
V(t,\cdot,a) = \left\{\begin{array}{lllll}
\Psi(a-t,t)V_0(\cdot,a-t),&\text{ if }a-t\geq 0,\\
\Psi(0,a)V(t-a,\cdot,0)&\text{ if }t-a > 0,\\
\end{array}\right.
\end{align*}
Then, we substitute this to the age-structure condition to obtain
\begin{align*}
V(t,x,0)&=\int_0^{\min\{t,A\}} \beta(x,a) \Psi(0,a)V(t-a,\cdot,0)(x)da\\& + \int^A_{\min\{t,A\}} \beta(x,a) \Psi(a-t,t)V_0(\cdot,a-t)(x)da
\end{align*}
From here, following the approach in \cite[Section 4]{chekroun_global_2020} and \cite[Section 4]{yang_asymptotical_2023}, we define the next infected generation operator (NIGO) as follows
\begin{align}\label{for:NIGO}
\mathcal{J}(\varphi)(x):=\int_0^A \beta(x,a) \Psi(a)(\varphi)(x)da
\end{align}
and $\mathcal{R}_0:=r(\mathcal{I})$, where $\Psi(a)(\varphi)=\Psi(0,a)\varphi$, $r(\mathcal{I})$ is the spectral radius of $\mathcal{I}$ on $\mathcal{Y}$. The number $\mathcal{R}_0$ is called \textit{basic reproduction number}. 


\begin{theorem}
Assume \ref{cond:cond1} to \ref{cond:cond5} hold.  Let $\mathcal{J}$ be defined by \eqref{for:NIGO}. Then, the \textit{basic reproduction number} $\mathcal{R}_0 := r(\mathcal{J})$ is a simple positive eigenvalue of operator $\mathcal{I}$, associated with a positive eigenvector $v^0=v^0(x)\in \mathcal{Y}_{+}\setminus \{0\}$. Moreover, there is no other eigenvalue having positive eigenvector.
\end{theorem}

\begin{proof}
One can check that 
\begin{enumerate}
\item $\mathcal{J}:\mathcal{Y}_{+}\rightarrow \mathcal{Y}_{+}$.

\item There exists $M>0$ such that $||\mathcal{J}(\varphi)||_{\infty}\leq M ||\varphi||_{\infty}$ for any $\phi \in \mathcal{Y}$.

\item \text{$\mathcal{J}$ is a compact mapping} thanks to Schauder estimates and \ref{cond:cond5}.

\item $\mathcal{J}(\mathcal{Y}_+\setminus \{0\})\subset \text{Int}(\mathcal{Y}_+)$.

\end{enumerate}
By applying Krein-Rutman theorem (see, for instance, \cite[Theorem 3.2]{amann_fixed_1976}), we obtain the desire results and $\mathcal{R}_0>0$.
\end{proof}

It is known that $\mathcal{J}^*:\mathcal{Y}^* \rightarrow \mathcal{Y}^*$ also satisfies these four conditions (see \cite[Section 2.6 and Theorem 6.4]{brezis_functional_2011}), where $\mathcal{Y}^*$ and $\mathcal{J}^*$ are the dual space of $\mathcal{Y}$ and the dual operator of $\mathcal{J}$, respectively. Then, we apply again the Krein-Rutman theorem $r(\mathcal{J}^*)$ is the principal eigenvalue of $\mathcal{J}^*$. By the standard of spectral theory, $r(\mathcal{J}^*)=r(\mathcal{J})=\mathcal{R}_0$. Let $w^0\in \mathcal{Y}^*$ be the eigenvector of $\mathcal{J}^*$ corresponding to $\mathcal{R}_0$ and $w^0(y)>0~(\text{resp}. \geq 0)$ for any $y\in \mathcal{Y}$ and $y>0~(\text{resp}. \geq 0)$ on $\overline{\O}$.

On the other hand, consider case $\mu^*=-p+q+\mu$. Then, there exists unique principal eigenvalue $\lambda_{d_v,-p+q+\mu}\in \R$ such that
\begin{align}\label{eq:eigenvalue}
\begin{split}
\left\{\begin{array}{llll}
d_v\Delta_N\phi-\phi_{a}+
\left[p(x)-q(x)-\mu(x,a)\right]\phi=\lambda_{d_v,-p+q+\mu}\phi,&(x,a)\in\mathcal{O}_A,\\
\phi(x,0) = \displaystyle\int_0^{A} \beta(x,a) \phi(x,a)da,&x\in \O.
\end{array}\right.
\end{split}
\end{align}
where $\phi=\phi(x,a)$  is the eigenfunction with associated the principal eigenvalue $\lambda_{d_v,-p+q+\mu}$. Note that $\lim\limits_{a\rightarrow A}\phi(x,a)=0$ uniformly with $x\in \overline{\O}$ thanks to the assumption \ref{cond:cond4} (see Appendix \ref{Appen A}). Thus, by defining $\phi(x,A)=0$, it is possible to obtain that $\phi\in C\left(\overline{\mathcal{O}_A}\right)$. Furthermore, it follows from the classical bootstrap arguments that $\phi\in C(\overline{\mathcal{O}_A})\cap C^{2,1}({\mathcal{O}_A})$ (see \cite[Theorem 2.6]{delgado_nonlinear_2008} and \cite[Proposition 1]{ducrot_travelling_2007} for techniques) and $\phi>0$ on $\overline{\O}\times [0,A)$. In addition, for any other eigenvalue $\lambda_0$ of $d_v\Delta_N-\partial_{a}+
\left[p(x)-q(x)-\mu(x,a)\right]I$, one has $Re(\lambda_0)<\lambda_{d_v,-p+q+\mu}$. 

It is known that
\begin{align*}
\int_0^A \beta(\cdot,a) \Psi(a)(v^0) da = \mathcal{R}_0 v^0;\quad \mathcal{J}^*w^0=\mathcal{R}_0 w^0.
\end{align*}
Alternatively, with $u^0(x)=\phi(x,0)$, then $\Psi(a)u^0=\phi(\cdot,a)$. Thus, from \eqref{eq:eigenvalue}, we define 
\begin{align*}
\mathcal{J}_{\lambda_{d_v,-p+q+\mu}}(u_0)(x):=\int_0^A \beta(\cdot,a)e^{-\lambda_{d_v,-p+q+\mu} a} \Psi(a)(u^0) da = u^0.
\end{align*}
Thanks to positivity of $\mathcal{J}$, $\lambda\mapsto \left<w^0,\mathcal{J}_{\lambda}(u^0)\right>$ is decreasing.

Let $\mathcal{R}_0>1$, if $\lambda_{d_v,-p+q+\mu}\leq 0$, then
\begin{align*}
\left<w^0,u^0\right>=\left<w^0,\mathcal{J}_{\lambda_{d_v,-p+q+\mu}}u^0\right>\geq \left<w^0,\mathcal{J}u^0\right>=\left<\mathcal{J}^*w^0,u^0\right>=\mathcal{R}_0\left<w^0,u^0\right>.
\end{align*}
Thus, $\mathcal{R}_0\leq 1$. This is contradiction. Hence, if $\mathcal{R}_0>1$, then $\lambda_{d_v,-p+q+\mu}> 0$.  Similarly, $\mathcal{R}_0\leq 1$ then $\lambda_{d_v,-p+q+\mu}\leq 0$. Let us state the theorem about the relationship between $\lambda_{d_v,-p+q+\mu}$ and $\mathcal{R}_0$.


\begin{theorem}
Assume \ref{cond:cond1} to \ref{cond:cond5} hold.  Let $\mathcal{I}$ be defined by \eqref{for:NIGO}, the \textit{basic reproduction number} $\mathcal{R}_0 = r(\mathcal{I})$ and $\lambda_{d_v,-p+q+\mu}$ is the principal eigenvalue of \eqref{eq:eigenvalue}. Then, $\mathcal{R}_0>1$ if and only if $\lambda_{d_v,-p+q+\mu}>0$. Furthermore, $\mathcal{R}_0=1$ if and only if $\lambda_{d_v,-p+q+\mu}=0$.
\end{theorem}

\begin{remark}
Although the condition $\lambda_{-m+\mu} > 0$ is required to derive equation \eqref{eq:linerize} from the original system \eqref{eq:main}, it is not necessary that the existence of $\mathcal{R}_0$  depends on the sign of the principal eigenvalue $\lambda_{-m+\mu}$.

\end{remark}



\subsection{\bf Existence}

In this section, we adapt the technique in \cite{zhao_spatiotemporal_2023} to obtain the existence of disease-free equilibrium and endemic equilibrium.

First, let us discuss about the disease-free equilibrium.


\begin{proof}[Proof of Theorem \ref{main:theo1} \ref{DFE}]
 The existence and uniqueness of \eqref{eq:nodisease} have been proven in Theorem A.\ref{Lemma A.7} with the help of $\lambda_{d_u,-m+\mu}> 0$. Thanks to the classical bootstrap argument, one has
\begin{align*}
u_1\in C(\overline{\mathcal{O}_A})\cap C^{2,1}(\mathcal{O}_A)
\end{align*}

Since $\mathcal{R}_0\leq 1$, it is known that $\lambda_{d_v,-p+q+\mu}\leq 0$. Next, assume by the contraction, there exists $(u,v)$ is a solution of \eqref{eq:main1} such that $u>0$ and $v>0$. Furthermore, we have
\begin{align*}
u,~v\in C(\overline{\mathcal{O}_A})\cap C^{2,1}(\mathcal{O}_A).
\end{align*}
Let $\phi>0$ be the eigenfunction associated with $\lambda_{d_v,-p+q+\mu}$, recall \eqref{eq:eigenvalue}
\begin{align*}
\begin{split}
\left\{\begin{array}{llll}
\phi_{a}-d_v\Delta_N\phi-
[p(x)-q(x)-\mu(x,a)]\phi=-\lambda_{-p+q+\mu}\phi,&(x,a)\in\mathcal{O}_A,\\
\phi(x,0) = \displaystyle\int_0^{A} \beta(x,a) \phi(x,a)da,&x\in \O.
\end{array}\right.
\end{split}
\end{align*}
Thus, one gets
\begin{align*}
\phi_{a}-d_v\Delta_N\phi-
[p(x)-q(x)-\mu(x,a)]\phi \geq 0
\end{align*}
Since $\dfrac{u\phi}{u+\phi}\leq \phi$, we get that
\begin{align*}
\phi_{a}-d_v\Delta_N\phi-
\left[p(x)\dfrac{u}{u+\phi}-q(x)-\mu(x,a)\right]\phi \geq 0
\end{align*}
On the other hand, we have
\begin{align*}
\begin{split}
\left\{\begin{array}{llll}
 v_{a}-d_v\Delta_Nv-
\left[\dfrac{p(x)u}{u+v}-q(x)-\mu(x,a)\right]v=0,&(x,a)\in\mathcal{O}_A,\\
v(x,0) = \displaystyle\int_0^{A} \beta(x,a) v(x,a)da,&x\in \O.
\end{array}\right.
\end{split}
\end{align*}
By applying comparison principle in Theorem A.\ref{Lemma:A.10}-Step 2, we obtain that
\begin{align*}
v\leq \phi
\end{align*}
Replace $\phi$ with $c_0\phi$ for some $c_0>0$, one still has
\begin{align*}
v\leq c_0\phi
\end{align*}
Letting $c_0 \rightarrow 0$, $v\leq 0$. Contradiction. Thus, there is no positive solution of \eqref{eq:main1}. Furthermore, in this case, it is not possible to obtain a non-negative solution $v\geq 0$ of \eqref{eq:main1} in a similar manner. Hence, the solution $(u_1,0)$ is the only  non-trivial solution of \eqref{eq:main1}. This completes the proof.

%
%


\end{proof}
This theorem shows that the disease may not spread in the case when $\mathcal{R}_0\leq 1$ or the infected population $v=0$. 

In the next theorem, we study the existence of positive endemic steady state in the case the basic reproduction number $\mathcal{R}_0>1$.


\begin{proof}[Proof of Theorem \ref{main:theo1} \ref{EE}]


Since $\mu^*\mapsto \lambda_{d,\mu^*}$ is decreasing (see Section \ref{sec 2.1}) and $p\geq 0$, one has $\lambda_{d_u,-m+\mu}>\lambda_{d_u,-m+p+\mu}>0 $.  Furthermore, since $\mathcal{R}_0>1$, one has $\lambda_{d_v,-p+q+\mu}>0$.

Now, inspired by the technique in \cite[Theorem 2.6]{zhao_spatiotemporal_2023}  the strategy is to use the homotopy invariance of Leray-Schauder topological degree to find the positive solution of \eqref{eq:main1}. First, let us define
\begin{align*}
f^l(x,u)=[m(x)-p(x)-n(x)u]u;\textbf{ }\quad f^r(x,u,v)=[m(x)-n(x)u]u+q(x)v,
\end{align*}
and
\begin{align*}
f^c(x,u,v)=\left[m(x)-n(x)u-\dfrac{p(x)v}{u+v}\right]u+q(x)v;~ g(x,u,v)=\left[\dfrac{p(x)u}{u+v}-q(x)\right]v.
\end{align*}
for the future purposes. Based on the arguments in \cite[Section 1.1]{allen_asymptotic_2008}, it is known that $f^l$, $f^c$, $f^r$ and $g$ are local Lipchitz functions.

Since $0\leq \dfrac{v}{u+v} \leq 1$, it is easy to see that
\begin{align*}
f^l(x,u)\leq f^c(x,u,v)\leq f^r(x,u,v)
\end{align*}
for any $x\in \O$ and $u\geq 0$, $v\geq 0 $ such that $u+v\neq 0$. Now, let us consider the following auxiliary system
\begin{align}\label{eq:L}
\begin{split}
\left\{\begin{array}{l}
 u_{a}=d_u\Delta_Nu+f^l(x,u)-\mu(x,a)u,\\
 v_{a}=d_v\Delta_Nv+g(x,u,v)-\mu(x,a)v,\\
u(x,0) = \displaystyle\int_0^{A} \beta(x,a) u(x,a)da,\\
v(x,0) = \displaystyle\int_0^{A} \beta(x,a) v(x,a)da.
\end{array}\right.
\end{split}
\end{align}
Thanks to Theorem A.\ref{Lemma A.7} ($k=m-p$) and the fact that $\lambda_{d_u,-m+p+\mu}>0$, the first and third equations admit a unique positive solution $u_l>0$ on $\overline{\O}\times [0,A)$ and $u_1(x,A)=0$ for any $x\in \overline{\O}$ (see Remark \ref{lim0} in Appendix). 

Next, let us consider the existence of solution $v_l$ to
\begin{align}\label{eq:temp}
\begin{split}
\left\{\begin{array}{l}
 v_{a}=d_v\Delta_Nv+vh(x,u_l,v)-\mu(x,a)v,\\
v(x,0) = \displaystyle\int_0^{A} \beta(x,a) v(x,a)da.
\end{array}\right.
\end{split}
\end{align}
where $h(x,u_l(x,a),v)=\dfrac{p(x)u_l(x,a)}{u_l(x,a)+v}-q(x)$. Then, since $\lambda_{d_v,-p+q+\mu}>0$, \eqref{eq:temp} admits a unique positive solution $v_l>0$ on $\overline{\O}\times [0,A)$ by thanks to Theorem A.\ref{Lemma:A.10} (with $\lambda_{d_v,-p+q+\mu}^*=-\lambda_{d_v,-p+q+\mu}<0$). Thus, the system  \eqref{eq:L} has a unique positive solution $(u_l,v_l)$. The uniqueness is obvious since all the solutions of \eqref{eq:temp} must satisfy the first and third one whose uniqueness are guaranteed. In addition, by standard parabolic regularity and bootstrap arguments, similar to the one in \cite[Proposition 1]{ducrot_travelling_2007}, one can check that
\begin{align*}
u^l,v^l\in C(\overline{\mathcal{O}_A})\cap C^{2,1}(\mathcal{O}_A).
\end{align*} 

Next, we construct sub- and super- solutions to form the domain for the topological degree. For $0<\epsilon<1$, we define $\underline{u}_l=\epsilon u_l$. Then, one has $0<\underline{u}_l<u_l$ on $\overline{\O}\times (0,A)$, $\underline{u}_l(x,A)=0$ for any $x\in \overline{\O}$ and
\begin{align*}
f^l(x,a,\underline{u}_l)-\mu(x,a)\epsilon u_l&=[m(x)-p(x)-n(x)\epsilon u_l-\mu(x,a)]\epsilon u_l\\
&\geq \epsilon[m(x)-p(x)-n(x) u_l-\mu(x,a)] u_l\\
&=(\underline{u}_l)_a-d_u\Delta_N\underline{u}_l.
\end{align*}
The age-structured condition is obvious. Thus, a similar process to that used for $v_l$ is applied for $u=\underline{u}_l$ to obtain that $\underline{v}_l$ is a positive solution of the following equation
\begin{align}\label{eq:temp1}
\begin{split}
\left\{\begin{array}{l}
 v_{a}=d_v\Delta_Nv+g(x,a,\underline{u}_l,v)-\mu(x,a) v,\\
v(x,0) = \displaystyle\int_0^{A} \beta(x,a) v(x,a)da.
\end{array}\right.
\end{split}
\end{align}
A simple calculation yields that $g(x,a,u_l,v_l)\geq g(x,a,\underline{u}_l,v_l)$. Therefore, thanks to Lemma A.\ref{comparison}, $v_l\geq \underline{v}_l$ on $\overline{\O}\times [0,A)$. Let $0<\epsilon<1$. We can choose $\epsilon \underline{v_l}$ instead of $\underline{v_l}$ so that $v_l\geq \underline{v_l}>\epsilon \underline{v_l}$ on $\overline{\O}\times [0,A)$ and $\underline{v}_l$ is still a sub-solution. Thus, we may assume $v_l>\underline{v}_l$ on $\overline{\O}\times [0,A)$. The pair $(\underline{u}_l,\underline{v}_l)$ will be used as the sub-solution.


Next, we construct the super-solution. We recall $\overline{u}_r:=MY_{\alpha}$, $\overline{v}_r := 2KMY_{\alpha}$ from Lemma \ref{Lemma:super}. It then follows that $(\overline{u}_r,\overline{v}_r)$ is super-solution of the system \eqref{eq:main1}. By Lemma A.\ref{comparison}, $u_l<\overline{u}_r$ and $v_l<\overline{v}_r$ on $\overline{\O}\times [0,A)$. 




Now, let us define the domain and the homotopy function for the topological degree as follows 
\begin{align*}
\mathcal{C}&:=\left\{(u,v)\in \mathcal{X}_{0}\times\mathcal{X}_{0}:
\begin{matrix}\underline{u}_l < u< \overline{u}_r,~\underline{v}_l < v< \overline{v}_r \text{ on }\overline{\mathcal{O}_{A}}
\end{matrix}
\right\}
\end{align*}
and $\mathcal{F}:[0,1]\times \mathcal{C}\rightarrow \mathcal{X}_{0}\times \mathcal{X}_{0}$
\begin{align*}
\mathcal{F}(t)(u,v)=\left(\begin{matrix}
(\mathcal{B}_{d_u,\mu}+\lambda I)^{-1}[\lambda u+(1-t)f^l(x,u)+tf^c(x,u,v)]\\
(\mathcal{B}_{d_v,\mu}+\lambda I)^{-1}[\lambda v+g(x,u,v)]
\end{matrix}\right).
\end{align*}
for some $\lambda>0$ large enough. Clearly, $\mathcal{C}$ is an open set of $\mathcal{X}_{0}\times\mathcal{X}_{0}$, $(u_l,v_l)\in \mathcal{C}$ and $\mathcal{F}$ is continuous. Using similar arguments to the one in \cite[Theorem 1]{guo_on_1994}, one can check that $\mathcal{F}(t)$ is compact for any $t\in [0,1]$. Because of the definitions of $(\underline{u}_l,\underline{v}_l)$ and $(\overline{u}_r,\overline{v}_r)$, it is known that $\mathcal{F}(t)(u,v)\neq (u,v)$ on $\partial \mathcal{C}$ for any $t\in [0,1]$. Thus, the topological degree $\deg(I-\mathcal{F}(t),\mathcal{C},0)$ is well-defined and is independent of $t$. 

Next, let us calculate the value of $\deg(I-\mathcal{F}(0),\mathcal{C},0)$. Define $\mathcal{T}_{\lambda,d,\mu}=(\mathcal{B}_{d,\mu}+\lambda I)^{-1}$ and $A:=\left[I-D\mathcal{F}(0)\right](u_l,v_l)$. Direct calculations yield that
\begin{align*}
A=\left(\begin{matrix}
\mathcal{T}_{\lambda,d_u,\mu}\left[\mathcal{B}_{d_u,\mu}+ c_{11}(x)I\right]&0\\
\mathcal{T}_{\lambda,d_v,\mu}\left(\mathcal{B}_{d_v,\mu}-p(x)\dfrac{v_l^2}{(u_l+v_l)^2} I\right)&\mathcal{T}_{\lambda,d_v,\mu}\left[\mathcal{B}_{d_v,\mu}+c_{22}(x)I\right].
\end{matrix}\right).
\end{align*}
where $D\mathcal{F}(0)$ is the derivative of $\mathcal{F}(0)$ and $c_{11}(u)(x)=2n(x)u_l-m(x)-p(x)$, $c_{22}(x)=q(x)-p(x)\dfrac{u^2_l}{(u_l+v_l)^2} $. Put  
$\mathcal{L}=[I-D\mathcal{F}(0)](u_l,v_l)$. Similar to the one in \cite[Theorem 1]{guo_on_1994}, $\mathcal{B}_{d_u}$ and $\mathcal{B}_{d_v}$ have compact resolvents. It follows from \cite[Theorem 6.6]{brezis_functional_2011} that $\mathcal{L}$ is a Fredholm operator of index zero. 

On one hand, $(u_l,v_l)$ is the unique solution of \eqref{eq:L}, making it an isolated solution of $I-\mathcal{F}(0)$. On the other hand, we can check that
\begin{align*}
\sigma(\mathcal{L})\setminus\{\infty\}=\sigma_p(\mathcal{L})=\sigma_p(\mathcal{B}_{d_u,\mu}+c_{11}(x)I)\cup\sigma_p(\mathcal{B}_{d_v,\mu}+c_{22}(x)I). 
\end{align*}
where $\sigma_p(\mathcal{L})$ is the point spectrum (see \cite[Chapter 7]{kreyszig_introductory_1978}). Let $\lambda_1$ be the principal eigenvalue of $\mathcal{B}_{d_u,\mu}+c_{11}(x)I$ and $\lambda_2$ be the principal eigenvalue of $\mathcal{B}_{d_v,\mu}+c_{22}(x)I$. It follows from 
\begin{align*}
\mathcal{B}_{d_u,\mu}v_l+\left[q(x)-\dfrac{p(x)u_l}{u_l+v_l} \right]v_l=0
\end{align*}
that $0$ is the principal eigenvalue of $\mathcal{B}_{d_u,\mu}+\left[q(x)-\dfrac{p(x)u_l}{u_l+v_l} \right]I$. Thus, by arguments similar to those in \cite[Theorem 8]{delgado_nonlinear_2006} and the fact that $c_{22}(x)\geq \left[q(x)-\dfrac{p(x)u_l}{u_l+v_l} \right]$, we have that $\lambda_2>0$. Similarly, $\lambda_1>0$. As the results,
\begin{align*}
\mathfrak{s}\left(\mathcal{L}):=\inf\{Re(\lambda):\lambda\in \sigma(\mathcal{L})\right\}>0
\end{align*}
meaning $\left[I-D\mathcal{F}(0)\right](u_l,v_l)$ is invertible. Consequently, by applying the  topological degree theory in  Nirenberg \cite[Theorem 2.8.1]{nirenberg_nonlinear_2001}, one gets either $\deg(I-\mathcal{F}(0),\mathcal{C},0)=1$ or $=-1$. Using homotopy invariance, one has either $\deg(I-\mathcal{F}(1),\mathcal{C},0)=1$ or $-1$, meaning there exists $(u,v)\in \mathcal{C}$ such that $[I-\mathcal{F}(1)](u,v)=0$. As the results, $(u,v)$ is the solution of the endemic equilibrium \eqref{eq:main1} and $u>0$ and $v>0$ on $\overline{\O}\times [0,A)$. By standard parabolic regularity and bootstrap arguments, similar to the one in \cite[Proposition 1]{ducrot_travelling_2007}, we can check that
\begin{align*}
u,~v\in C(\overline{\mathcal{O}_{A}})\cap C^{2,1}(\mathcal{O}_{A}).
\end{align*} 
and $u(x,A)=v(x,A)=0$ for any $x\in \overline{\O}$. This completes the proof.


\end{proof}

\section{\bf Stability of disease-free and endemic equilibrium}\label{section:4}



Since the monotone semi-flow property of the system \eqref{eq:main} is not obtained, let us proceed by recalling the upper-system as follows
\begin{align}\label{eq:uppermain}
\left\{\begin{array}{lll}
u_t = \mathcal{A}_{d_u,\mu}u + F^r(u,v),&t>0,~(x,a)\in \mathcal{O}_A,\\
v_t = \mathcal{A}_{d_v,\mu}v + G(u,v),&t>0,~(x,a)\in \mathcal{O}_A,\\
u(0,x,a)=u_0(x,a),&(x,a)\in \mathcal{O}_A,\\
v(0,x,a)=v_0(x,a),&(x,a)\in \mathcal{O}_A.
\end{array}\right.
\end{align}
Thanks to the proof in Theorem \ref{theo:2},  this system admits a unique positive  global-in-time solution 
\begin{align*}
(u^r,v^r)=(u^r(t;u_0,v_0),v^r(t;u_0,v_0))=(u^r(t,x,a;u_0,v_0),v^r(t,x,a;u_0,v_0)).
\end{align*}
Furthermore, by similar calculation with \eqref{boundedblow}, the age-structured comparison principle holds for these solution based on the work in \cite[Theorem 4.3 and Theorem 4.5]{magal_monotone_2019}. We also define its steady state as follows
\begin{align}\label{eq:uppermain1}
\left\{\begin{array}{lll}
\mathcal{A}_{d_u,\mu}u + F^r(u,v)=0,&~(x,a)\in \mathcal{O}_A,\\
 \mathcal{A}_{d_v,\mu}v + G(u,v)=0,&~(x,a)\in \mathcal{O}_A.
\end{array}\right.
\end{align}
By similar arguments with the one in Theorem \ref{main:theo1} \ref{DFE}, one can prove that $v=0$ and $u=u_1$.

Next, consider $(u_0, v_0) \in \mathcal{X}_+\times \mathcal{X}_+$ with  $\supp(u_0)\cap \mathcal{O}_{A_2}\neq \emptyset$ and  $\supp(v_0)\cap \mathcal{O}_{A_2}\neq \emptyset$, $\supp({u_0}),~\supp({v_0})\subset \overline{\O}\times [0,A)$. It is known that $u_0,~v_0$ satisfying \eqref{bounded:above}, $u_0,~v_0\in  \mathcal{B}_{\pi}(\mathcal{O}_A)$ and, thus, 
$(u,v)$ is bounded in the sense that
\begin{align*}
||u^r(t;u_0,v_0)||_{L^{\infty}(\mathcal{O}_{A})}+||v^r(t;u_0,v_0)||_{L^{\infty}(\mathcal{O}_{A})}\leq 2N,~\forall t>0.
\end{align*} 
where $N>0$ is the constant in Theorem \ref{theo:2}. Hence, 
\begin{align}\label{pro:bounded}
||u^r(t;u_0,v_0)||_{L^{2}(\mathcal{O}_{A})}+||v^r(t;u_0,v_0)||_{L^{2}(\mathcal{O}_{A})}\leq H,~\forall t>0,
\end{align}
for some $H>0$. Furthermore, there exists $T^*_0>A$ such that
\begin{align*}
0<u(t)\leq u^r(t),~0<v(t)\leq v^r(t),~\forall t\geq T^*_0 \text{ on } \overline{\O}\times [0,A).
\end{align*}

Next, we define the $\omega$-limit set depending on the initial condition $(u_0,v_0)$
\begin{align*}
\omega(u_0,v_0)=&\{(w,k)\in\mathcal{E}\times \mathcal{E}:~\exists t_n\rightarrow \infty \text{ as }n \rightarrow \infty \text{ such that }\\
&u^r(t_n;u_0,v_0)\rightarrow w,~v^r(t_n;u_0,v_0)\rightarrow k \text{ weakly in }\mathcal{E}\}
\end{align*}
It is well-known that the $\omega$-limit set $\omega_{A}(u_0,v_0)$ is non-empty thanks to \eqref{pro:bounded}.


\subsection{\bf Stability of the disease-free case $\mathcal{R}_0<1$}

In this section, we study the stability of solution in the disease-free case $\mathcal{R}_0<1$. We expect the infected population $v(t)$ will tends to zero in this case. By adapting the technique in \cite[Theorem 3.3]{zhao_spatiotemporal_2023}, the following lemma shows this statements clearly. 

\begin{lemma}\label{lemma:conver0}
 Suppose that $(u_0, v_0) \in \mathcal{X}_+\times \mathcal{X}_+$ with $\supp(u_0),~\supp(v_0)\subset \O\times [0,A)$,  $u_0,~v_0\not\equiv 0$, $\lambda_{d_u,-m+\mu}>0$ and  $\mathcal{R}_0<1$. Then, $||v^r(t;u_0,v_0)||_{L^{\infty}(\mathcal{O}_A)}\rightarrow 0$ as $t\rightarrow \infty$. 
\end{lemma}

\begin{proof}
It follows from $\mathcal{R}_0<1$ that $\lambda_{d_v,-p+q+\mu}<0$. Also, it is known that $u_0,~v_0\in B_{\pi}(\mathcal{O}_A)$. Recall the solution $(u,v)$ of \eqref{eq:uppermain}  and $u(t),~v(t)>0$ for $t>0$ on $\mathcal{O}_A$.

Let $\phi$ be the positive eigenfunction associated with the principal eigenvalue $\lambda_{d_v,-p+q+\mu}$. Define 
\begin{align*}
\psi(t,x,a) =C_1e^{\lambda_{d_v,-p+q+\mu} t}\phi(x,a)>0,~\forall (x,a)\in \mathcal{O}_A,
\end{align*}
for some $C_1>0$ will be chosen later. Thus, direct calculations yields that
\begin{align*}
\psi_t+\psi_a= d_v\Delta_N\psi+
[p(x)-q(x)-\mu(x,a)]\psi
\end{align*}
Thanks to the fact that $\dfrac{uv}{u+v}\leq v$ for any $u,~v\geq 0$, one has
\begin{align*}
v_t^r+v_a^r\leq  d_v\Delta_N v^r+
\left[p(x)-q(x)-\mu(x,a)\right]v^r\ \text{ on }\mathcal{O}_A
\end{align*}
Furthermore, there exists $C_1>0$ large enough so that 
\begin{align*}
\psi(0,x,a)=C_1\phi(x,a)\geq v_0(x,a)
\end{align*}
for any $(x,a)\in \mathcal{O}_A$. By the comparison principle in Proposition \ref{prof:comparoneeq} for $\psi$ and $v$, one has
\begin{align*}
v^r(t,x,a)\leq \psi(t,x,a)=C_1e^{\lambda_{d_v,-p+q+\mu} t}\phi(x,a)
\end{align*}
Since $\lambda_{d_v,-p+q+\mu}<0$, we get
\begin{align*}
||v^r(t;u_0,v_0)||_{L^{\infty}(\mathcal{O}_A)}\rightarrow 0.
\end{align*}
This concludes the proof for $v$.

\end{proof}
\begin{remark}
The proof still holds if we assume there exists $K>0$ such that
\begin{align*}
v_0(x,a)\leq K\phi(x,a),~\forall (x,a)\in \mathcal{O}_A
\end{align*}
where $\phi$ is the eigenfunction associated with the eigenvalue $\lambda_{d_v,-p+q+\mu}$. 
The $\omega$-limit set can be written as follows
\begin{align*}
\omega(u_0,v_0)=\{(w,0)\}
\end{align*}

\end{remark}

%


\begin{lemma}\label{omegasetlimit}
 Suppose that all assumptions of Lemma \ref{lemma:conver0} are satisfied. Assume further that all convergences in the $\omega$-limit set $\omega(u_0,v_0)$ are strong in $L^2(\mathcal{O}_{A})$ (up to a sub-sequence). Then, any $(w,0)$ in $\omega(u_0,v_0)$ is an solution of \eqref{eq:uppermain1} or a disease-free equilibrium.
\end{lemma}

\begin{proof}
Inspired by \cite[Lemma 5.7]{langlais_large_1988}, we consider $\chi \in C^2(\overline{\mathcal{O}_{A}})$, $\supp(\chi)\subset \O \times [0,A)$ and $\rho \in C^{\infty}(\R)$ such that
\begin{align*}
\rho(s)\geq 0;\quad \supp(\rho)\subset [-1,1];\quad \int_{-1}^{1}\rho(s)ds=1
\end{align*}

Let $w\in \omega(u_0,v_0)$ and $t_n\geq 2$ be the sequence such that
\begin{align*}
u^r(t_n,\cdot,\cdot) \rightarrow w \text{ weakly in }L^2(\mathcal{O}_{A}).
\end{align*}
Multiplying the first equation of \eqref{eq:uppermain} by $\rho(t-t_n)\chi(x,a)$ and integrating over $[0,\infty)\times \mathcal{O}_{A}$, one has
\begin{align*}
&\int_{(t_n-1,t_n+1)\times \mathcal{O}_{A}}-\left[\chi\rho'(t-t_n)u^r-(\chi_a+d_u\Delta \chi -\mu(x,a)\chi)u^r\rho(t-t_n)\right]dtdxda\\
&=	\int_{(t_n-1,t_n+1)\times \O}\chi(x,0)\int_0^A\beta(x,a)u^r(t,x,a)\rho(t-t_n)dadtdx\\
&+\int_{(t_n-1,t_n+1)\times \mathcal{O}_{A }} F^r(u^r,v^r)\chi \rho(t-t_n)dtdxda
\end{align*}
Put $s=t-t_n$, by changing variable, we have that
\begin{align}\label{eq:R0<1}
\begin{split}
&\int_{(-1,1)\times \mathcal{O}_{A }}\left[-\chi\rho'(s)-(\chi_a+d_u\Delta \chi -\mu(x,a)\chi)\rho(s)\right]u^r(t_n+s,x,a)dsdxda\\
&=	\int_{(-1,1)\times \O}\chi(x,0)\int_0^A\beta(x,a)u^r(t_n+s,x,a)\rho(s)dadsdx\\
&+\int_{(-1,1)\times \mathcal{O}_{A }} F^r(u^r(t_n+s),v^r(t_n+s))\chi \rho(s)dsdxda.
\end{split}
\end{align}

\textbf{Claim}: For any $f\in L^{\infty}(\mathcal{O}_A)$, one has, as $n\rightarrow \infty$,
\begin{align*}
&\displaystyle \int_{(-1,1)\times\mathcal{O}_{A }}u^r(t_n+s,x,a)f(x,a)\rho(s)dsdxda\\&\rightarrow \int_{(-1,1)\times\mathcal{O}_{A }}w(x,a)f(x,a)\rho(s)dsdxda 
\end{align*}
\begin{proof}[Proof for the claim]
 Since $u(t_n,\cdot,\cdot) \rightarrow w \text{ weakly in }L^2(\mathcal{O}_{A })$, one gets 
\begin{align*}
\displaystyle \int_{\mathcal{O}_{A }}u^r(t_n+s,x,a)f(x,a)dxda\rightarrow \displaystyle \int_{\mathcal{O}_{A }}w(x,a)f(x,a)dxda
\end{align*}
and $||u^r(t_n+s,\cdot,\cdot)||_{L^{\infty}(\mathcal{O}_{A })}\leq N$, where $N>0$ is the constant in Theorem \ref{theo:2}. Thus, the mapping $s\mapsto \displaystyle \int_{\mathcal{O}_{A }}u^r(t_n+s,x,a)f(x,a)dxda$ is bounded for any $n\in \mathbb{N}$, $s\in [-1,1]$. By the Lebesgue dominated theorem, we obtain the desired result.
\end{proof}

Now, let us focus on 
\begin{align*}
&\int_{(-1,1)\times \mathcal{O}_{A }} F^r(u^r(t_n+s),v^r(t_n+s))\chi \rho(s)dsdxda
\\&=\int_{(-1,1)\times \mathcal{O}_{A }}\left[m(x)-n(x)u^r\right]u^r\chi \rho(s)dsdxda+\int_{(-1,1)\times \mathcal{O}_{A }}q(x)v^r\chi \rho(s)dsdxda\\
&=I+II+III
\end{align*}

First, it is easy to check that $|III|\leq  C_2e^{\lambda_{d_v,-p+q+\mu} t}$
for some $C_2>0$. Thus, one has $III\rightarrow 0$ as $n\rightarrow \infty$.
On the other hand, one has, as $n\rightarrow \infty$,
\begin{align*}
I &= \int_{(-1,1)\times \mathcal{O}_{A }}m(x)u^r(t_n+s,x,a)\chi \rho(s)dsdxda\\
&\rightarrow \int_{(-1,1)\times \mathcal{O}_{A }}m(x)w(x,a)\chi \rho(s)dsdxda
\end{align*}
since $m\rho \in L^{\infty}(\mathcal{O}_{A })$. Lastly,
\begin{align*}
II = \int_{(-1,1)\times \mathcal{O}_{A }}n(x)(u^r)^2(t_n+s,x,a)\chi \rho(s)dsdxda
\end{align*}
One can check that $t\mapsto ||(u^r)^2(t,\cdot,\cdot)||_{L^2(\mathcal{O}_{A })}$ is bounded. Thus, we can assume that (up to a sub-sequence)
\begin{align*}
(u^r)^2(t_n,\cdot,,\cdot)\rightarrow h \text{ weakly in }L^2(\mathcal{O}_{A }),
\end{align*}
for some $h\in L^2(\mathcal{O}_A)$. As the results, by similar arguments to the above claim, we get that
\begin{align*}
II \rightarrow \int_{(-1,1)\times \mathcal{O}_{A }}n(x)h(x,a)\chi \rho(s)dsdxda \text{ as } n\rightarrow \infty. 
\end{align*}
Passing to the limit $n\rightarrow \infty$, from \eqref{eq:R0<1}, one has
\begin{align*}
\begin{split}
&\int_{(-1,1)\times \mathcal{O}_{A}}\left[-\chi\rho'(s)-(\chi_a+d_u\Delta \chi -\mu(x,a)\chi)\rho(s)\right]w(x,a)dsdxda\\
&=	\int_{(-1,1)\times \O}\chi(x,0)\int_0^A\beta(x,a)w(x,a)\rho(s)dadsdx\\
&+\int_{(-1,1)\times \mathcal{O}_{A }}m(x)w(x,a)\chi \rho(s)dsdxda\\
&-\int_{(-1,1)\times \mathcal{O}_{A }}n(x)h(x,a)\chi \rho(s)dsdxda
\end{split}
\end{align*}
Thanks to the facts that $\displaystyle \int_{-1}^1\rho'(s)=0$ 
and $\displaystyle \int_{-1}^1\rho(s)=1$, we obtain
\begin{align}\label{eq:very-weak}
\begin{split}
&-\int_{\mathcal{O}_{A }}[\chi_a+d_u\Delta \chi -\mu(x,a)\chi)]w(x,a)dxda\\
&=	\int_{\O}\chi(x,0)\int_0^A\beta(x,a)w(x,a)dadx+\int_{ \mathcal{O}_{A }}m(x)w(x,a)\chi dxda\\
&-\int_{\mathcal{O}_{A }}n(x)h(x,a)\chi dxda
\end{split}
\end{align}
%


Next, since $u(t_n,\cdot,\cdot)$ converges strongly in $L^2(\mathcal{O}_{A })$ (up to a sub-sequence), it follows that
\begin{align*}
h=w^2
\end{align*}
Thus, the equation \eqref{eq:very-weak} can be written as follows
\begin{align}\label{eq:very-weak1}
\begin{split}
&-\int_{\mathcal{O}_{A }}[\chi_a+d_u\Delta \chi -\mu(x,a)\chi)]w(x,a)dxda\\
&=	\int_{\O}\chi(x,0)\int_0^A\beta(x,a)w(x,a)dadx+\int_{ \mathcal{O}_{A }}m(x)w(x,a)\chi dxda
\\
&-\int_{\mathcal{O}_{A }}n(x)w^2(x,a)\chi dxda.
\end{split}
\end{align}
Since $\chi$ is arbitrary, $w\geq 0$ is a very weak solution of \eqref{eq:nodisease}. By standard parabolic arguments and the age-structured initial condition, we can obtain $w$ is indeed a strong solution of \eqref{eq:nodisease}. Hence,
\begin{align*}
\omega_{A }(u_0,v_0)=\{(u_1,0)\}
\end{align*}

In addition, it is standard to obtain that
\begin{align*}
u^r(t,\cdot,\cdot)\rightarrow u_1 \text{ strongly in } L^2(\mathcal{O}_{A }) \text{ as }t\rightarrow \infty. 
\end{align*}
This completes the proof.

\end{proof}

\begin{remark}
The $L^2$-strong convergence condition is essential. To obtain the KPP-type term in the final result and thereby achieve the disease-free equilibrium, we require the $L^2$-weak convergence of the $"u^2"$ term. However, it is well-known that this weak convergence implies $L^2$-strong convergence of $u$.
\end{remark}


Next, inspired by the work of \cite[Section 7]{coville_simple_2010}, \cite[Theorem 4.11]{ducrot_age-structured_2024} and \cite[Proposition 5.3]{kang_effects_2022}, we obtain the following result

\begin{theorem}\label{Theo:10}
Assume \ref{cond:cond1} to \ref{cond:cond5} hold. Suppose further that $(u_0, v_0) \in \mathcal{X}_+\times \mathcal{X}_+$ with  $u_0,~v_0\not\equiv 0$, $\supp(u_0),~\supp(v_0)\subset \O\times [0,A)$, $\lambda_{d_u,-m+\mu}>0$ and  $\mathcal{R}_0<1$. Then, for any $A_2<A_3<A$, one has 
\begin{align*}
||u^r(t;u_0,v_0)-u_1||_{L^{\infty}(\overline{\mathcal{O}_{A_3}})}\rightarrow 0,~||v^r(t;u_0,v_0)||_{L^{\infty}(\overline{\mathcal{O}_{A_3}})}\rightarrow 0.
\end{align*}
\end{theorem}

\begin{proof}
Let $A_2<A_3<A$. Since $u_0\geq 0$ and $v_0\geq 0$, one has
\begin{align*}
u^r(T^*_0,x,a;u_0,v_0)>0,~v^r(T^*_0,x,a;u_0,v_0)>0,~\forall (x,a)\in \overline{\mathcal{O}_{A_3}}.
\end{align*}
Thanks to the characteristic line formula, there exists $\delta>0$ such that
\begin{align*}
u^r(T^*_0,x,a;u_0,v_0)>\delta,~v^r(T^*_0,x,a;u_0,v_0)>\delta,~\forall (x,a)\in \overline{\mathcal{O}_{A_3}}.
\end{align*}

Consider $\phi>0$ is the eigenfunction associated with $\lambda_{d_u,-m+\mu}>0$ 
\begin{align*}
\begin{split}
\left\{\begin{array}{l}
\phi_a-d_u\Delta_N\phi+\mu(x,a)\phi-m(x)\phi=-\lambda_{d_u,-m+\mu}\phi,\\
\phi(x,0) = \displaystyle\int_0^{A_3} \beta(x,a) \phi(x,a)da.
\end{array}\right.
\end{split}
\end{align*}
and $\varphi>0$ be the eigenfunction associated with $\lambda_{d_v,-p+q+\mu}<0$
\begin{align*}
\begin{split}
\left\{\begin{array}{l}
\varphi_a-d_v\Delta_N\varphi+\mu(x,a)\varphi-p(x)\varphi+q(x)\varphi=-\lambda_{d_v,-p+q+\mu}\varphi,\\
\varphi(x,0) = \displaystyle\int_0^{A_3} \beta(x,a) \varphi(x,a)da.
\end{array}\right.
\end{split}
\end{align*}
Define $\underline{u}:=\epsilon \phi$ and $\underline{v}:=0$ for some $\epsilon>0$. Choose $\epsilon$ small enough so that
\begin{align*}
\underline{u}_a-d_u\Delta_N\underline{u}+\mu(x,a)\underline{u}-m(x)\underline{u}&\leq -n(x)\underline{u}^2\leq -n(x)\underline{u}^2+q(x)\underline{v},
\end{align*}
and
\begin{align*}
\underline{v}_a-d_u\Delta_N\underline{v}+\mu(x,a)\underline{v}\leq -p(x)\dfrac{\underline{u}\underline{v}}{\underline{u}+\underline{v}}+q(x)\underline{v}.
\end{align*}
Thus, $(\underline{u},\underline{v})$ is a sub-solution of \eqref{eq:uppermain}. 

On the other hand, we choose $\alpha$ large enough so that the upper-bound age-structured condition of $Y_{\alpha}$ is satisfied. Put $\overline{u}:=K_1Y_{\alpha}(a)$, $\overline{v}:=K_2\varphi(x,a)$ for some $K_1,~K_2$ large enough. One can show that $\overline{v}$ satisfies
\begin{align*}
w_{a}&=d_v\Delta_Nw+
[p(x)-q(x)-\mu(x,a)]w-\lambda_{d_v,-p+q+\mu}w\\
&\geq d_v\Delta_Nw+\left[p(x)\dfrac{\overline{u}}{\overline{u}+w}-q(x)-\mu(x,a)\right]w
\end{align*} 
since $\lambda_{d_v,-p+q+\mu}<0$. It follows from arguments similar to those in Lemma \ref{Lemma:super} that $(\overline{u},\overline{v})$ is a super-solution of \eqref{eq:uppermain1}. Choose $\epsilon$ smaller and $K_1$, $K_2$ larger so that
\begin{align*}
&\underline{u}\leq u^r(T^*_0;u_0,v_0)\leq \overline{u};\\
&\underline{v}\leq v^r(T^*_0 ;u_0,v_0)\leq \overline{v}.
\end{align*}
on $\overline{\mathcal{O}_{A_3}}$. By the uniqueness and the comparison principle (Proposition \ref{prof:comparsys}), one has
\begin{align}\label{mainproof}
\begin{array}{llll}
&u^r(t ;\underline{u},\underline{v})\leq u^r(t+T^*_0 ;u_0,v_0)\leq u^r(t ;\overline{u},\overline{v});\\
&v^r(t ;\underline{u},\underline{v})\leq v^r(t+T^*_0 ;u_0,v_0)\leq v^r(t ;\overline{u},\overline{v}),
\end{array}
\end{align}
on $\overline{\mathcal{O}_{A_3}}$, for any $t>0$. It is known that $(u^r(t ;\overline{u},\overline{v}),v^r(t ;\overline{u},\overline{v}))$ and $(u^r(t ;\underline{u},\underline{v}),v^r(t ;\underline{u},\underline{v}))$ exist globally in time on $\overline{\mathcal{O}_{A}}$ (see Remark \ref{remark2} and \ref{remark3}).

Due to the uniqueness of \eqref{eq:uppermain} on $\overline{\mathcal{O}_{A}}$, it is easy to check that
\begin{align*}
u(t ;\underline{u},\underline{v})=u(t ;\underline{u});~v(t ;\underline{u},\underline{v})=0~\text{  on $\mathcal{O}_{A}$ for any $t>0$}
\end{align*}
and $u(t ;\underline{u})$ solves 
\begin{align*}
\begin{split}
\left\{\begin{array}{ll}
u_t + u_{a}=d_u\Delta_N u+m(x)u-n(x)u^2
-\mu(x,a)u,&t>0,~(x,a)\in\mathcal{O}_A, \\
u(t,x,0) = \displaystyle\int_0^{A} \beta(x,a) u(t,x,a)da,&t>0,\textbf{ }x\in\O,\\
u(0,x,a) = \underline{u}(x,a),&(x,a)\in\mathcal{O}_{A}.
\end{array}\right.
\end{split}
\end{align*}

Now, let us work with the right-hand side of \eqref{mainproof}. By comparison principle for age-structured model in Proposition \ref{prof:comparsys}, one has
\begin{align*}
u^r(t,x,a;\overline{u},\overline{u})\leq \overline{v};\quad v^r(t,x,a;\overline{u},\overline{v})\leq \overline{v},~\forall (x,a)\in \mathcal{O}_{A},~t>0.
\end{align*}
As the results, the uniqueness of solution admits that
\begin{align*}
\begin{array}{llll}
u^r(t+s,x,a;\overline{u},\overline{v})\leq u^r(t,x,a;\overline{u},\overline{v}),\\ v^r(t+s,x,a;\overline{u},\overline{v})\leq v^r(t,x,a;\overline{u},\overline{v}),
\end{array}
~\forall (x,a)\in \mathcal{O}_{A},~t,~s>0.
\end{align*}
Thus, there exists $w_1$ and $w_2$ such that as $t\rightarrow \infty$ 
\begin{align*}
&u^r(t,x,a;\overline{u},\overline{v}) \rightarrow w_1(x,a) \text{ pointwise},\\
&v^r(t,x,a;\overline{u},\overline{v})\rightarrow w_2(x,a) \text{ pointwise},
\end{align*}
for any $(x,a)\in \mathcal{O}_{A}$. By Lebesgue dominated convergence theorem and the fact that $(u,v)$ is bounded on $(0,\infty)\times \mathcal{O}_{A}$, one gets 
\begin{align*}
u^r(t_n ;\overline{u},\overline{v})\rightarrow w_1 \text{ strongly in }L^2(\mathcal{O}_{A}),\\
v^r(t_n ;\overline{u},\overline{v})\rightarrow w_2 \text{ strongly in }L^2(\mathcal{O}_{A}).
\end{align*}
for any sequence $t_n\rightarrow \infty$ as $n\rightarrow \infty$ (up to a sub-sequence). As the results, similar to the one in Lemma \ref{omegasetlimit}, one must have $w_1=u_1,~w_2=0$ and, as $t\rightarrow \infty$,
\begin{align*}
\begin{array}{llll}
&u^r(t ;\overline{u},\overline{v})\rightarrow u_1 &\text{ strongly in }L^2(\mathcal{O}_{A}),\\
&v^r(t ;\overline{u},\overline{v})\rightarrow 0 &\text{ strongly in }L^2(\mathcal{O}_{A}),\\
&u^r(t,x,a;\overline{u},\overline{v}) \rightarrow u_1(x,a) &\text{ pointwise},\\
&v^r(t,x,a;\overline{u},\overline{v})\rightarrow 0 &\text{ pointwise},
\end{array}
\end{align*}
where $u_1$ solves
\begin{align*}
\begin{split}
\left\{\begin{array}{llll}
 u_{a}=d_u\Delta_Nu +m(x)u-n(x)u^2
-\mu(x,a)u,&(x,a)\in\mathcal{O}_{A}, \\
u(x,0) = \displaystyle\int_0^{A} \beta(x,a) u(x,a)da,&x\in \O.
\end{array}\right.
\end{split}
\end{align*}
It follows from Dini's theorem that
\begin{align*}
\begin{array}{llll}
&u^r(t ;\overline{u},\overline{v}) \rightarrow u_1 &\text{ uniformly on }\overline{\mathcal{O}_{A_3}},\\
&v^r(t ;\overline{u},\overline{v})\rightarrow 0 &\text{ uniformly on }\overline{\mathcal{O}_{A_3}}.
\end{array}
\end{align*}
By similar arguments for the left-hand side of \eqref{mainproof}, one has
\begin{align*}
\begin{array}{llll}
&u^r(t ;\underline{u},\underline{v})=u^r(t ;\underline{u}) \rightarrow u_1 &\text{ uniformly on }\overline{\mathcal{O}_{A_3}}.\\
\end{array}
\end{align*}
This is possible since 
\begin{align*}
\begin{split}
\left\{\begin{array}{llll}
 u_{a}=d_u\Delta_Nu +m(x)u-n(x)u^2
-\mu(x,a)u,&(x,a)\in\mathcal{O}_{A_3}, \\
u(x,0) = \displaystyle\int_0^{A_3} \beta(x,a) u(x,a)da,&x\in \O.
\end{array}\right.
\end{split}
\end{align*}
admits the unique solution $u_1$ and $\supp(\beta)\subset \overline{\O}\times (A_0,A_2] \subset[0,A_3]$. Consequently, one can prove that
\begin{align*}
\begin{array}{llll}
&u^r(t ;u_0,v_0) \rightarrow u_1 &\text{ uniformly on }\overline{\mathcal{O}_{A_3}},\\
&v^r(t ;u_0,v_0)\rightarrow 0 &\text{ uniformly on }\overline{\mathcal{O}_{A_3}}.
\end{array}
\end{align*}
This completes the proof.
\end{proof}

\phantom{1}

Now, let us return the system \eqref{eq:main}. Recall 
\begin{align*}
(u,v)=(u(t),v(t))=(u(t;u_0,v_0),v(t;u_0,v_0))
\end{align*}
is the solution of the system \eqref{eq:main}. Thanks to estimations in Theorem \ref{theo:2}, one can have that
\begin{align*}
&0<u(t;u_0,v_0)\leq u^r(t;u_0,v_0),\\
&0<v(t;u_0,v_0)\leq v^r(t;u_0,v_0)\leq C_2 e^{\lambda_{d_v,-p+q+\mu} t},
\end{align*}
where $C_2$ can be found in Lemma \ref{lemma:conver0}. Thus, one has, as $t\rightarrow \infty $
\begin{align}\label{zeroinfected}
||v(t;u_0,v_0)||_{L^{\infty}(\mathcal{O}_A)}\rightarrow 0,
\end{align}
and 
\begin{align*}
0\leq \liminf\limits_{t\rightarrow \infty} u(t;u_0,v_0)\leq \limsup\limits_{t\rightarrow \infty} u(t;u_0,v_0) \leq u_1 \text{ uniformly on }\overline{\mathcal{O}_{A_3}}.
\end{align*}
for $A_2<A_3<A$. In addition, the following properties hold 
\begin{align*}
||u(t;u_0,v_0)||_{L^{\infty}(\mathcal{O}_{A})}+||v(t;u_0,v_0)||_{L^{\infty}(\mathcal{O}_{A})}\leq 2N,~\forall t>0.
\end{align*}
and
\begin{align*}
||u(t;u_0,v_0)||_{L^{2}(\mathcal{O}_{A})}+||v(t;u_0,v_0)||_{L^{2}(\mathcal{O}_{A})}\leq H,~\forall t>0,
\end{align*} 

We are position to prove the stability of disease-free equilibrium for system \eqref{eq:main}.

\begin{proof}[Proof of Theorem \ref{Theo:2} \ref{sta:DFE}]
Consider $A_2<A_3<A$. Thanks to the fact that $u>0$ on $\mathcal{O}_{A_3}$ and \eqref{zeroinfected}, it follows from the mild formula that for any $\varepsilon>0$, there exists $T=T(\varepsilon)>0$ such that
\begin{align*}
&\left|u(t)-\Phi_{d_u,\mu}(t)(u_0)-\int_0^t\Phi_{d_u,\mu}(t-s)\left[m(x)u(s)-n(x)u^2(s)\right]ds\right|\\
&<\varepsilon \int_0^t\Phi_{d_u,\mu}(t-s)ds,
\end{align*}
for any $t>T$ and uniformly on $\overline{\mathcal{O}_{A}}$. Due to the uniqueness, one can rewrite as follows
\begin{align*}
&\left|u(t+T)-\Phi_{d_u,\mu}(t)(u(T))-\int_{T}^{t+T}\Phi_{d_u,\mu}(t+T-s)\left[m(x)u(s)-n(x)u^2(s)\right]ds\right|\\
&<\varepsilon \int_{T}^{t+T}\Phi_{d_u,\mu}(t+T-s)ds,
\end{align*}
for any $t>0$ and uniformly on $\overline{\mathcal{O}_{A}}$. Consider $A_2<A_3<A$ and $\delta$ small enough. Choose $\varepsilon = \delta \min\limits_{ \overline{\mathcal{O}_{A_3}}}n(x)u_1^2>0$, one can deduce that
\begin{align*}
u(t+T)&>\Phi_{d_u,\mu}(t)(u(T))\\
&+\int_{T}^{t+T}\Phi_{d_u,\mu}(t+T-s)\left[m(x)u(s)-n(x)u^2(s)-\delta n(x)u_1^2 \right]ds
\end{align*}
Define
\begin{align}\label{eq:below}
\begin{split}
w(t) &:= \Phi_{d_u,\mu}(t)(u(T;u_0,v_0))\\
&+\int_{T}^{t+T}\Phi_{d_u,\mu}(t+T-s)\left(m(x)w(s)-n(x)w^2(s)-\delta  n(x)u_1^2 \right)ds
\end{split}
\end{align}
This may have the steady state $w^-$ satisfying
\begin{align}\label{steady-state1}
\left\{\begin{array}{llll}
w_a-\Delta_Nw +\mu(x,a)w= m(x)w-n(x)w^2  -\delta n(x)u_1^2,\\
w(x,0) = \displaystyle \int_0^A\beta(x,a)w(x,a)da.
\end{array}\right.
\end{align}
To prove the existence of the steady state, for $\epsilon>0$ that will be chosen later, we define $w^{-}=(1-\epsilon)u_1$. Then, one has
\begin{align*}
w^{-}_a - \Delta_N w^- = m(x)w^--n(x)w^- u_1=m(x)w^--n(x)(w^-)^2 -\epsilon(1-\epsilon) n(x)u_1^2
\end{align*}
At this point, we choose $\epsilon>0$ is the positive solution of $\epsilon^2 -\epsilon + \delta =0$. If $0<\delta<\dfrac{1}{4}$, we obtain $\epsilon = \dfrac{1-\sqrt{1-4\delta}}{2}>0$. In addition, it follows from a bootstrap argument that any steady-state solution belongs to $ C\left(\overline{\mathcal{O}_{A_3}}\right)\cap C^{2,1}\left(\mathcal{O}_{A_3}\right)$.


Now, inspired of the technique in \cite[Theorem 3.3]{zhao_spatiotemporal_2023}, we define $\kappa = 2\sup\limits_{\O}\dfrac{m}{n}$ and
\begin{align*}
j(x,\tau)=\left\{\begin{array}{llll}
m(x)\tau,&\tau \leq 0\\
m(x)\tau -n(x)\tau^2,&0\leq \tau \leq \kappa\\
n(x)\kappa^2+m(x)\tau-2n(x)\kappa\tau,&\tau \geq \kappa.
\end{array}\right.
\end{align*}
Clearly, $j\in C^{1}\left(\overline{\O}\times \R\right)$. We define  the following Banach space $H^1_{\beta}(\mathcal{O}_{A_3})$
\begin{align*}
H^1_{\beta}(\mathcal{O}_{A_3}):=\left\{\phi \in H^1(\mathcal{O}_{A_3}):\phi(x,0)=\displaystyle\int_0^A\beta(x,a)\phi(x,a)da\right\}
\end{align*}

Define $\mathcal{F}:H^1_{\beta}(\mathcal{O}_{A_3})\times \R\rightarrow \mathcal{L}(H^1_{\beta}(\mathcal{O}_{A_3}),\R)$, the duality space of $H^1_{\beta}(\mathcal{O}_{A_3})$, as follows
\begin{align*}
\left<\mathcal{F}(w,\delta)~|~v\right>&:=\left<\mathcal{A}_{d_u,\mu}w~|~v\right>+\left<j(x,w)-\delta n(x)u^2_1~|~v\right>\\
&=-d_u\int_{\O}\nabla w\nabla v dx-\int_{\O}w_a v dx - \int_{\O}\mu(x,a)w v dx\\
& +\left<j(x,w)-\delta n(x)u^2_1~|~v\right>.
\end{align*}
for $w,~v\in H^1_{\beta}(\mathcal{O}_{A_3})$. One can check that $\mathcal{F}\in C^1\left(H^1_{\beta}(\mathcal{O}_{A_3})\times \R\right)$ and
\begin{align*}
\mathcal{F}(u_1,0)=0,~\left<\partial_{w}\mathcal{F}(u_1,0)w~|~v\right>=\left<\mathcal{A}_{d_u,\mu}w~|~v\right>+\left<[m(x)-2n(x)u_1]w~|~v\right>
\end{align*}

Similar to the technique used in Theorem \ref{main:theo1} \ref{EE}, it is known that $2nu_1-m+\mu\geq nu_1-m+\mu$ and $d_u\Delta_Nu_1-(u_1)_{a} 
-\left[m(x)-n(x)u_1-\mu(x,a)\right]u_1=0$. Thus, $\lambda_{d_u,-m+2nu_1+\mu}<\lambda_{d_u,-m+nu_1+\mu}=0$. This leads to the fact that the spectrum of $\partial_{w}\mathcal{F}(u_1,0)$ is in the negative real-part of the complex plane and, thus, $\partial_{w}\mathcal{F}(u_1,0)$ is invertible. It then follows from the implicit function theorem that $\mathcal{F}(w,\delta)=0$ has a unique solution on $H^1_{\beta}(\mathcal{O}_A)$ for any $\delta>0$ small enough. Similar to Theorem \ref{Theo:10}, it follows from \eqref{eq:below} that there exists $T_0>0$ such that
\begin{align*}
w(t)>w^- -\epsilon ||u_1||_{\infty},~\forall t>T_0,\text{ uniformly on }\overline{\mathcal{O}_{A_3}}.
\end{align*}
By the age-structured comparison principle (Proposition \ref{prof:comparsys1}), one has
\begin{align*}
u(t;u_0,v_0)> u_1-2\epsilon ||u_1||_{\infty}=u_1-(1-\sqrt{1-4\delta}) ||u_1||_{\infty},~\forall t>T+T_0.
\end{align*}
$\text{uniformly on }\overline{\mathcal{O}_{A_3}}$. This implies that
\begin{align*}
u_1-(1-\sqrt{1-4\delta}) ||u_1||_{\infty}\leq \liminf\limits_{t\rightarrow \infty} u(t;u_0,v_0)\leq \limsup\limits_{t\rightarrow \infty} u(t;u_0,v_0) \leq u_1 .
\end{align*}
$\text{uniformly on }\overline{\mathcal{O}_{A_3}}$. This is true for any $\delta>0$ small enough. One can deduce that
\begin{align*}
\lim\limits_{t\rightarrow \infty} u(t;u_0,v_0)=u_1 \text{ uniformly on }\overline{\mathcal{O}_{A_3}}.
\end{align*}

Finally, it is easy to check that 
\begin{align*}
\begin{array}{llll}
&u(t;u_0,v_0) \rightarrow u_1 \text{ pointwise on $\mathcal{O}_{A}$,}\\
\end{array}
\end{align*}
since we can choose a suitable $A_3$ for each $a$. It follows from Lebesgue dominated convergence theorem that
\begin{align*}
\begin{array}{llll}
&u(t;u_0,v_0)\rightarrow u_1 \text{ strongly in }L^2(\mathcal{O}_A),\\
\end{array}
\end{align*}
This concludes the proof.

\end{proof}

%



%

%

\subsection{\bf Stability of the endemic equilibrium $\mathcal{R}_0>1$}

Let us begin the section with some remarks. For any $(w,k)\in \omega(u_0,v_0)$, there exists $t_n$ such that
\begin{align*}
\begin{array}{llll}
&u^r(t_n;u_0,v_0)\rightarrow w &\text{ weakly on }L^2(\mathcal{O}_{A}),\\
&v^r(t_n;u_0,v_0)\rightarrow k &\text{ weakly on }L^2(\mathcal{O}_{A}).
\end{array}
\end{align*}
as $n\rightarrow \infty$. If these convergences are strong (up to a sub-sequence), we can prove $(w,k)$ is the solution of \eqref{eq:uppermain1} by similar calculations to the one in Lemma \ref{omegasetlimit}. An example for the convergences to be strong is by choosing $(u_0,v_0)$ as a super-solution or a sub-solution of \eqref{eq:uppermain1}. However, due to the lack of uniqueness, the global stability of the endemic equilibrium is not generally achievable, i.e., convergence as 
$t\rightarrow \infty$ cannot be ensured. As the results, the estimation of $(u,v)$, the solution of \eqref{eq:main}, becomes more complicated. Recall, there exists $T^*_0>A$ such that
\begin{align*}
u(t)>0,~v(t)>0,~\forall t\geq T^*_0 \text{ on } \overline{\O}\times [0,A).
\end{align*}

Let us state some results about long-time dynamics of \eqref{eq:main}.


\begin{proof}[Proof of Theorem \ref{Theo:2} \ref{sta:EE}]

Consider $A_2<A_3<A$, recall that
\begin{align*}
\displaystyle\int_0^{A} \beta(x,a) \phi(x,a)da=\displaystyle\int_0^{A_2} \beta(x,a) \phi(x,a)da = \displaystyle\int_0^{A_3} \beta(x,a) \phi(x,a)da
\end{align*}
since $\supp(\beta)\subset \overline{\O}\times (A_0,A_2]$. Let $\phi>0$ be the eigenfunction associated with $\lambda_{d_u,-m+\mu}>0$ 
\begin{align*}
\begin{split}
\left\{\begin{array}{l}
\phi_a-d_u\Delta_N\phi+\mu(x,a)\phi-m(x)\phi=-\lambda_{d_u,-m+\mu}\phi,\\
\phi(x,0) = \displaystyle\int_0^{A_3} \beta(x,a) \phi(x,a)da.
\end{array}\right.
\end{split}
\end{align*}
and $\varphi>0$ be the eigenfunction associated with $\lambda_{d_v,-p+q+\mu}>0$
\begin{align*}
\begin{split}
\left\{\begin{array}{l}
\varphi_a-d_v\Delta_N\varphi+\mu(x,a)\varphi-p(x)\varphi+q(x)\varphi=-\lambda_{d_v,-p+q+\mu}\varphi,\\
\varphi(x,0) = \displaystyle\int_0^{A_3} \beta(x,a) \varphi(x,a)da.
\end{array}\right.
\end{split}
\end{align*}

Similar to the one in Theorem \ref{Theo:10}, we choose $\epsilon_1>0$ small enough so that
\begin{align*}
\epsilon_1 \phi\leq u(T^*_0;u_0,v_0) \text{ on }\overline{\mathcal{O}_{A_3}}.
\end{align*}
and $\underline{u}:=\epsilon_1 \phi$ satisfies
\begin{align*}
\underline{u}_a-d_u\Delta_N\underline{u}+\mu(x,a)\underline{u}-m(x)\underline{u}&\leq -n(x)\underline{u}^2-p(x)\underline{u}
\end{align*}


Next, choose $h=\underline{u}$, $k_1=p$ and $k_2=q$ in Theorem A.\ref{Lemma:A.10}-Step 1 and $\epsilon_2>0$ small enough such that $\underline{v}:=\epsilon_2 \varphi$ satisfies
\begin{align*}
\underline{v}_{a}-d\Delta_N\underline{v}+
\mu(x,a)\underline{v}\leq p(x)\dfrac{\underline{u}}{\underline{u}+\underline{v}}\underline{v}-q(x)\underline{v} \text{ on }\mathcal{O}_{A_3},
\end{align*}
and
\begin{align*}
\epsilon_2 \varphi\leq v(T^*_0;u_0,v_0) \text{ on }\overline{\mathcal{O}_{A_3}}.
\end{align*}
As the results, $(\underline{u},\underline{v})$ is a sub-solution of the system associated $\textbf{F}^l$ in Lemma \ref{local} $\text{on }\overline{\mathcal{O}_{A_3}}$. Now, choose $\overline{u}:=K_1Y_{\alpha}(a)$ and $\overline{v}:=K_2Y_{\alpha}(a)$ with $K_1$, $K_2$ and $\alpha$ large enough such that $(\overline{u},\overline{v})$ is a super-solution of the system associated $\textbf{F}^r$ in Lemma \ref{Lemma:super} and
\begin{align*}
\begin{array}{llll}
&\underline{u}\leq u(T^*_0;u_0,v_0)\leq \overline{u}\\
&\underline{v}\leq v(T^*_0;u_0,v_0)\leq \overline{v}
\end{array}
\text{ on $\overline{\mathcal{O}_{A_3}}$.}
\end{align*}
By the uniqueness and the comparison principle (Proposition \ref{prof:comparsys1}), it follows that
\begin{align}\label{mainproof1}
\begin{array}{llll}
&\underline{u}\leq u(t+T^*_0;u_0,v_0)\leq \overline{v}\\
&\underline{v}\leq v(t+T^*_0;u_0,v_0)\leq \overline{v}
\end{array}
\text{ on $\overline{\mathcal{O}_{A_3}}$, for any $t>0$. }
\end{align}
As the result, there exist $\varepsilon_1,~\varepsilon_2>0$ such that
\begin{align*}
\begin{array}{llll}
&0<\varepsilon_1< \liminf\limits_{t\rightarrow \infty}u(t;u_0,v_0)\leq \limsup\limits_{t\rightarrow \infty}u(t;u_0,v_0)\leq  K_1,\\
&0<\varepsilon_2< \liminf\limits_{t\rightarrow \infty}v(t;u_0,v_0)\leq \limsup\limits_{t\rightarrow \infty}v(t;u_0,v_0)\leq K_2,
\end{array} \text{ uniformly on $\overline{\mathcal{O}_{A_3}}$.}
\end{align*}
which conclude the proof.
\end{proof}
 The results in Theorem \ref{Theo:2} \ref{sta:EE} show the persistent of the disease in the population when $\mathcal{R}_{0}>1$, a potential outbreak of the disease.


%

\section{\bf Stability of the endemic equilibrium $\mathcal{R}_0>1$ when $m=n=0$}\label{section:5}
In this section, we reduce the logistic term $mu-nu^2$. The main objective is to establish the stability of the endemic equilibrium by showing the uniqueness of the steady state in a certain sense, as will be discussed later. Although the background setting is similar, we will highlight some important differences. From this point, we assume that
\begin{enumerate}[label=(B\arabic*)]  
\setcounter{enumi}{0}
\item \label{condA'1} $d_u=d_v=d,~p,~q>0 \text{ are  constants}$ and $m=n=0$. 
\end{enumerate}

Let us consider an alternative version of the assumptions \ref{cond:cond3} to \ref{cond:cond5} as follows
\begin{enumerate}[label=(B\arabic*)]  
\setcounter{enumi}{1}
 \item \label{condA'2} $\mu\in C^{1}([0,A))$, $\mu>0$ satisfies
\begin{align*} 
\int_0^A \mu(k)dk = \infty.
\end{align*}
and
\begin{align*}
a_0\mapsto \sup_{a\in [0,a_0]}\mu(a)\text{ is continuous with $a_0\in [0,A]$.}
\end{align*}

\item \label{condA'3} $\beta$ is a non-negative function in $C^{1}([0,A])$ and
there exists $A_0,~A_1,~A_2 \in (0,A)$, $A_0<A_1<A_2$ such that 
\begin{align*}
\beta(a)>0,~\forall a\in (A_0,A_1) \text{ and }\supp(\beta)\subset(A_0,A_2],
\end{align*}
and
\begin{align*}
R:=\displaystyle\int_0^A\beta(a)e^{-\displaystyle\int_0^a\mu(k)dk}da= 1.
\end{align*}

\end{enumerate}

In \cite{gurtin_non-linear_1974}, the authors interpreted $R$ as the expected number of children born to an individual or the net reproduction rate.
For convenience, we define
\begin{align*}
\pi_0(a) := e^{-\displaystyle\int_0^a\mu(k)dk},~L_0(a) := e^{\displaystyle\int_0^a\mu(k)dk}.
\end{align*}


Let us consider the initial condition $u_0,~v_0$ of the system \eqref{eq:main-no_m_n} satisfies 
\begin{align}\label{bounded:above_nospace}
0\leq u_0\leq K\pi_0,~0\leq v_0\leq H\pi_0 \text{ on }\mathcal{O}_A.
\end{align}
for some constants $K,~H>0$. Similar to Lemma \ref{local}, there exists 
\begin{align*}
(u,v)=(u(t;u_0,v_0),v(t;u_0,v_0))=(u(t,x,a;u_0,v_0),v(t,x,a;u_0,v_0))
\end{align*}
is a solution of \eqref{eq:main-no_m_n} and it is uniqueness in local time. Then, from first and second equation, with $w:=u+v$, one has
\begin{align*}
w_t+w_a=d\Delta_Nw-\mu(a)w.
\end{align*}
Furthermore, one can check $\overline{w}:=(K+H)\pi_0(a)$ satisfying $\overline{w}_a= -\mu(a)\overline{w}$. It follows from the age-structured comparison principle that $w\leq \overline{w}$ on $\mathcal{O}_A$. As the results, $0\leq u\leq \overline{w}$ on $\mathcal{O}_A$ and $0\leq v\leq \overline{w}$ on $\mathcal{O}_A$ for suitable time $t>0$, meaning the solution exists globally for any $t>0$. We further obtain that $u(t)>0,~v(t)>0$ for $t>0$ if $u_0\not\equiv 0$, $v_0\not\equiv 0$.

Now, let us consider the system
\begin{align}\label{eq:main_nospace}
\begin{split}
\left\{\begin{array}{ll}
u_t + u_{a}=
-\dfrac{puv}{u+v}+qv-\mu(a)u,&t>0,~a\in(0,A),\\
v_t + v_{a}=
\dfrac{puv}{u+v}-qv-\mu(a)v,&t>0,~a\in(0,A),\\
u(t,0) = \displaystyle\int_0^{A} \beta(a) u(t,a)da,&t>0,\\
v(t,0) = \displaystyle\int_0^{A} \beta(a) v(t,a)da,&t>0,\\
u(0,a) = u_0(a),~
v(0,a) = v_0(a),&a\in(0,A),
\end{array}\right.
\end{split}
\end{align}
Similarly, there exists $(u,v)=(u(t),v(t))=(u(t,a;u_0,v_0),v(t,a;u_0,v_0))$ is a solution of \eqref{eq:main_nospace} and it is uniqueness. On the other hand, this solution is also the solution of \eqref{eq:main-no_m_n}. As the result, the solution of \eqref{eq:main-no_m_n} is independent of space when the initial condition is independent of space.

Let us analyze the steady state system \eqref{eq:main1_nospace}.
For $u_0\in C_c([0,A))\setminus\{0\}$, $u_0\geq 0$, when $v=0$, \eqref{eq:main-no_m_n} admits a unique solution $u(t)=u(t,a;u_0)$ satisfying
\begin{align*}
\left\{\begin{array}{llll}
u_t + u_{a}=
-\mu(a)u,\\
u(t,0)=\displaystyle \int_0^A\beta(a)u(t,a)da,\\
u(0,a)=u_0(a)
\end{array}
\right.
\end{align*}
and \eqref{eq:main_nospace} admits
\begin{align*}
u_*(a) = u_*(0)\pi_0(a). 
\end{align*}
for some constant $u_*(0)\geq 0$. From the work of \cite[Remark 2.3.1 and Theorem 2.3.3]{anita_2000}, for each $u_0\geq 0$ such that \begin{align*}
\int_0^A\beta(a)u_0(a)da\neq 0
\end{align*}
there exists $u_*(0)>0$ such that $u(t)\rightarrow u_*$ as $t\rightarrow \infty$ uniformly on $a\in (0,A)$. Next, we linearize the second equation of \eqref{eq:main-no_m_n} around the  disease-free equilibrium $E^0=(u_*,0)$ to obtain
\begin{align}\label{linear_age1}
\left\{\begin{array}{llll}
v_t+v_a=(p-q-\mu(a))v,&t>0,~a\in (0,A),\\
v(t,0)=\displaystyle\int_0^A\beta(a)v(t,a)da.
\end{array}\right.
\end{align}
Using the characteristic line method in \cite[Chapter 2]{evans_partial_2010}, we obtain that
\begin{align*}
v(t,a)=\left\{\begin{array}{lll}
e^{(p-q)t}\dfrac{\pi_0(a)}{\pi_0(a-t)}v_0(a-t)&a\geq t\\
e^{(p-q)a}\pi_0(a)v(t-a,0)&a<t.
\end{array}\right.
\end{align*}
We substitute this to the age-structure condition to obtain
\begin{align*}
v(t,0)&=\int_0^{\min(t,A)}\beta(a)\pi_0(a)e^{(p-q)a}v(t-a,0)da\\
&+\int^A_{\min\{t,A\}} \beta(a)\dfrac{\pi_0(a)}{\pi_0(a-t)} e^{(p-q)t}v_0(a-t)da.
\end{align*}
From here, similar to the one in \cite{chekroun_global_2020,yang_asymptotical_2023,ducrot_travelling_2009}, we define the basic reproduction number $\mathcal{R}_0$ as follows
\begin{align*}
\mathcal{R}_0 := \int_0^{A}\beta(a)\pi_0(a)e^{(p-q)a}da.
\end{align*}
Thanks to the condition \ref{condA'3}, one has $\mathcal{R}_0 > 1$ if and only if $p>q$. In addition, $p=q$ if and only if $\mathcal{R}_0 = 1$. On the other hand, without state the exact formula of the next infected generation operator $\mathcal{I}$, one has 
\begin{theorem}
Assume \ref{condA'1} to \ref{condA'3} hold. Let $\mathcal{I}$ be the next infected generation operator (NIGO)  and $\lambda_{d_v,-p+q+\mu}$ denote the principal eigenvalue associated with  \eqref{linear_age1}. Then, $\mathcal{R}_0 = r(\mathcal{I})$, the spectral radius of $\mathcal{I}$ and $\mathcal{R}_0>1$ if and only if $\lambda_{d_v,-p+q+\mu}>0$. Furthermore, $\mathcal{R}_0=1$ if and only if $\lambda_{d_v,-p+q+\mu}=0$.
\end{theorem}
\begin{remark}
Consequently, $\lambda_{d_v,-p+q+\mu}>0$ if and only if $p>q$. Furthermore, $\lambda_{d_v,-p+q+\mu}=0$ if and only if $p=q$.
\end{remark}

The remainder of this section closely follows the structure of the previous one. Next, from the first and second equations of \eqref{eq:main1_nospace}, with $w:=u+v$, one has
\begin{align*}
 w(a) = -\mu(a)w
\end{align*}
which is $w^*(a)=w(0)\pi_0(a)$ for some $w^*(0)>0$. 

We are in position to state the existence and uniqueness of \eqref{eq:main1_nospace} in the case endemic equilibrium as follows

\begin{theorem}\label{theo:16}
Assume \ref{condA'1} to \ref{condA'3} hold. Suppose further that  $\mathcal{R}_0>1$. Then, for any $w^*(0)>0$, \eqref{eq:main1_nospace} admits a unique positive solution $(u^*,v^*)$.
\end{theorem}


\begin{proof}
Let $A_3$ be the number satisfying $A_2<A_3<A$. Using the relation $v=w-u$, it follows from the second equation that
\begin{align}\label{eq:KPPv}
v_a+\mu(a) v = (p-q)v - \dfrac{p}{w^*(0)\pi_0(a)} v^2 \text{ on }\O\times (0,A_3)
\end{align}
Similar to Theorem A.\ref{Lemma A.7} and the fact that $\mathcal{R}_0>1$, for each $w^*(0)>0$,  there exists a unique $v^*$ is the solution of \eqref{eq:KPPv}. As the results, with the uniqueness, one can extend $v^*$ to the domain $\overline{\mathcal{O}_{A}}$ using similar arguments to the one in Theorem A.\ref{Lemma A.7}. This same holds for $u^*$.
\end{proof}

\begin{remark}
Consider the system
\begin{align}\label{eq:main2_nospace}
\begin{split}
\left\{\begin{array}{ll}
u_{a}=d\Delta_Nu
-\dfrac{puv}{u+v}+qv-\mu(a)u,&(x,a)\in\mathcal{O}_A,\\
v_{a}=d\Delta_Nu+
\dfrac{puv}{u+v}-qv-\mu(a)v,&(x,a)\in\mathcal{O}_A,\\
u(0) = \displaystyle\int_0^{A} \beta(a) u(a)da,\\
v(0) = \displaystyle\int_0^{A} \beta(a) v(a)da.\\
\end{array}\right.
\end{split}
\end{align}
One can check the uniqueness thanks to Hopf's lemma and strong maximum principle for the parabolic equation in a similar manner to Lemma A.\ref{comparison}. Thus, the system \eqref{eq:main1_nospace} and \eqref{eq:main2_nospace} admit the same unique positive solution.

\end{remark}


Next, let us analyze the system \eqref{eq:main-no_m_n}. Assume $u_0,~v_0\in \mathcal{X}_+$, $u_0\leq H \pi_0(a)$, $u_0\leq K \pi_0(a)$ for some $K,~H>0$ and $u_0,~v_0,~u_0+v_0\not\equiv 0$. From \eqref{bounded:above_nospace}, we get
\begin{align*}
\begin{split}
\left\{\begin{array}{ll}
w_t+w_a=d\Delta_Nw-\mu(a)w,\\
w(t,x,0) = \displaystyle\int_0^{A} \beta(a) w(t,x,a)da,\\
w(0,x,a) = w_0(x,a).
\end{array}\right.
\end{split}
\end{align*}
where $w:=u+v>0$ and $w_0:=u_0+v_0\geq 0$ on $\mathcal{O}_A$. Consider $A_2<A_3<A$. Based on similar arguments to \cite[Theorem 4.9]{langlais_large_1988}, it follows from $R=1$  that
\begin{align*}
w(t;u_0,v_0)\rightarrow w^*(0)\pi_0(a)\varphi_1(x) \text{ in $L^2(\mathcal{O}_{A_3})$ }
\end{align*} 
where
\begin{align*}
w^*(0):=\dfrac{\displaystyle\int_0^A\beta(b)\displaystyle\int_0^be^{-\displaystyle \int_b^e\mu(s)ds}\int_{\O}w_0(x,e)\varphi_1(x)dxdeda}{\displaystyle\int_0^A\beta(b)b\pi_0(b)db}>0.
\end{align*}

Note that since $\varphi_1=1$ is the only $L^2$-normalized positive eigenvector associated with the first eigenvalue $\lambda_1=0$ of $-d\Delta_N$, one can omit the eigenvector in this result.
Hence, the point-wise convergence on $\mathcal{O}_A$ is established. Since $w$ is bounded on $\mathcal{O}_A$, it follows that
\begin{align*}
w(t;u_0,v_0)\rightarrow w^*(0)\pi_0(a)\ \text{ strongly in $L^2(\mathcal{O}_A)$ }
\end{align*} 

Next, the second equation can be rewritten as follows
\begin{align*}
v_t+v_a + \mu(a)v= d\Delta_N v + (p-q)v - \dfrac{p}{w}v^2 \text{ for }t>0.
\end{align*}
Since $\mathcal{R}_0>1$, we shall prove that, $\text{as }t\rightarrow \infty$,
\begin{align*}
v(t;u_0,v_0)\rightarrow v^*(a)>0 \text{ pointwise on $\mathcal{O}_A$ }
\end{align*}
where $v^*$ is the unique positive solution of 
\begin{align}\label{Steadystate}
\left\{\begin{array}{llll}
(v^*)_a+\mu(a) v^* = (p-q)v^* - \dfrac{p}{w^*(0)\pi(a)}(v^*)^2,\\
v^*(0) = \displaystyle\int_0^A\beta(a)v^*(a)da.
\end{array}\right.
\end{align}
Similarly, one has $u(t)\rightarrow u^*:=w^*-v^*$.

We are in position to prove the stability of endemic equilibrium for system \eqref{eq:main-no_m_n}.

\begin{proof}[Proof of Theorem \ref{theo:stab_no_m_n}]
For $A_2<A_3<A$, the key point is to show $v_0=K Y_{\alpha}$  is a super solution of the following equation
\begin{align}\label{KPPw}
\left\{\begin{array}{llll}
v_t+v_a + \mu(a)v=d\Delta_N v + (p-q)v - \dfrac{p}{w}v^2,~t>0,~(x,a)\in \mathcal{O}_{A_3}\\
v(t,x,0)=\displaystyle \int_0^A\beta(a)v(t,x,a)da.
\end{array}\right.
\end{align}
 for $\alpha,~K>0$ large enough. One can check that
\begin{align*}
(v_0)_t+(v_0)_a+\mu(a)v_0&-(p-q)v_0+\dfrac{p}{w}v_0^2\\
&= \alpha KY_{\alpha}-KY_{\alpha}^2-(p-q)KY_{\alpha} + K^2\dfrac{p}{w}Y^2_{\alpha}\\
 &\geq KY_{\alpha}(\alpha -Y_{\alpha}-p+q)\geq 0
\end{align*}
provided that $\alpha>0$ large enough and $w(t;u_0,v_0)>0$ on $\mathcal{O}_{A_3}$, for any $t>0$. Thus, $v_0$ is a super solution of \eqref{KPPw}. Note that the choice of $K$ and $\alpha$ is independent with $w$. By the age-structured comparison principle (Proposition \ref{prof:comparoneeq}) for the domain $\mathcal{O}_{A_3}$, one can prove that
\begin{align*}
v(t;u_0,v_0)\leq v_0\text{ and } v(t+s;u_0,v_0)\leq v(t;u_0,v_0)
\end{align*}
for any $(x,a)\in \mathcal{O}_{A_3}$, $t,~s>0$. As the results, there exists a measurable function $h^*\geq 0$  such that
\begin{align*}
v(t;u_0,v_0) \rightarrow h^* \text{ as $t\rightarrow \infty$ pointwise on }\mathcal{O}_{A_3}.
\end{align*}
The rest is similar to Lemma \ref{omegasetlimit} and Theorem \ref{Theo:10}. We highlight an important detail in nonlinear term
\begin{align*}
III = -\int_{(-1,1)\times \mathcal{O}_{A_3}}\dfrac{p v^2(t_n+s;u_0,v_0)}{w(t_n+s;u_0,v_0)}\chi \rho(s)dsdxda
\end{align*}
It is known that
\begin{align*}
\dfrac{p v^2(t_n+s;u_0,v_0)}{w(t_n+s;u_0,v_0)} \rightarrow \dfrac{p (h^*)^2}{w^*(0)\pi(a)}
\end{align*}
and
\begin{align*}
\left|\dfrac{p v^2(t_n+s;u_0,v_0)}{w(t_n+s;u_0,v_0)}\right|=\dfrac{p v^2(t_n+s;u_0,v_0)}{u(t_n+s;u_0,v_0)+v(t_n+s;u_0,v_0)}\leq pv(t_n+s)\leq pM.
\end{align*}
for some $M>0$. Thanks to Lebesque dominated convergence theorem, one has
\begin{align*}
III \rightarrow -\int_{(-1,1)\times \mathcal{O}_{A_3}}\dfrac{p (h^*)^2}{w^*(0)\pi(a)}\chi \rho(s)dsdxda.
\end{align*}
It then follows that $h^*\equiv v^*\text{ is the unique solution of \eqref{Steadystate} and}$ 
\begin{align*}
v(t;u_0,v_0)\rightarrow v^* \text{ strongly in }L^2(\mathcal{O}_{A_3}).
\end{align*} 
Define $u^* =w^*(0)\pi(a)-v^*$. One can show that 
\begin{align*}
u(t;u_0,v_0)\rightarrow u^* \text{ strongly in }L^2(\mathcal{O}_{A_3}).
\end{align*}

Finally, it is easy to check that 
\begin{align*}
\begin{array}{llll}
&u(t;u_0,v_0) \rightarrow u^* &\text{ pointwise on $\mathcal{O}_{A}$,}\\
&v(t;u_0,v_0) \rightarrow v^* &\text{ pointwise on $\mathcal{O}_{A}$,}
\end{array}
\end{align*}
since we can choose a suitable $A_3$ for each $a$. It follows from Lebesgue dominated convergence theorem that
\begin{align*}
\begin{array}{llll}
&u(t;u_0,v_0)\rightarrow u^* &\text{ strongly in }L^2(\mathcal{O}_A),\\
&v(t;u_0,v_0)\rightarrow v^* &\text{ strongly in }L^2(\mathcal{O}_A),\\
\end{array}
\end{align*}
This concludes the proof.

\end{proof}

\appendix
\section{\bf Spectral Theory for Age-Structured Model with Neumann Boundary Conditions} \label{Appen A}

In this section, we present the spectral theory of age-structured model for Neumann boundary condition, analogues results with \cite{delgado_nonlinear_2006}. This work is essential since death rate $\mu$ blows up at $a=A$ in the sense of $L^1$. 

\begin{defa}\label{DefA.1}
Denote by $L^2_+(\O):=\{f\in L^2(\O):f(x)\geq 0 ~a.e,~x\in \O\}$. We say that $u\in L^2_+{(\O)}$ is \textit{quasi-interior point} of $L^2_+{(\O)}$, which is written as $u\gg 0$, if 
\begin{align*}
\int_{\O}u(x)f(x)dx>0~\text{ for all }f\in L^2_+{(\O)} \text{ and nontrivial}.
\end{align*}
\end{defa}
We say $H^{-1}(\O)=\mathcal{L}(H^1(\O),\R)$, the duality of $H^1(\O)$ and $\left<\cdot,\cdot\right>$ is scalar product for the duality of $H^{-1}(\O)$ and $H^1(\O)$.

\begin{defa}
For any $f\in L^2(\mathcal{O}_A)$, we say $u\in  L^2(0,A;H^1(\O))$, $u_a+qu\in L^2(0,A;H^{-1}(\O))$ is a weak solution if 
\begin{align*}
\left\{\begin{array}{llll}
\displaystyle\int_0^A\left<u_a+qu,v\right>da+\displaystyle\int_{\mathcal{\mathcal{O}_A}}\nabla u\cdot\nabla v dxda=\int_{\mathcal{O}_A}f(x,a)vdxda,\\
u(x,0) = \phi(x),&x\in \O.
\end{array}\right.
\end{align*}
for any $v\in L^2(0,A;H^1(\O))$.
\end{defa}

Let us start with the eigenvalue-eigenfuntion equation as follows
\begin{align}\label{eq:onevareigen}
\begin{split}
\left\{\begin{array}{llll}
\phi_{a}-d\Delta_N\phi+
\mathcal{M}(a)\phi=\lambda^*\phi,&(x,a)\in\mathcal{O}_A,\\
\dfrac{\partial \phi}{\partial \textbf{n}}=0,&x\in \partial \O,~a\in (0,A),\\
\phi(x,0) = \displaystyle\int_0^{A} \gamma(a) \phi(x,a)da,&x\in \O.
\end{array}\right.
\end{split}
\end{align}
where $d>0$ and 
\begin{enumerate}[label=(C\arabic*)] 
\item \label{cond:B1} $\mathcal{M}\in L^{\infty}(0,A_0)$ for any $0<A_0<A$ and
\begin{align*}
\int_0^A\mathcal{M}(a)da=\infty
\end{align*}
\item \label{cond:B2} $\gamma \in L^{\infty}(0,A)$, $\gamma\geq 0$ and there exists $a_0,~a_1\in (0,A)$ and $a_0< a_1$ such that $\gamma>0$ on $(a_0,a_1)$.

\end{enumerate}

\begin{lma}\label{Lemma A.1}
Suppose that \ref{cond:B1} and \ref{cond:B2} hold. Then, \eqref{eq:onevareigen} admits a positive solution on $\overline{\O}\times (0,A)$ if and only if
\begin{align*}
\lambda^* = \lambda_1+r_\mathcal{M}
\end{align*}
where $r_\mathcal{M}$ is the unique solution of 
\begin{align*}
\int_0^A\gamma(a)\exp\left(r_\mathcal{M} a - \int_0^a \mathcal{M}(s)ds\right)da=1
\end{align*}
and $\lambda_1= 0$ is the principal eigenvalue of $-d \Delta_N$. Furthermore, one has
\begin{align*}
\phi(x,a)=\exp\left(r_\mathcal{M} a - \int_0^a \mathcal{M}(s)ds\right)\varphi_1(x)
\end{align*}
where $\varphi_1$ is the eigenfunction associated with $\lambda_1$ and $\varphi_1>0$ on $\overline{\O}$.
\end{lma}
\begin{proof}
To start the proof, we adapt the technique in \cite[Theorem 3.5]{langlais_large_1988} to show all solutions of \eqref{eq:onevareigen} is separable. We denote by $(\lambda_i,\varphi_i)$ the eigenvalues and eigenfunctions of the Neumann problem of $-d\Delta_N $ in $\O$, meaning
\begin{align}\label{eq:NeumannL}
\begin{split}
\left\{\begin{array}{llll}
-d\Delta_N\varphi_i=\lambda_i\varphi_i,&x\in\O,\\
\dfrac{\partial \varphi_i}{\partial \textbf{n}}=0&x\in \partial \O.
\end{array}\right.
\end{split}
\end{align}
It is well-known that $\{u_i\}_{i\geq 1}$ is an orthonormal basis of $L^2(\O)$. Thus, one has
\begin{align*}
\phi(x,0)=\sum_{i\geq 1} v_i \varphi_i \text{ in }L^2(\O)
\end{align*}
for some sequence $\{v_i\}$ in $\R$. Then,
\begin{align*}
\phi(x,a)=\sum_{i\geq 1}v_i \exp\left(-\lambda_i a+\lambda^* a - \int_0^a \mathcal{M}(s)ds\right) \varphi_i
\end{align*}
By using the age-structured condition, one gets
\begin{align*}
\sum_{i\geq 1} v_i \varphi_i=\sum_{i\geq 1}v_i \varphi_i \int_0^A\gamma(a) \exp\left(r_i a - \int_0^a \mathcal{M}(s)ds\right)da 
\end{align*}
where $r_i:=\lambda^* - \lambda_i$. This leads to 
\begin{align*}
\text{ either }v_i=0 \text{ or }\int_0^A\gamma(a) \exp\left(r_i a - \int_0^a \mathcal{M}(s)ds\right)da=1
\end{align*}
Thanks to the fact that $\displaystyle\int_0^A\gamma(a) \exp\left(r_i a - \int_0^a \mathcal{M}(s)ds\right)da=1$ admits at most one solution $r_i=r_\mathcal{M}$, we conclude the separable property. 

Now, consider $\phi(x,a)=X(x)A(a)$. Then, one has
\begin{align*}
A(a)=p_0\exp\left(r a - \int_0^a \mathcal{M}(s)ds\right)
\end{align*}
and 
\begin{align*}
A(0)=\int_0^A\gamma(a) A(a) da \text{ if and only if }r=r_\mathcal{M}.
\end{align*}
On the other hand, one gets
\begin{align*}
\begin{split}
\left\{\begin{array}{llll}
-d\Delta_NX=(\lambda^* - r_\mathcal{M})X,&x\in\O,\\
\dfrac{\partial X}{\partial \textbf{n}}=0&x\in \partial \O.
\end{array}\right.
\end{split}
\end{align*}
Using the fact that $\varphi_1$ is only the positive eigenfunction on $\overline{\O}$, we conclude that $\lambda^* = \lambda_1 + r_\mathcal{M}=r_\mathcal{M}$.
\end{proof}

\begin{remark}
Note that $\lim\limits_{a\rightarrow A}\phi(x,a)=0$ uniformly with $x$ since $\displaystyle\int_0^A\mathcal{M}(a)da=\infty$. In fact, it is well-known that $\varphi_1$ is the constant function. Thus, the solution can be written as $\phi(a,x)=C\exp\left(r_\mathcal{M} a - \displaystyle\int_0^a \mathcal{M}(s)ds\right)$ for some constant $C>0$. Furthermore, if the solution is regular enough, we can recover the Neumann boundary condition with the appropriate test functions.
\end{remark}

\begin{lma}
Assume \ref{cond:B1} holds. Consider $u$ is a solution (in weak sense) of
\begin{align}\label{eq:paraNeu}
\begin{split}
\left\{\begin{array}{llll}
u_{a}-d\Delta_Nu+
\mathcal{M}(a)u=f(x,a),&(x,a)\in\mathcal{O}_A,\\
\dfrac{\partial u}{\partial \textbf{n}}=0&x\in \partial \O,~a\in (0,A),\\
u(x,0) = \phi(x),&x\in \O.
\end{array}\right.
\end{split}
\end{align}
and $f\geq 0$, $\phi\geq 0$ and some of inequalities are strict. Then, $u\gg 0$ on $\mathcal{O}_A$.

If $f=\phi=0$, then $u=0$.
\end{lma}
\begin{proof}
The proof is the analogous result to the one in \cite[Lemma 6]{delgado_nonlinear_2006} based on the fact that the semi-group of Neumann-Laplacian is strongly positive (see, for instance, \cite[Theorem C-III-3.2(b)]{zbMATH03937879}). We modify a detail as follows: The solution is a constant function $u=C$ in the case $f=\phi=0$. Then, 
\begin{align*}
\mathcal{M}(a)C=0,~\forall a\in(0,A).
\end{align*}
If $C\neq 0$, then $\mathcal{M}(a)=0,~\forall a\in(0,A)$. Contradiction with \ref{cond:B1}. Thus, $C=0$.
\end{proof}

Let us consider the following condition
\begin{enumerate}[label=(D\arabic*)]
\item \label{cond:C1}$\mu^*=\mu^*(x,a)\in L^{\infty}(\overline{\O}\times (0,A_0))$ for any $0<A_0<A$ and 
\begin{align*}
\mu^*_{\min}(a):=\inf\limits_{x\in \overline{\O}} \mu^*(x,a),~\mu^*_{\max}(a):=\sup\limits_{x\in \overline{\O}} \mu^*(x,a)
\end{align*}
satisfy 
\begin{align*} 
\int_0^a \mu_{\max}(k)dk <\infty,~\forall a\in (0,A),~\int_0^A \mu_{\min}(k)dk = \infty.
\end{align*}
\item  \label{cond:C2}$\beta^*=\beta^*(x,a) \in W^{1, \infty}(\overline{\O}\times(0,A))$, $\beta^*\geq 0$ and there exists $a_0,a_1 \in (0,A)$ and $a_0< a_1$ such that 
\begin{align*}
\beta^*_{\min}:=\inf_{x\in \overline{\O}}\beta^*(x,\cdot)>0 \text{ on }(a_0,a_1) \text{ and }\supp(\beta^*)\subset \overline{\O}\times (a_0,A).
\end{align*}
\end{enumerate}

Next, let us state the main results of this Appendix A. 
\begin{align}\label{eq:onevareigenage}
\begin{split}
\left\{\begin{array}{llll}
u_{a}-d\Delta_Nu+
\mu^*(x,a)u=\lambda^* u,&(x,a)\in\mathcal{O}_A,\\
\dfrac{\partial u}{\partial \textbf{n}}=0,&x\in \partial \O,~a\in (0,A),\\
u(x,0) = \displaystyle\int_0^{A} \beta^*(x,a) u(x,a)da,&x\in \O.
\end{array}\right.
\end{split}
\end{align}
We call an eigenvalue $\lambda^*$ is \textit{principal} if there exists an eigenfunction $\phi>0$ in $\overline{\O}\times (0,A)$.

Now, we state the main result
\begin{tha}\label{TheoA.1}
Assume \ref{cond:C1} and \ref{cond:C2} hold. Then, there exists a unique principal eigenvalue of \eqref{eq:onevareigenage}, denoted by $\lambda^*_{d,\mu^*}$. Moreover, it is simple and the only eigenvalue having a positive eigenfunction. The positive eigenfunctions can be taken bounded. Furthermore, for any other eigenvalue $\lambda^*$ of \eqref{eq:onevareigenage}, it holds that
\begin{align*}
\text{Re}(\lambda^*)> \lambda^*_{d,\mu^*}.
\end{align*}
Finally, the map $\mu^*\mapsto \lambda^*_{d,\mu^*}$ is increasing.
\end{tha}

\begin{proof}
 There are several steps to prove this result as follows

\textbf{Step 1:} Let us consider the result for the general case
\begin{align}\label{eq:paraNeu}
\begin{split}
\left\{\begin{array}{llll}
u_{a}-d\Delta_Nu+
\mu^*(x,a) u=f(x,a),&(x,a)\in\mathcal{O}_A,\\
\dfrac{\partial u}{\partial \textbf{n}}=0&x\in \partial \O,~a\in (0,A),\\
u(x,0) = \phi(x),&x\in \O.
\end{array}\right.
\end{split}
\end{align}
where $f\in L^2(\mathcal{O}_A)$, $f\geq 0$, $\phi \in L^2(\O)$, $\phi \geq 0$, $\mu^*$ satisfies \ref{cond:C1}. Then, there exists a unique weak solution $u$ of \eqref{eq:paraNeu}. Moreover, for any $0<A_0<A$, one has $u\in C([0,A_0];L^2(\O))$. Furthermore, we obtain the following comparison principles
\begin{enumerate}[label=(\roman*)]
\item If $f\geq 0$ and $\phi\geq 0$, then $u\geq 0$. If some of the inequalities are strict, we deduce that $u\gg 0$.

\item If $f_1\geq f_2\geq  0$ and $\phi_1\geq \phi_2\geq 0$ and $\mu^*_1\leq \mu^*_2$ in their respective domains, then $u_1\geq u_2$, where $u_i$, $i=1,2$ is the solution of \eqref{eq:paraNeu} with $f=f_i$, $\phi=\phi_i$ and $q=q_i$. In addition, one can have $u_1>u_2$ if one of the inequalities is strict.
\end{enumerate}

The proof of this step is similar to the one in \cite[Lemma 7]{delgado_nonlinear_2006} with some modifications as follows
\begin{enumerate}
\item We consider the weak convergence of $w_n$  in the functional space $L^2(0,A;H_0^{1}(\O))$ instead of $L^2(0,A;H^{1}(\O))$.
\item We use $\phi \in C_c^{\infty}(\O\times (0,A))$ to obtain the first equation and then use $\phi \in C^{\infty}_c(\overline{\O}\times (0,A))$ to recover the Neumann boundary condition.
\item In the uniqueness, with $u_1$ and $u_2$ are two distinct solutions and $w=u_1-u_2$. Then, $w=C$ for some constants $C$. Using the fact that $\displaystyle \int_0^A \mu^*(x,a)da=\infty$, one has $C=0$.
\end{enumerate}

\textbf{Step 2:} Before investigating the main theorem, some preliminary results are needed. For any $\phi\in L^2(\O)$, we consider $z(\phi)$ is the solution of the following equation
\begin{align}\label{eq:oneNeumannmu}
\begin{split}
\left\{\begin{array}{llll}
u_{a}-d\Delta_Nu+
\mu^*(x,a)u=0,&(x,a)\in\mathcal{O}_A,\\
\dfrac{\partial u}{\partial \textbf{n}}=0,&x\in \partial \O,~a\in (0,A),\\
u(x,0) = \phi(x),&x\in \O.
\end{array}\right.
\end{split}
\end{align}
and define $\mathcal{B}_{\lambda^*}:L^2(\O)\rightarrow L^2(\O)$ by
\begin{align*}
\mathcal{B}_{\lambda^*}(\phi)=\int_0^A\beta^*(x,a)z(\phi)(x,a)e^{\lambda^* a}da
\end{align*}
This operator is well-defined thanks to \text{Step 1} and $\beta^*\in L^{\infty}$. The following statements hold
\begin{enumerate}[label=(\roman*)]
\item The operator $\mathcal{B}_{\lambda^*}$ is compact and positive operator.
\item For $\phi\in L^2(\O)$, one has
\begin{align*}
\mathcal{A}_{\lambda^*}(\phi)\leq \mathcal{B}_{\lambda^*}(\phi)\leq \mathcal{C}_{\lambda^*}(\phi)
\end{align*}
where
\begin{align*}
&\mathcal{A}_{\lambda^*}(\phi):=\int_0^A\beta_{\min}^*(x,a)w(\phi)(x,a)da,\\
&\mathcal{C}_{\lambda^*}(\phi):=\int_0^A\beta_{\max}^*(x,a)y(\phi)(x,a)da,
\end{align*}
and $w(\phi)$ (resp. $y(\phi)$) is the solution of \eqref{eq:paraNeu} for the case $f=0$, $m=\mu^*_{\min}$ (resp. $m=\mu^*_{\max}$) with the initial condition $\phi$. 
\item $\mathcal{B}_{\lambda^*}$ is an irreducible operator.
\item If $\phi$ is a fixed point of $\mathcal{B}_{\lambda^*}$, then $\lambda^*$ is an eigenvalue of \eqref{eq:onevareigenage}.

\item If $(\lambda^*,\phi)$ is a pair of eigenvalue-eigenfunction of \eqref{eq:onevareigenage}. Then, $\phi(x):=u(x,0)$ is a fixed point of $\mathcal{B}_{\lambda^*}$.
\end{enumerate}

The proof of this step is the analogous result to the one in \cite[Lemma 9]{delgado_nonlinear_2006} with the help of strong positivity to obtain irreducible property. We highlight an important detail in the proof as follows: Due to the properties of the mapping $\phi\mapsto z(\phi)$ and the semi-group $T_{d\Delta_N}$ generated by $d\Delta_N$, it follows from \ref{cond:C2} that $\mathcal{B}_{\lambda^*}$ is compact. We end the proof.

\phantom{1}

Let us denote by $r(\mathcal{B}_{\lambda^*})$ the spectral radius of $\mathcal{B}_{\lambda^*}$. Then, similar to arguments in \cite[pp. 373]{delgado_nonlinear_2006}, one has $r(\mathcal{B}_{\lambda^*})>0$. By the Krein–Rutman theorem, $r(\mathcal{B}_{\lambda^*})$ is an algebraically simple eigenvalue with a quasi-interior eigenfunction and is, therefore, the only eigenvalue associated with a positive eigenfunction.

\textbf{Step 3:} $\lambda^*_2$ is a principal eigenvalue of \eqref{eq:onevareigenage} if and only if $r(\mathcal{B}_{\lambda^*_2})=1$. The proof is similar to that of \cite[Corollary 10]{delgado_nonlinear_2006}. Thus, we omit the details here.

Thanks to this result, an important result is established as the first step to prove the existence of the principal eigenvalue of \eqref{eq:onevareigenage}.

\textbf{Step 4:} Assume \ref{cond:B1} and \ref{cond:B2} hold. Then, 
\begin{align*}
r(\mathcal{D}_{\lambda_1+r_\mathcal{M}})=1,
\end{align*}
where 
\begin{align*}
\mathcal{D}_{\lambda^*}(\phi)=\int_0^A\gamma(a)u(\phi)(x,a)e^{\lambda^* a}da,
\end{align*}
and $u(\phi)$ is the solution of \eqref{eq:paraNeu} for the case $f=0$.

The technique for the proof can be found in \cite[Proposition 11]{delgado_nonlinear_2006}. Thus, we omit it.

\textbf{Step 5:} Lastly, the proof is a similar manner to the one in \cite[Theorem 8]{delgado_nonlinear_2006}. We highlight some important details as follows
\begin{enumerate}
\item It is easy to check that if $\lambda^*_1<\lambda^*_2$, then $\mathcal{B}_{\lambda^*_1}<\mathcal{B}_{\lambda^*_2}$. As the results, $r(\mathcal{B}_{\lambda^*_1})<r(\mathcal{B}_{\lambda^*_2})$ (see \cite[Theorem 3.2(v)]{amann_fixed_1976}).
\item Let $\phi_0\gg 0$ be the a principal eigenfunction associated with $r(\mathcal{B}_{\lambda^*})$, meaning
\begin{align*}
\mathcal{B}_{\lambda^*}(\phi_0)=r(\mathcal{B}_{\lambda^*})\phi_0. 
\end{align*}
and, for any $\varepsilon>0$, one has 
\begin{align*}
\mathcal{B}_{\lambda^*+\varepsilon}(\phi_0)\leq  e^{\varepsilon A}\mathcal{B}_{\lambda^*}(\phi_0)=r(\mathcal{B}_{\lambda^*})\phi_0 e^{\varepsilon A}
\end{align*}
This is essential to prove the continuous of $\lambda^* \mapsto r(\mathcal{B}_{\lambda^*})$.
\item Applying \text{Step 4} for $\mathcal{D}_{\lambda^*}=A_{\lambda^*}$ and $\mathcal{M}=\mu^*_{\max}$, one has 
\begin{align*}
r\left(\mathcal{A}_{\lambda_1+r_{\mu^*_{\max}}}\right)=1.
\end{align*}
Similarly, $r\left(\mathcal{C}_{\lambda^*_1+r_{\mu^*_{\min}}}\right)=1$. Thus, we can check that 
\begin{align*}
r\left(\mathcal{B}_{\lambda_1+r_{\mu^*_{\min}}}\right)\leq r\left(\mathcal{C}_{\lambda_1+r_{\mu^*_{\min}}}\right)=1 = r\left(\mathcal{A}_{\lambda_1+r_{\mu^*_{\max}}}\right)\leq r\left(\mathcal{B}_{\lambda_1+r_{\mu^*_{\max}}}\right) 
\end{align*}
where $\lambda_1=0$ is the principal eigenvalue of $-d\Delta_N$, $r_{m}$ is in Lemma \ref{Lemma A.1}. Then, there exists $\lambda^*_{d,\mu^*}\in\left(r_{\mu^*_{\min}},r_{\mu^*_{\max}}\right)$ such that $r\left(\mathcal{B}_{\lambda^*_{d,\mu^*}}\right)=1$. By applying the Krein–Rutman theorem,  $\lambda^*_{d,\mu^*}$ is algebraically simple and $Re(\lambda^*)>\lambda^*_{d,\mu^*}$ for any other eigenvalue $\lambda^*$ of \eqref{eq:onevareigenage}.

\item Assume $q_1<q_2$. Prove that $\lambda^*_{d,q_1}<\lambda^*_{d,q_2}$. Suppose otherwise, $\lambda^*_{d,q_1}\geq \lambda^*_{d,q_2}$. Applying \text{Step 1}(ii), one has $z_{q_1}(\phi)>z_{q_2}(\phi)$, where $ z_{q_i}(\phi)$, $i=1,2$ is the solution of \eqref{eq:oneNeumannmu} with initial condition $\phi$. Then,
\begin{align*}
\mathcal{B}_{\lambda^*_{d,q_1}}^{q_1}(\phi)>\mathcal{B}_{\lambda^*_{d,q_1}}^{q_2}(\phi)\geq\mathcal{B}_{\lambda^*_{d,q_2}}^{q_2}(\phi) 
\end{align*}
where $\mathcal{B}_{\lambda^*}^{q_i}$, $i=1,2$ corresponding with the case $q=q_i$, $i=1,2$ in \eqref{eq:oneNeumannmu}.
On the other hand, it is known that
\begin{align*}
r\left(\mathcal{B}_{\lambda^*_{d,q_2}}^{q_1}\right)=r\left(\mathcal{B}_{\lambda^*_{d,q_2}}^{q_2}\right)=1,
\end{align*}
which contradicts \cite[Theorem 3.2(v)]{amann_fixed_1976}. Hence, it follows that 
\begin{align*}
\lambda^*_1(q_1)<\lambda^*_1(q_2).
\end{align*}

\end{enumerate}
\end{proof}

\begin{remark} \label{lim0}
We can check that if $\phi$ is eigenfunction corresponding to $\lambda^*_1$, then $\lim\limits_{a\rightarrow A}\phi(x,a)=0$ uniformly with $x\in \overline{\O}$. 
\end{remark}

Let us discuss about the sign of the principal eigenvalue $\lambda^*_{d,\mu^*}$ in the case $\gamma=\gamma(a)$ first. Consider $\mu^*=0$, let $\phi$ be the positive eigenvalue of $\lambda^*_{d,0}$, then, one has
\begin{align*}
\begin{split}
\left\{\begin{array}{llll}
\phi_{a}=\lambda^*_{d,0} \phi,&~a\in (0,A),\\
\phi(0) = \displaystyle\int_0^{A} \gamma(a) \phi(a)da,
\end{array}\right.
\end{split}
\end{align*}
since $\phi$ independent with $x$. Simple calculation yields that
\begin{align*}
\phi(a)=e^{\lambda^*_{d,0}a}\text{ and } \int_0^A\gamma(a)e^{\lambda^*_{d,0}a}da =1
\end{align*}
We distinguish three cases: $\|\gamma\|_{L^1(0,A)}=1$, $\|\gamma\|_{L^1(0,A)}<1$ and $\|\gamma\|_{L^1(0,A)}>1$.

\textbf{Case 1}: $\|\gamma\|_{L^1(0,A)}=\displaystyle\int_0^A\gamma(a)da=1$. Then, one can prove that $\lambda^*_{d,0}=0$ since the function
\begin{align*}
\mathcal{F}:\lambda\rightarrow \int_0^A\gamma(a)e^{\lambda_0 a}da
\end{align*}
is increasing and $\mathcal{F}(0)=1$. Then, with $\mu>0$ (resp. $\mu<0$), we have that $\lambda_{d,\mu}^*>(\text{resp. }<)~\lambda^*_{d,0}=0$.

\textbf{Case 2}: $\|\gamma\|_{L^1(0,A)}=\displaystyle\int_0^A\gamma(a)da<1$. One can show there exists $\lambda_0>0$ such that
\begin{align*}
\int_0^A\gamma(a)e^{\lambda_0 a}da=1
\end{align*}
Now, let us consider
\begin{align*}
\begin{split}
\left\{\begin{array}{llll}
\phi_{a}-\lambda_0\phi=\lambda^*_{d,\lambda_0} \phi,&~a\in (0,A),\\
\phi(0) = \displaystyle\int_0^{A} \gamma(a) \phi(a)da,
\end{array}\right.
\end{split}
\end{align*}
leading to 
\begin{align*}
\int_0^A\gamma(a)e^{\left(\lambda^*_{d,\lambda_0}+\lambda_0\right) a}da=1
\end{align*}
Thus, $\lambda^*_{d,\lambda_0}=0$. With $\mu>\lambda_0$ (resp. $\mu<\lambda_0$), we have that $\lambda_{d,\mu}^*>(\text{resp. }<)~\lambda^*_{d,\lambda_0}=0$.

\textbf{Case 3}: $\|\gamma\|_{L^1(0,A)}=\displaystyle\int_0^A\gamma(a)da<1$. Similar to \textbf{Case 2}.

The general case $\beta=\beta(x,a)$ is similar but requires more advanced techniques, which we omit here. Consequently, the sign of the principal eigenvalue can be positive, negative, or zero, depending on $\mu^*$.


Next, we state the sub-supersolution method for the age-structured model. Let $g:\overline{\mathcal{O}_A}\times \R\rightarrow \R$ be a measurable function.
\begin{defa}
We say that a function $u\in  L^2(0,A;H^1(\O))$, $u_a+qu\in L^2(0,A;H^{-1}(\O))$ is a solution if $g(x,a,u)\in L^2(\mathcal{O}_A)$ and
\begin{align}\label{eq:solution}
\left\{\begin{array}{llll}
\displaystyle\int_0^A\left<u_a+qu,v\right>da+\displaystyle\int_{\mathcal{\mathcal{O}_A}}\nabla u\cdot\nabla v dxda= \int_{\mathcal{O}_A}g(x,a,u)vdxda,\\
u(x,0) = \displaystyle\int_0^A\beta^*(x,a)u(x,a)dxda,&x\in \O.
\end{array}\right.
\end{align}
for any $v\in L^2(0,A;H^1(\O))$.
\end{defa}

\begin{defa}
We say that a function $\overline{u}\in  L^2(0,A;H^1(\O))$ is a super-solution  if $\overline{u}_a+q\overline{u}\in L^2(0,A;H^{-1}(\O))$, $g(x,a,\overline{u})\in L^2(\mathcal{O}_A)$ and
\begin{align}\label{eq:solution}
\left\{\begin{array}{llll}
\displaystyle\int_0^A\left<\overline{u}_a+q\overline{u},v\right>da+\displaystyle\int_{\mathcal{\mathcal{O}_A}}\nabla \overline{u}\cdot\nabla v dxda\geq \int_{\mathcal{O}_A}g(x,a,\overline{u})vdxda,\\
\overline{u}(x,0) \geq  \displaystyle\int_0^A\beta^*(x,a)\overline{u}(x,a)dxda,&x\in \O.
\end{array}\right.
\end{align}
for any $v\in L^2(0,A;H^1(\O)),~v\geq 0$.
\end{defa}
We also define the sub-solution by interchanging the inequalities.

\begin{tha}
Assume \ref{cond:C1} and \ref{cond:C2} hold and for any $M>0$, there exists $L>0$ such that
\begin{align*}
|g(x,a,s_1)-g(x,a,s_2)|\leq L|s_1-s_2|,~\forall (x,a)\in \overline{\mathcal{O}_A},~s_1,s_2\in (-M,M). 
\end{align*}
Suppose further there exists a pair of sub- and super- solution $(\underline{u},\overline{u})$, $0\leq \underline{u}\leq \overline{u}$ and $\overline{u}\in L^{\infty}(\mathcal{O}_A)$. Then, \eqref{eq:solution} admits a solution.

\end{tha}
\begin{proof}
We can adapt the proof in \cite[Section 3]{delgado_nonlinear_2006} to the case of the homogeneous Neumann boundary condition, with some modifications similar to those in Theorem A.\ref{TheoA.1}-Step 1. Thus, we omit the details.
\end{proof}

Next, it is necessary to prove the existence of some age-structure models. Consider
\begin{align}\label{eq:KPPage}
\begin{split}
\left\{\begin{array}{llll}
u_{a}-d\Delta_Nu+
\mu(x,a)u=u[k(x,a)-n_0(x,a)u],&(x,a)\in\mathcal{O}_A,\\
\dfrac{\partial u}{\partial \textbf{n}}=0,&x\in \partial \O,~a\in (0,A),\\
u(x,0) = \displaystyle\int_0^{A} \beta(x,a) u(x,a)da,&x\in \O.
\end{array}\right.
\end{split}
\end{align}
where $\mu$ and $\beta$ are the one in \ref{cond:cond3} - \ref{cond:cond4} and \ref{cond:cond5}, $k,~n_0\in C^2(\overline{\mathcal{O}_A})$ and $n_0>0$. Inspired by the work of \cite[Proposition 5.3]{kang_effects_2022}, we consider the following eigenvalue-eigenfunction problem 
\begin{align*}
\begin{split}
\left\{\begin{array}{llll}
u_{a}-d\Delta_Nu+
\mu(x,a)u-k(x,a)u=\lambda^* u,&(x,a)\in\mathcal{O}_A,\\
\dfrac{\partial u}{\partial \textbf{n}}=0,&x\in \partial \O,~a\in (0,A),\\
u(x,0) = \displaystyle\int_0^{A} \beta(x,a) u(x,a)da,&x\in \O.
\end{array}\right.
\end{split}
\end{align*}
This problem admits a unique principal eigenvalue  $\lambda^*_{d,-k+\mu}$. 

We begin with the comparison principle for sub- ans super-solutions as follows

\begin{lma}\label{comparison}
Suppose \ref{cond:cond3} to \ref{cond:cond5} hold. Let $v\in C(\overline{\mathcal{O}_A})$ be a super-solution of \eqref{eq:KPPage} and  $u\in C(\overline{\mathcal{O}_A})$ be a sub-solution of \eqref{eq:KPPage}. Suppose further, for any $0<A_3<A$, that
\begin{align*}
u,v\in C^{2,1}(\mathcal{O}_{A_3})
\end{align*}
and $u>0,~v>0$ on $\overline{\mathcal{O}_{A_3}}$. Then, $u\leq v$ on $\overline{\mathcal{O}_A}$. In addition, if either $v$ is a super-solution but not a solution or $u$ is a sub-solution but not a solution, then $u<v$ on $\overline{\O}\times [0,A)$.
\end{lma}

\begin{proof}
By arguments similar to those in Remark \ref{lim0}, we can show that $v(x,A)=0$.

The strategy is to prove the inequality on $\overline{\mathcal{O}_{A_3}}$ for some $0<A_3<A$ to avoid the blow-up property of $\mu$. Then, we pass to the limit to obtain the result. Let us recall $0<A_0<A_2<A$ in \ref{cond:cond5} such that $\supp({\beta})\subset \overline{\O}\times (A_0,A_2]$.

Consider $A_2<A_3<A$. It is known that $u,~v\in C\left(\overline{\mathcal{O}_{A_3}}\right)$. Thus, there exists $\alpha>0$ such that $\alpha u\leq v$ on $\overline{\mathcal{O}_{A_3}}$. Define
\begin{align*}
\alpha^*=\sup\{\alpha>0:\alpha u\leq v\text{ on }\overline{\mathcal{O}_{A_3}}\}.
\end{align*}
Put $w(x,a)=v(x,a)-\alpha^*u(x,a)\geq 0$. It is easy to show that there exists $(x_0,a_0)\in \overline{\mathcal{O}_{A_3}}$ such that $w(x_0,a_0)=0$. If $\alpha^*\geq 1$, then we are done. Suppose otherwise $\alpha^*<1$. 

If  $a_0=0$, by \ref{cond:cond5}, 
\begin{align*}
0=w(x_0,0)=\int_0^A\beta(x_0,a)w(x_0,a)da=
\int_0^{A_2}\beta(x_0,a)w(x_0,a)da.
\end{align*}
One can find $a_1 \in (0,A_2)\subset (0,A_3)$ such that $w(x_0,a_1)=0$. Thus, without loss generality, we may assume $a_0>0$. It is easy to check that
\begin{align*}
w_a-d\Delta_N w + \mu(x,a)w&\geq  v(k-n_0v)-\alpha^* u(k-n_0u)\\
&> v(k-n_0v)-\alpha^* u(k-\alpha^*n_0u)\\
&=(k-n_0v-\alpha^*n_0u)w,
\end{align*}
or
\begin{align*}
w_a-d\Delta_N w + (\mu(x,a)-k+n_0v+\alpha^*n_0u)w>0.
\end{align*}
There exists $C_0>0$ large enough such that
\begin{align*}
w_a-d\Delta_N w + C_0w>0.
\end{align*}
Since $v$ is super-solution and $u$ is the sub-solution, it is easy to check that $\dfrac{\partial w}{\partial \textbf{n}}\geq 0$ on $\partial \O$ and $(x_0,a_0)$ is a minimum point of $w$. By Hopf's lemma and the strong maximum principle for the parabolic equation, we must have $w\equiv 0$ on $\overline{\O}\times [0,a_0]$. Contradiction. Hence, $\alpha^*\geq 1$ and $u\leq v$ on $\overline{\mathcal{O}_{A_3}}$. As the result, one gets that $u\leq v$ on $\overline{\mathcal{O}_A}$. 

To prove the last statement, we assume, for the sake of contradiction, there exists $A_2<A_3<A$, $(x_2,a_2)$ and $0<a_2<A_3$ such that $w(x_2,a_2):=u(x_2,a_2)-v(x_2,a_2)=0$. By same arguments, one has $u=v$. Contradiction.

This completes the proof.

\end{proof}

%
%

Let us state the existence result as follows
\begin{tha}\label{Lemma A.7}
Assume \ref{cond:cond3} to \ref{cond:cond5} hold. Suppose further $\lambda^*_{d,-k+\mu}<0$. Then, \eqref{eq:KPPage} admits a unique solution.
\end{tha}

\begin{proof}
The proof divides into two steps

\textbf{Step 1: Existence}. Let $\phi$ be the bounded eigenfunction associated with the eigenvalue $\lambda^*_{d,-k+\mu}$, $\phi>0$ on $\overline{\O}\times [0,A)$ and $\phi(x,A)=0$ for any $x\in \overline{\O}$. We choose $\epsilon>0$ small enough so that 
\begin{align*}
\epsilon n_0(x,a)\phi \leq \epsilon \max n_0 \max \phi \leq -\lambda^*_{d,-k+\mu}
\end{align*}
Define $\underline{u}:=\epsilon \phi$ for some $\epsilon>0$, then
\begin{align*}
\underline{u}_{a}-d\Delta_N\underline{u}+
\mu(x,a)\underline{u}-k(x,a)\underline{u}=\lambda^*_{d,-k+\mu} \underline{u}\leq -n_0(x,a)\underline{u}^2.
\end{align*} 
which is a sub-solution.

Next, for any $A_2<A_3<A$, recall $\pi=\pi(a)=e^{-\displaystyle\int_0^a \mu_{\min}(k)dk}$, we define $\overline{u}:=M\pi(a)$ for some constant $M>0$. Then, 
\begin{align*}
\begin{array}{lllll}
\overline{u}_a-d\Delta_N \overline{u} +\mu\overline{u}-\overline{u}k+n_0\overline{u}^2&= -\mu_{\min}\overline{u}+\mu\overline{u}-\overline{u}k+n_0\overline{u}^2\\ &\geq  M^2n_0\pi^2(a) - kM\pi(a)\geq 0
\end{array}
\text{ on }\overline{\mathcal{O}_{A_3}}
\end{align*}
provided that $M>0$ large enough. The age-structure can be proved as follows
\begin{align*}
\int_0^{A_3}\beta(x,a)\overline{u}(x,a)da = M \int_0^{A_3}\beta(x,a)\pi(a)da \leq M\int_0^{A_3}\beta_{\max}(a)\pi(a)da\leq M
\end{align*}
since \ref{cond:cond5}. Hence, $\overline{u}$ is a super-solution. By applying Lemma A.\ref{comparison}, we get $\underline{u}\leq \overline{u}$. Then, thanks to sub-super solution method, there exists $u_{A_3}$ is a solution of the equation \eqref{eq:KPPage} on $\overline{\mathcal{O}_{A_3}}$. 

\textbf{Step 2: Uniqueness}. Let $A_3$ be a positive number satisfying $A_2<A_3<A$, $u$ and $v$ be two nonnegative bounded solution of \eqref{eq:KPPage}. Using the bootstrap arguments in \cite[Proposition 1]{ducrot_travelling_2007} and \cite[pp 22/480]{ducrot_travelling_2009}, one has
\begin{align*}
u,~v\in C(\overline{\mathcal{O}_{A_3}}) \cap C^{2,1}(\mathcal{O}_{A_3})
\end{align*}
By strong parabolic maximum principle, one gets $u>0$, $v>0$ on $\overline{\O} \times [0,A_3]$. Since $u,~v$ are solutions, one can apply the comparison principle in Lemma A.\ref{comparison} so that $u\leq v$ and $v\leq u$ on $\overline{\mathcal{O}_{A_3}}$.

\textbf{Step 3: Extension}. By the uniqueness, one can extend the solution $u$ on $\overline{\O}\times [0,A_{\max})$, for some $A_2<A_{\max}\leq A$. Furthermore, we have
\begin{align*}
u_a-d\Delta_N u \leq Cu -\mu_{\min}(a)u \text{ on }\overline{\O}\times (0,A_{max})
\end{align*}
for some $C\geq k(x,a)$ on $\overline{\mathcal{O}_A}$, $C$ is independent of $A_3$ and $u\geq 0$ on $[0,A_{\max})$. Thus, by parabolic comparison principle, one gets
\begin{align*}
0<u\leq e^{Ca-\displaystyle\int_0^a\mu_{\min}(s)ds} T_{d\Delta_N}(a)u_0 \text{ on }\overline{\O}\times (0,A_{\max})
\end{align*}
where $u_0(x)=u(x,0)$, $T_{d\Delta_N}$ is the $C_0$ semi-group generated by $-d\Delta_N$. Note that  to calculate $u_0$, one only needs the data of $u$ on $\overline{\O}\times [0,A_2]\subset\overline{\O}\times [0,A_{max})$ since $\supp(\beta)\subset \overline{\O}\times (A_0,A_2]$. As the results, the right hand side is independent of the $A_{max}$ and 
\begin{align*}
\left\|e^{Ca-\displaystyle\int_0^a\mu_{\min}(s)ds} T_{d\Delta_N}(a)u_0\right\|_{L^{\infty}(\mathcal{O}_A)}\leq e^{CA}||u_0||_{L^{\infty}(\O)}<\infty
\end{align*}
Consequently, one can extend the solution $u$ on $\overline{\O}\times [0,A)$ or $A_{\max}=A$. Furthermore, by passing to the limit $a\rightarrow A$, we obtain that $u(x,A)=0$ uniformly in $x\in \overline{\O}$. Thus, one can extend $u\in C(\overline{\mathcal{O}_A})$ so that $u>0$ on $\overline{\O} \times [0,A)$ and $u(x,A)=0$ for any $x\in \overline{\O}$. 

This completes the proof.

\end{proof}

\begin{remark} \label{lim1}
One can check that $u(x,A)=0$ and $\lim\limits_{a\rightarrow A} u(x,a)=0$ uniformly with $x\in \overline{\O}$ thanks to regular parabolic comparison principle.
\end{remark}

Lastly, consider
\begin{align}\label{eq:KPPage1}
\begin{split}
\left\{\begin{array}{llll}
u_{a}-d\Delta_Nu+
\mu(x,a)u=k_1(x,a)\dfrac{h(x,a)u}{h(x,a)+u}-k_2(x,a)u,&(x,a)\in\mathcal{O}_A,\\
\dfrac{\partial u}{\partial \textbf{n}}=0,&\text{on } \partial \O \times (0,A),\\
u(x,0) = \displaystyle\int_0^{A} \beta(x,a) u(x,a)da,&x\in \O.
\end{array}\right.
\end{split}
\end{align}
where $\mu$ and $\beta$ are the one in \ref{cond:cond3} - \ref{cond:cond4} and \ref{cond:cond5}, $k_1,~k_2,~h\in C^2(\overline{\mathcal{O}_A})$, $k_1\geq 0$ is non-trivial, $k_1>0$ on $\overline{\mathcal{O}_A}$ and $h>0$ on $\overline{\O}\times [0,A)$ and $h(x,A)=0$ for any $x\in \overline{\O}$. For convenience, we define
\begin{align*}
g(x,a,u)=\left\{\begin{array}{llll}
k_1(x,a)\dfrac{h(x,a)u}{h(x,a)+u}-k_2(x,a)u,&h>0 \text{ and }u>0,\\
0&h=0\text{ or }u=0.
\end{array}\right.
\end{align*}
It is easy to prove that $g:\overline{\mathcal{O}_A}\times [0,\infty)$.

\begin{tha} \label{Lemma:A.10}
Assume \ref{cond:cond3} to \ref{cond:cond5} hold. Suppose further $\lambda^*_{d,-k_1+k_2+\mu}<0$. Then, \eqref{eq:KPPage1} admits an unique solution $u$ on $\overline{\mathcal{O}_A}$ such that $u>0$ on $\overline{\O}\times [0,A)$ and $u(x,A)=0$ for any $x\in \overline{\O}$.
\end{tha}

\begin{proof}

We divide the proof into three steps:

\textbf{Step 1: Existence}. Let $A_2<A_3<A$ and $\phi$ be the bounded positive eigenfunction associated with the eigenvalue $\lambda^*_{-k+\mu}$. It is known that $\phi,~h>0$ on $\overline{\mathcal{O}_{A_3}}$. Define $\underline{u}:=\epsilon \phi$ for some $\epsilon>0$, then
\begin{align*}
\underline{u}_{a}-d\Delta_N\underline{u}+
\mu(x,a)\underline{u}=\lambda^*_{d,-k_1+k_2+\mu} \underline{u} +k_1(x,a)\underline{u}-k_2(x,a)\underline{u}.
\end{align*} 
Consider 
\begin{align*}
\left|k_1(x,a)\dfrac{h(x,a)}{h(x,a)+\underline{u}}-k_1(x,a)\right|&=\left|k_1(x,a)\dfrac{\epsilon \phi}{h(x,a)+\epsilon \phi}\right|\\
&\leq k_1(x,a)\dfrac{\epsilon \phi}{h(x,a)}\leq K \epsilon.
\end{align*}
for some $K>0$. Choose $\epsilon>0$ small enough so that $K \epsilon<-\lambda^*_{d,-k_1+k_2+\mu}$. Then,
\begin{align*}
k_1(x,a)-k_1(x,a)\dfrac{h(x,a)}{h(x,a)+\underline{u}}<-\lambda^*_{d,-k_1+k_2+\mu}
\end{align*}
As the results,
\begin{align*}
\underline{u}_{a}-d\Delta_N\underline{u}+
\mu(x,a)\underline{u}\leq k_1(x,a)\dfrac{h(x,a)}{h(x,a)+\underline{u}}\underline{u}-k_2(x,a)\underline{u},
\end{align*}
which is a sub-solution. Let $\overline{u}:=M\pi(a)$ for some constant $M>0$. Then, 
\begin{align*}
\overline{u}_a-d\Delta_N \overline{u} +\mu\overline{u}-k_1\dfrac{h\overline{u}}{h+\overline{u}}+k_2\overline{u}\geq -k_1 h +k_2\overline{u}= Mk_2\pi(a)-k_1 h \geq 0,
\end{align*}
provided that $M>0$ large enough and $\pi(a)>0$ on $[0,A_3]$. The age-structure can be shown as follows
\begin{align*}
\int_0^A\beta(x,a)\overline{u}(x,a)da = M \int_0^A\beta(x,a)\pi(a)da \leq M \int_0^A\beta_{\max}(a)\pi(a)da\leq M
\end{align*}
thanks to \ref{cond:cond5}. Thus, $\underline{u}$ is a super-solution. We can choose $M>0$ large enough so that $\underline{u}\leq \overline{u}$ on $\overline{\mathcal{O}_{A_3}}$. By applying sub-super solution method on $\overline{\mathcal{O}_{A_3}}$, there exists $u_{A_3}$ is a solution of the equation \eqref{eq:KPPage1} on $\overline{\O}\times[0,A_3]$. This is possible since 
\begin{align*}
\int_0^A\beta(x,a)u(x,a)da= \int_0^{A_2}\beta(x,a)u(x,a)da= \int_0^{A_3}\beta(x,a)u(x,a)da.
\end{align*}

\textbf{Step 2: Comparison principle and uniqueness}. The comparison principle on $\overline{\mathcal{O}_{A_3}}$ can be obtain by a similar manner to the one in Lemma  A.\ref{comparison}. We highlight an important detail: Consider $\alpha^*<1$ and $w=v-\alpha^*u$ as in Lemma  A.\ref{comparison}, one has
\begin{align*}
w_a - d\Delta_Nw +\mu(x,a)w+ k_2(x,a)w&\geq k_1h\dfrac{hw+(1-\alpha^*)uv}{(h+u)(h+v)}\\&>k_1\dfrac{h^2w}{(h+u)(h+v)}\geq 0.
\end{align*}
The conclusion $u<v$ on $\overline{\mathcal{O}_{A_3}}$ still hold if either $v$ is a super-solution but not a solution or $u$ is a sub-solution but not a solution. Hence, the uniqueness can be achieved in this case.

\textbf{Step 3: Extension}. By the uniqueness, one can extend the solution $u$ on $\overline{\O}\times [0,A_{\max})$, for some $A_2<A_{\max}\leq A$. Furthermore, we have
\begin{align*}
u_a-d\Delta_N u \leq Cu -\mu_{\min}(a)u \text{ on }\overline{\O}\times (0,A_{\max})
\end{align*}
for some $C>k_1-k_2$ on $\overline{\mathcal{O}_A}$, $C$ is independent of $A_{\max}$. Thus, by parabolic comparison principle, one gets
\begin{align*}
0<u\leq e^{Ca-\displaystyle\int_0^a\mu_{\min}(s)ds} T_{d\Delta_N}(a)u_0 \text{ on }\overline{\O}\times (0,A_{\max})
\end{align*}
where $u_0(x)=u(x,0)$, $T_{d\Delta_N}$ is the $C_0$ semi-group generated by $-d\Delta_N$. The rest proceeds in a similar manner to Step 3 of Theorem A.\ref{Lemma A.7}.
\end{proof}



\begin{center}
\textbf{Data availability statements and conflict of interest}
\end{center}
Data sharing not applicable to this article as no datasets were generated or analyzed
during the current study. There is no conflict of interest to this work.


\addcontentsline{toc}{section}{\bf References}
\begin{thebibliography}{99}


\bibitem{allen_asymptotic_2008}
L. J. S. Allen, B. M. Bolker, Y. Lou, and A. L. Nevai, {\it Asymptotic profiles of the steady
states for an sis epidemic reaction-diffusion model,} Discrete Contin. Dyn. Syst. {\bf 21}
(2008), no. 1, 1–20.

\bibitem{almeida_asymptotic_2022}
L. Almeida, B. Perthame, and X. Ruan, {\it An asymptotic preserving scheme for capturing
concentrations in age-structured models arising in adaptive dynamics,} J. Comput.
Phys. {\bf 464} (2022), 111335.


\bibitem{amann_fixed_1976}
H. Amann, {\it Fixed point equations and nonlinear eigenvalue problems in ordered Banach spaces}, SIAM Rev. {\bf 18} (1976), 620–709.


\bibitem{ambrosio_functions_2000}
L. Ambrosio, N. Fusco, and D. Pallara, {\it Functions of bounded variation and free discontinuity problems}, Oxford University Press, 2000.


\bibitem{anita_2000}
S. Anita, {\it Analysis and control of age-dependent population dynamics}, Kluwer Academic,
Dordrecht, 2000.



\bibitem{arendt_heat_2005}

W. Arendt, {\it Heat kernels}, ISEM 2005/2006.

\bibitem{bekkal_brikci_age-and-cyclin-structured_2008}
F. Bekkal Brikci, J. Clairambault, B. Ribba, and B. Perthame, {\it An age-and-cyclin-structured cell population model for healthy and tumoral tissues}, J. Math. Biol. {\bf 57} (2008), no. 1, 91–110.


\bibitem{brezis_functional_2011}
H. Brezis, {\it Functional Analysis, Sobolev Spaces and Partial Differential Equations}, Springer New York, 2011.



\bibitem{chekroun_global_2020}
A. Chekroun and T. Kuniya, {\it Global threshold dynamics of an infection age-structured SIR epidemic model with diffusion under the Dirichlet boundary condition}, J.  Differential Equations. {\bf 269} (2020), no. 8, 117–148.



\bibitem{conti_asymptotic_2005}
M. Conti, S. Terracini, and G. Verzini, {\it Asymptotic estimates for the spatial segregation of competitive systems}, Adv. Math. {\bf 195} (2005), no. 2, 524–560.



\bibitem{coville_simple_2010}
J. Coville, {\it On a simple criterion for the existence of a principal eigenfunction of some nonlocal operators}, J.  Differential Equations {\bf 249} (2010), no. 11, 2921–2953.


\bibitem{delgado_nonlinear_2008}
M. Delgado, M. Molina-Becerra, and A. Suárez, {\it Nonlinear age-dependent diffusive
equations: A bifurcation approach}, J.  Differential Equations {\bf 244} (2008), no. 9, 2133–2155.



\bibitem{delgado_nonlinear_2006}
M. Delgado, M. Molina-Becerra, and A. Suárez, {\it A nonlinear age-dependent model with spatial diffusion}, J. Math. Anal. Appl. {\bf 313} (2006), no. 1, 366–380.



\bibitem{deng_pulsating_2024}
L. Deng and A. Ducrot, {\it Pulsating waves in a multidimensional reaction-diffusion
system of epidemic type}, J. Eur. Math. Soc. (2024), doi 10.4171/JEMS/1511.


\bibitem{ducrot_travelling_2007}
A. Ducrot, {\it Travelling wave solutions for a scalar age-structured equation}, Discrete
Contin. Dyn. Syst. Ser. B {\bf 7} (2007), no. 2, 251–273.



\bibitem{ducrot_differential_2022}
A. Ducrot, Q. Griette, Z. Liu, and P. Magal, {\it Differential equations and population
dynamics I: Introductory approaches}, Springer International Publishing, 2022.

\bibitem{ducrot_age_2024}
A. Ducrot, H. Kang, and S. Ruan, {\it Age-structured models with nonlocal diffusion of Dirichlet type. I: Principal spectral theory and limiting properties}, J. Anal. Math. {\bf 155} (2025), 327–390.


\bibitem{ducrot_age-structured_2024}
A. Ducrot, H. Kang, and S. Ruan, {\it Age-structured models with nonlocal diffusion of
Dirichlet type. II: Global dynamics,} Isr. J. Math.  {\bf 266}  (2025), 219–257.





\bibitem{ducrot_travelling_nodate}
A. Ducrot and P. Magal, {\it Travelling wave solutions for an infection-age structured
epidemic model with external supplies}, Nonlinearity {\bf 24} (2018), 2891–2911.



\bibitem{ducrot_travelling_2009}
A. Ducrot and P. Magal, {\it Travelling wave solutions for an infection-age structured
model with diffusion}, Proc. Roy. Soc. Edinburgh, Section: A {\bf 139} (2009), no. 3, 459–482.


\bibitem{ducrot_travelling_2010}
A. Ducrot, P. Magal, and S. Ruan, {\it Travelling wave solutions in multigroup agestructured
epidemic models}, Arch. Rational Mech. Anal. {\bf 195} (2010), no. 1, 311–331.


\bibitem{ducrot_integrated_2021}
A. Ducrot, P. Magal, and A. Thorel, {\it An integrated semigroup approach for age structured
equations with diffusion and non-homogeneous boundary conditions}, Nonlinear
Differ. Equ. Appl. {\bf 28} (2021), no. 5, 49.


\bibitem{evans_partial_2010}
L. C. Evans, {\it Partial differential equations}, American Mathematical Soc., 2010.


\bibitem{ferriere_stochastic_2009}
R. Ferrière and V. C. Tran, {\it Stochastic and deterministic models for age-structured
populations with genetically variable traits}, ESAIM: Proc. {\bf 27} (2009), 289–310.




\bibitem{fonseca_degree_1995}
I. Fonseca and W Gangbo, {\it Degree theory in analysis and applications}, Oxford University, PressOxford, 1995.


\bibitem{ge_sis_2015}
J. Ge, K. I. Kim, Z. Lin, and H. Zhu, {\it A SIS reaction–diffusion–advection model in a
low-risk and high-risk domain}, J.  Differential Equations {\bf 259} (2015), no. 10, 5486–5509.

\bibitem{guo_on_1994}
B. Z. Guo and W. L. Chan, {\it On the semigroup for age dependent dynamics with spatial
diffusion}, J. Math. Anal. Appl. {\bf 184} (1994), 190–199.

\bibitem{gurtin_non-linear_1974}
M. E. Gurtin and R. C. MacCamy, {\it Non-linear age-dependent population dynamics},
Arch. Rational Mech. Anal. {\bf 54} (1974), no. 3, 281–300.


\bibitem{kang_effects_2022}
H. Kang, {\it The effects of diffusion on the principal eigenvalue for age-structured models
with random diffusion}, Proc. R. Soc. Edinb. A: Math. {\bf 152} (2022), no. 1, 258–280.


\bibitem{kang_principal_2022}
H. Kang and S. Ruan, {\it Principal spectral theory and asynchronous exponential growth
for age-structured models with nonlocal diffusion of Neumann type}, Math. Ann. {\bf 384}
(2022), no. 1, 1–49.


\bibitem{kang_principal_2023}
H. Kang and S. Ruan, {\it Principal spectral theory in multigroup age-structured models
with nonlocal diffusion}, Calc. Var. Partial Differential Equations {\bf 62} (2023), no. 7, 197.


\bibitem{kreyszig_introductory_1978}
E. Kreyszig, {\it Introductory functional analysis with applications}, John Wiley and Sons,
Hoboken, 1978.

\bibitem{langlais_large_1988}
M. Langlais, {\it Large time behavior in a nonlinear age-dependent population dynamics
problem with spatial diffusion}, J. Math. Biol. {\bf 26} (1988), no. 3, 319–346.


\bibitem{li_analysis_2017}
B. Li, H. Li, and Y. Tong, {\it Analysis on a diffusive SIS epidemic model with logistic
source}, Z. Angew. Math. Phys. {\bf 68} (2017), no. 4, 96.



\bibitem{li_varying_2017}
H. Li, R. Peng, and F. B.Wang, {\it Varying total population enhances disease persistence:
Qualitative analysis on a diffusive SIS epidemic model}, J.  Differential Equations {\bf 262} (2017), 885–913.

\bibitem{li_bifurcation_2024}
Z. Li and S. Terracini, {\it Bifurcation for the Lotka-Volterra competition model},
arXiv.2404.13410 (2024).



\bibitem{lunardi_analytic_1995}
A. Lunardi, {\it Analytic semigroups and optimal regularity in parabolic problems},
Springer Basel, 1995.



\bibitem{magal_theory_2018}
P. Magal and S. Ruan, {\it Theory and applications of abstract semilinear Cauchy problems},
Springer International Publishing, 2018.


\bibitem{magal_monotone_2019}
P. Magal, O. Seydi, and F. B. Wang, {\it Monotone abstract non-densely defined Cauchy problems applied to age structured population dynamic models}, J. Math. Anal. Appl. {\bf 479} (2019), no. 1, 450–481.


\bibitem{zbMATH03937879}
R. Nagel (ed.), {\it One-parameter semigroups of positive operators}, Lect. Notes Math., Springer, Cham, 1986.



\bibitem{nirenberg_nonlinear_2001}
L. Nirenberg, {\it Topic in nonlinear functional analysis}, American Mathematical Society,
Providence, 2001.


\bibitem{nordmann_dynamics_2021}
S. Nordmann and B. Perthame, {\it Dynamics of concentration in a population structured
by age and a phenotypic trait with mutations. convergence of the corrector}, J.  Differential Equations {\bf 290} (2021), 223–261.


\bibitem{nordmann_dynamics_2018}
S. Nordmann, B. Perthame, and C. Taing, {\it Dynamics of concentration in a population
model structured by age and a phenotypical trait}, Acta Appl. Math. {\bf 155} (2018), no. 1,
197–225.


\bibitem{noris_uniform_2010}
B. Noris, H. Tavares, S. Terracini, and G. Verzini, {\it Uniform Hölder bounds for nonlinear
Schrödinger systems with strong competition}, Comm. Pure Appl. Math. {\bf 63}
(2010), no. 3, 267–302.


\bibitem{park_optimal_1998}
E. J. Park, M. Iannelli, M. Y. Kim, and S. Anita, {\it Optimal harvesting for periodic
age-dependent population dynamics}, SIAM J. Appl. Math. {\bf 58} (1998), no. 5, 1648–
1666.

\bibitem{pazy_semigroups_2012}
A. Pazy, {\it Semigroups of linear operators and applications to partial differential equations},
Springer, 1983.


\bibitem{peng_reactiondiffusion_2012}
R. Peng and X. Q. Zhao, {\it A reaction–diffusion SIS epidemic model in a time-periodic
environment}, Nonlinearity {\bf 25} (2012), no. 5, 1451–1471.


\bibitem{roget_long-time_2019}
T. Roget, {\it On the long-time behaviour of age and trait structured population dynamics},
Discrete Contin. Dyn. Syst. Ser. B {\bf 24} (2019), no. 6, 2551–2576.


\bibitem{metzler_random_2000}
R. S. Metzler and J. Klafter, {\it The random walk’s guide to anomalous diffusion: A
fractional dynamics approach}, Phys. Rep. {\bf 339} (2024), no. 1, 1–77.


\bibitem{soave_anisotropic_2023}
N. Soave and S. Terracini, {\it An anisotropic monotoncity formula, with applications to
some segregation problems}, J. Eur. Math. Soc. {\bf 25} (2023), no. 9, 3727–3765.



\bibitem{tavares_existence_2011}
H. Tavares, S. Terracini, G. Verzini, and T. Weth, {\it Existence and nonexistence of
entire solutions for non-cooperative cubic elliptic systems}, Comm. Partial Differential
Equations {\bf 36} (2011), no. 11, 1988–2010.



\bibitem{terracini_uniform_2016}
S. Terracini and A. Verzini G.and Zilio, {\it Uniform Hölder bounds for strongly competing
systems involving the square root of the Laplacian}, J. Eur. Math. Soc. {\bf 18} (2016),
no. 12, 2865–2924.


\bibitem{tian_traveling_2022a}
X. Tian and S. Guo, {\it Traveling wave solutions for nonlocal dispersal Fisher–KPP model with age structure}, Appl. Math. Lett. {\bf 123} (2022), 107593.


\bibitem{tian_traveling_2022b}
X. Tian and S. Guo, {\it Traveling waves of an epidemic model with general nonlinear
incidence rate and infection-age structure}, Z. Angew. Math. Phys. {\bf 73} (2022), no. 4, 167.


\bibitem{walker_age-dependent_2010}
C. Walker, {\it Age-dependent equations with non-linear diffusion}, Discrete Contin. Dyn.
Syst. Ser. A {\bf 26} (2010), no. 2, 691–712.


\bibitem{walker_coexistence_2010}
C. Walker, {\it Coexistence steady states in a predator–prey model}, Arch. Math. {\bf 95} (2010),
no. 1, 87–99.


\bibitem{walker_global_2010}
C. Walker, {\it Global bifurcation of positive equilibria in nonlinear population models}, J.  Differential Equations {\bf 248} (2010), no. 7, 1756–1776.


\bibitem{walker_nonlocal_2011}
C. Walker, {\it On nonlocal parabolic steady-state equations of cooperative or competing
systems}, Nonlinear Anal. Real World Appl. {\bf 12} (2011), no. 6, 3552–3571.


\bibitem{walker_positive_2011}
C. Walker, {\it On positive solutions of some system of reaction-diffusion equations with
nonlocal initial conditions}, J. Reine Angew. Math. {\bf 660} (2011).


\bibitem{walker_well-posedness_2023}
C. Walker, {\it Well-posedness and stability analysis of an epidemic model with infection
age and spatial diffusion}, J. Math. Biol. {\bf 87} (2023), no. 3, 52.


\bibitem{walker_principle_2022}
C. Walker and J. Zehetbauer, {\it The principle of linearized stability in age-structured
diffusive populations}, J.  Differential Equations {\bf 341} (2022), 620–656.


\bibitem{wu_theory_1996}
J. Wu, {\it Theory and applications of partial functional differential equations}, Springer
New York, 1996.


\bibitem{wu_asymptotic_2016}
Y. Wu and X. Zou, {\it Asymptotic profiles of steady states for a diffusive SIS epidemic
model with mass action infection mechanism}, J.  Differential Equations {\bf 261} (2016), no. 8, 4424–4447.


\bibitem{yang_asymptotical_2023}
J. Yang, M. Gong, and G. Q. Sun, Asymptotical profiles of an age-structured footand-
mouth disease with nonlocal diffusion on a spatially heterogeneous environment,
J.  Differential Equations {\bf 377} (2023), 71–112.


\bibitem{zhao_spatiotemporal_2023}
G. Zhao and S. Ruan, Spatiotemporal dynamics in epidemic models with Lévy flights:
A fractional diffusion approach, J. Math. Pures Appl. {\bf 173} (2023), no. 9, 243–277.


\end{thebibliography}
\end{document}